\documentclass{article}
\usepackage{blindtext}

\usepackage[english]{babel}

\usepackage{graphicx}
\usepackage{amssymb}
\usepackage{amsthm}
\usepackage{amsmath,amssymb}
\usepackage[all]{xy}
\usepackage{lmodern}
\usepackage{lipsum}
\usepackage{color}
\definecolor{blue}{rgb}{0.12,0.47,0.87}
\usepackage{xcolor}

\usepackage{geometry}
\title{Almost blenders and parablenders}
\author{S\'ebastien Biebler}

\newcommand\blfootnote[1]{%
  \begingroup
  \renewcommand\thefootnote{}\footnote{#1}%
  \addtocounter{footnote}{-1}%
  \endgroup
}

\begin{document}

\newtheorem{theorem}{Theorem}[subsection]
\newtheorem*{theoremlilou}{Theorem C}
\newtheorem*{theoremliloum}{Theorem B}
\newtheorem*{theoremliloumm}{Theorem D}
\newtheorem*{theoremcontinu}{Theorem (Przytycki)}
\newtheorem*{theoremliloumMm}{Theorem A}
\newtheorem*{coro}{Corollary C} 
\newtheorem*{theoremliloumf}{Main Theorem}
\newtheorem*{theoremcontinuf}{Theorem (Przytycki)}
\newtheorem*{corofffff}{Theorem } 
\newtheorem*{corof}{Main Corollary} 
\newtheorem*{coroC}{Conjecture} 
\newtheorem*{coroC15}{Conjecture A (Berger \cite{berc})} 
\newtheorem*{coroC16}{Conjecture B (\cite{rio} Conj. 3.1)} 
\newtheorem{Lemma}[theorem]{Lemma}
\newtheorem{Definition}[theorem]{Definition}
\newtheorem{Definition2}[theorem]{Definition}
\newtheorem*{Definition2K}{Definition}
\newtheorem*{Definition21}{Definition}
\newtheorem*{Definition217}{Definition (\cite{HK} P.53)}
\newtheorem*{Definition224}{Definition}
\newtheorem{Sublemma}[theorem]{Sublemma}
\newtheorem{Remark}[theorem]{Remark}
\newtheorem{Proposition}[theorem]{Proposition}
\newtheorem*{Proposition55}{Proposition}
\newtheorem{Fact}{Fact}
\newtheorem*{theoremB}{Theorem E (Bowen)}
\newtheorem*{theoremcontinub}{Theorem F (Berger \cite{b9} Prop. 1.6,  Theorem C.5 \cite{bbb137})}
\newtheorem*{Definition22}{Definition (\cite{b9, bbb})}
\newtheorem*{Question}{Question}

\maketitle
\abstract{ \noindent A blender for a surface endomorphism  is a hyperbolic basic set for which the union of the local unstable manifolds contains robustly  an open set. Introduced by Bonatti and Díaz in the 90s, blenders turned out to have many powerful applications to differentiable dynamics. In particular, a generalization in terms of jets, called parablenders, allowed Berger to prove the existence of generic  families displaying robustly infinitely many sinks. In this paper, we introduce analogous notions in a measurable point of view. We define an almost blender as a hyperbolic basic  set for which a prevalent perturbation has a local unstable set having positive Lebesgue measure. Almost parablenders are defined similarly in terms of jets. We study families   of endomorphisms  of $\mathbb{R}^2$   leaving invariant  the continuation of  a  hyperbolic basic  set. When some inequality involving the entropy and  the  maximal contraction along stable manifolds is satisfied, we  obtain  an almost blender or parablender. This answers partially a conjecture of Berger.     The proof is based on thermodynamic formalism: following works of Mihailescu, Simon, Solomyak and Urbański, we study families of fiberwise unipotent  skew-products   and we give conditions under which these maps have limit sets of positive measure inside their fibers. \blfootnote{\footnotesize{The author is supported by ERC project 818737 \emph{Emergence of wild differentiable dynamical systems}.}}
 
 }
\tableofcontents
\newpage
\section{Introduction}

\subsection{Blenders and almost blenders}

Fractal sets have played a central role in the development of differentiable dynamics. Among several examples, a central notion is that of blender, cast by Bonatti and Díaz in the 90s. It was first introduced in the invertible setting in \cite{bt5}  to construct robustly transitive nonhyperbolic diffeomorphisms. A blender is a hyperbolic basic set on which the dynamics has a special behavior: its unstable set forms an   «impenetrable wall» in the sense that it intersects any perturbation of a submanifold of dimension lower than the   
 stable dimension. In the case of surface endomorphisms, this notion takes the following simpler form:
  \vspace{0.3cm}

\begin{Definition2K}
A $C^{r}$-{\bf blender} for a $C^{r}$-endomorphism $\mathcal{F}$ of a surface $\mathcal{S}$ is a hyperbolic basic  set $\mathcal{K}$ s.t. an  union of its local unstable manifolds has $C^r$-robustly a non-empty interior: there exists a non-empty open set  $\mathcal{U} \subset \mathcal{S}$ included in an union of  local unstable manifolds of the continuation $\tilde{\mathcal{K}}$ of $\mathcal{K}$ for any map $\tilde{\mathcal{F}}$ $C^{r}$-close to $\mathcal{F}$.
\end{Definition2K} 

  \vspace{0.2cm}
  
  \noindent This property turned out to have many other powerful applications: for example $C^1$-density of stable ergodicity \cite{sc}, robust homoclinic tangencies \cite{bt7,bibi3} and thus  Newhouse phenomenon, the existence of generic  families displaying robustly infinitely many sinks \cite{b9}, robust bifurcations in complex dynamics \cite{dm,taflin,bibi}, fast growth of the number of periodic points \cite{bb,AST} ... Thus the following question is of fundamental interest: \emph{when do blenders appear ?} \newline
  
  \noindent In this direction,   Berger proposed the following conjecture:

   \vspace{0.3cm}

\begin{coroC15}
Let $\mathcal{F}$ be a $C^r$-local diffeomorphism of a manifold $\mathcal{M}$, for $r \ge 2$.  Let $\mathcal{K}$ be a hyperbolic basic set for $\mathcal{F}$. Suppose that the topological entropy $h_\mathcal{F}$ of $\mathcal{F} | \mathcal{K}$ satisfies:
$$ h_\mathcal{F} > \mathrm{dim} \text{ } \mathcal{E}^s  \cdot   |   \log m( D \mathcal{F} )  |  \hspace{0.5 cm}  \text{ with }   \hspace{0.5 cm}
m( D  \mathcal{F} ) := \min \limits_{z \in \mathcal{K}, u \in \mathcal{E}^s_z , || u|| =1}     \text{ }  || D_z \mathcal{F}(u) ||   \, ,$$
  and $\mathcal{E}^s$ the stable bundle of $\mathcal{K}$. Then there exists a $C^r$-neighborhood $\mathcal{U}$ of $\mathcal{F}$ and an infinite codimensional subset $\mathcal{N} \subset  \mathcal{U}$ such that for every $\tilde{\mathcal{F}} \in \mathcal{U} \setminus \mathcal{N}$, the continuation $\tilde{\mathcal{K}}$ of $\mathcal{K}$ is a $C^r$-blender.
\end{coroC15}

  \vspace{0.3cm}
  
  \noindent Note that we cannot hope that $\mathcal{K}$ is itself systematically a $C^r$-blender under the assumptions of Conjecture A.  Here is an easy counterexample. \newline
  
 \noindent {\bf Counterexample: } Let us consider the doubling map $f : x \mapsto 2x \text{ mod } 1$ on the circle $\mathbb{S}:= \mathbb{R} / \mathbb{Z}$. The whole circle is a hyperbolic basic set of repulsive type and $f$ is a $C^\infty$-local diffeomorphism. The  topological entropy $h_f$  is equal to $\log 2 >0$. Let us pick $\lambda<1$ close to 1 s.t.  $h_f  > |   \log  \lambda   |$. 
    The map $\mathcal{F} : (x,y) \in \mathbb{S} \times \mathbb{R}  \mapsto (f(x), \lambda y) \in \mathbb{S} \times \mathbb{R}$ leaves invariant the hyperbolic basic set  of saddle type $\mathcal{K}:= \mathbb{S} \times \{ 0 \}$. Moreover $\mathcal{F}$ is a  $C^\infty$-local diffeomorphism and $h_\mathcal{F} = h_f  > |   \log  \lambda   |$. However the unstable set of $\mathcal{K}$ is included in $\mathbb{S} \times \{0\}$ and thus has empty interior. 
\newline

\noindent  Conjecture A is also linked to a program proposed by  Díaz \cite{puc} on the thermodynamical study of blenders. \newline

\noindent In this article,  we give an answer  to both questions, in a  measurable point of view. 
We are going to define a measurable variant  of the notion of blender, called «almost blender». This will be a  hyperbolic basic set having its  unstable set of positive Lebesgue measure instead of being with non-empty interior. Also, this property will be asked to be 
« robust » in a measurable way instead of topological. Let us precise this « measurable robustness ». We recall that this notion is tricky since there is no canonical measure on the space $C^r(\mathcal{M}, \mathcal{M})$ of $C^r$-endomorphisms of a manifold $\mathcal{M}$. But we need an  analogous of the finite-dimensional notion of « Lebesgue almost every » in an infinite-dimensional setting.  Nevertheless, there are several notions of prevalence or typicity which generalize this concept. A panorama    has been drawn by Hunt and Kaloshin in \cite{HK}, by Ott and Yorke in \cite{OY} or by Ilyashenko and Li in \cite{IL}.  Here is one of these notions of prevalence, particularly adapted to our case:

  \vspace{0.3cm} 

  \begin{Definition217}
  We say that a set $E$ in a Banach space $B$ is \emph{finite-dimensionally  prevalent} if there exists a continuous family  $(v_q)_{q \in \mathcal{Q}}$ of vectors $v_q  \in B$, for $q$ varying in a neighborhood $\mathcal{Q}$ of $0$ in $\mathbb{R}^m$ with $m >0$ and $v_0 = 0$, s.t. for every fixed $v \in B$, we have that $v + v_q \in E$ for $\mathrm{Leb}_m$ a.e. $q \in  \mathcal{Q}$.
 \end{Definition217}

  
   \vspace{0.3cm} 
  
  \noindent In other terms, we require that for some finite-dimensional family of perturbations, if we start at any point in $B$, then by adding a perturbation randomly chosen with respect to the Lebesgue measure, we are in $E$ with probability 1. 
  A similar notion, simply called « prevalence », has been designed by Sauer, Ott and Casdagli (see Definition 3.5 in \cite{OY}, or also \cite{SYC} and  \cite{HSY}), for completely metrizable topological vector spaces and with the additive condition that $v_q$ is a linear function of $q$. In our results, we will have this additional linearity but since we do not need it,  we will take inspiration from the above definition. See Remark 1 P.53 in \cite{HK} for details.
  \newline

  \noindent  We restrict ourselves in this article to the case where the manifold $\mathcal{M}$ is equal to $\mathbb{R}^2$ and endowed with its usual Euclidean metric. The vector space  $C^r(\mathbb{R}^2,\mathbb{R}^2)$ of $C^r$-endomorphisms of $\mathbb{R}^2$ is endowed with the topology given by the uniform $C^r$-norm:
  $$ || \mathcal{F} ||_{C^r} :=  \sup_{0 \le i \le r, z \in \mathbb{R}^2} ||D^i_z  \mathcal{F}  || $$
  when $0  \le r < \infty$. 
The space of $C^r$-bounded $C^r$-endomorphisms endowed with this norm is a Banach space. Since we are interested in properties depending only on perturbations of a map on a compact set, we can restrict ourselves to $C^r$-bounded $C^r$-endomorphisms if necessary and so the above definition of prevalence from  \cite{HK} could fit to our setting. Similarly, for  
$r =\infty$, we endow $ C^\infty(\mathbb{R}^2,\mathbb{R}^2)$ with the union of uniform $C^s$-topologies on $K_j$  among integers $s$ and $j$, for an exhausting sequence of compact sets $K_j$ of $\mathbb{R}^2$,  which  gives to $ C^\infty(\mathbb{R}^2,\mathbb{R}^2)$ a complete metrizable topology. \newline

\noindent
However we cannot hope that a blender-like property, even in a weak sense, holds true densely in $C^r(\mathbb{R}^2,\mathbb{R}^2)$, even less in a prevalent way. This is why we introduce the following immediate adaptation for perturbations of a map:

   \vspace{0.3cm}

 \begin{Definition21}
 A property $(P)$ holds true for a \emph{prevalent $C^r$-perturbation} of a $C^r$-endomorphism $\mathcal{F}$ of  $\mathbb{R}^2$ if there exists a  $C^r$-neighborhood $\mathcal{U}$ of  $\mathcal{F}$ and a continuous family  $(\Sigma_q)_{q \in \mathcal{Q}}$ of $C^r$-endomorphisms $\Sigma_q$ of $\mathbb{R}^2$, for $q$ varying in a neighborhood $\mathcal{Q}$ of $0$ in $\mathbb{R}^m$ with $m >0$ and $\Sigma_0 = 0$, s.t. for every fixed $\mathcal{G} \in \mathcal{U}$, the map $\mathcal{G}+\Sigma_q$ satisfies $(P)$ for $\mathrm{Leb}_m$ a.e. $q \in  \mathcal{Q}$.
 \end{Definition21}

\noindent In particular, property  $(P)$ holds true for an  arbitrary small perturbation of $\mathcal{F}$.  \newline 

\noindent 
   Here is now a new concept, which formalizes a measurable variant of blenders.

   \vspace{0.3cm} 

\begin{Definition21}
A  hyperbolic basic set $\mathcal{K}$ for a  $C^r$-endomorphism $\mathcal{F}$ of  $\mathbb{R}^2$, with $r \ge 1$,  is an {\bf almost $C^r$-blender} if an  union of     local unstable manifolds of the continuation $\tilde{\mathcal{K}}$  of $\mathcal{K}$              has  positive  measure    for a    prevalent   $C^r$-perturbation  $\tilde{\mathcal{F}}$  of $\mathcal{F}$:
$$\mathrm{Leb}_{2}( W^u_{\mathrm{loc}}( \tilde{\mathcal{K}})) >0    \, .  $$  
\end{Definition21}

  \vspace{0.2cm}

\noindent A \emph{hyperbolic basic set} for $\mathcal{F}$ is a compact, $\mathcal{F}$-invariant, hyperbolic, transitive set $ \mathcal{K} $   
s.t.  periodic points of  $\mathcal{F} |  \mathcal{K} $  are dense in $ \mathcal{K} $   (basic notions about hyperbolic sets for endomorphisms are in the Appendix). \newline 

\noindent 
Our first result  gives an answer to   Conjecture A, in a measurable point of view:

    \vspace{0.3cm}

  \begin{theoremliloumMm} 
Let $\mathcal{F}$ be a  $C^r$-local diffeomorphism of $\mathbb{R}^2$, with $ 2 \le r \le \infty$.  Let $\mathcal{K}$ be a hyperbolic   basic set for $\mathcal{F}$. 
Suppose  that the topological entropy $h_{\mathcal{F}}$ of $\mathcal{F} |  \mathcal{K}$  satisfies $   h_\mathcal{F}    >    | 
\log m(D \mathcal{F}) |$.  Then  $\mathcal{K}$ is an almost $C^r$-blender. 
\end{theoremliloumMm}

  \vspace{0.3cm}

\noindent 
     Conjecture A seems a very difficult problem in its full generality. A related question is     the following long-standing open problem:

  \vspace{0.3cm}

\begin{coroC16}
Let $\mu$ be the self similar measure associated to some IFS $\Psi = (\psi^a)_{a \in \mathcal{A}}$ formed by a finite number of contracting similarities $\psi^a$ on  $\mathbb{R}$. Suppose that there are no exact overlaps and that the similarity dimension of the IFS is strictly larger than 1. Then $\mu$ is absolutely continuous with respect to $\mathrm{Leb}_1$. 
\end{coroC16}

  \vspace{0.3cm}

\noindent One can refer to the survey of Hochman  \cite{rio} for more details. Let us also point out that the creation of blenders had also been investigated by Moreira and Silva \cite{co}.

 \subsection{Parablenders and almost parablenders}

\noindent Berger introduced in \cite{b9} a variant of blenders, defined for families of maps this time,  where not only the unstable set of a hyperbolic set, but also the set of jets of points inside  unstable manifolds  contains an open set. Such sets were named parablenders (« para » standing for « parameter »). Parablenders were introduced to prove the existence of generic families displaying robustly infinitely many sinks, which gave a counter-example to a conjecture of Pugh and Shub from the 90s \cite{ps}.  
  \vspace{0.3cm}

\begin{Definition22}
A {\bf $C^{r}$-parablender} at $p_0 \in \mathcal{P}$ for a $C^{r}$-family $(\mathcal{F}_p)_{p \in \mathcal{P}}$ of    endomorphisms of a surface   $\mathcal{S}$, $r \ge 1$, parametrized by a parameter $p$   in an open subset $\mathcal{P} \subset \mathbb{R}^d$, is a continuation $(\mathcal{K}_p)_{p \in \mathcal{P}}$ of hyperbolic basic   sets $\mathcal{K}_p$ for $\mathcal{F}_{p}$ s.t.:
\begin{itemize}
\item for every $(\gamma_p)_{p \in \mathcal{P}}$ in a non-empty open set of $C^r$-families of points $\gamma_p \in \mathcal{S}$,
\item for every $C^{r}$-family $(\tilde{\mathcal{F}}_p)_{p \in \mathcal{P}}$ of endomorphisms $C^r$-close to $(\mathcal{F}_p)_{p \in \mathcal{P}}$,
\end{itemize}
there exists a $C^r$-family $(\zeta_p)_{p \in \mathcal{P}}$ of points $\zeta_p \in \mathcal{S}$ s.t.:
\begin{itemize}
\item there is a local unstable manifold of $\mathcal{K}_{p_0}$ whose continuation for  $\tilde{\mathcal{F}}_p$ contains $\zeta_p$,  for any $p \in \mathcal{P}$,
\item the $r$-jets of $\zeta_p$ and $\gamma_p$ at $p_0$ are equal:
$$     (\zeta_p, \partial_p \zeta_p , \ldots, \partial^r_p  \zeta_p  )_{     | p = p_0} =  (\gamma_p, \partial_p \gamma_p, \ldots, \partial^r_p \gamma_p )_{     | p = p_0}     \, .$$
\end{itemize}
\end{Definition22} 

  \vspace{0.3cm}

\noindent In particular, $\mathcal{K}_{p_0}$  is a  $C^r$-blender for $\mathcal{F}_{p_0}$ if $(\mathcal{K}_p)_{p \in \mathcal{P}}$ is a  $C^r$-parablender for $(\mathcal{F}_p)_{p \in \mathcal{P}}$  at $p_0 $. In a subsequent work \cite{bb}, Berger used parablenders  to prove the existence of generic families of maps displaying robustly a fast growth of the number of periodic points, solving a problem  of Arnold \cite{ar} in the finitely differentiable case. \newline


\noindent  From now on, we work with $C^r$-families $(\mathcal{F}_p)_{p   }$ of endomorphisms $\mathcal{F}_p$ of  $\mathbb{R}^2$, with $2 \le r \le \infty$, parametrized by a parameter $p$ varying in   $\mathcal{P} := (-1,1)^d$ with $1\le d<\infty$. In fact, we will need to work with families which admit some extension on a larger parameter space. We therefore fix an  open set $\mathcal{P}’ \subset \mathbb{R}^d$ s.t. $\mathcal{P} \Subset \mathcal{P}’$. We then define a $C^r$-family $(\mathcal{F}_p)_{p }$ of endomorphisms $\mathcal{F}_p$ of  $\mathbb{R}^2$ as an element of $C^r(\mathcal{P}’ \times \mathbb{R}^2, \mathbb{R}^2 )$. We endow this   space   with the uniform $C^r$-topology when $0  \le r < \infty$, and with the union of uniform $C^s$-topologies on $K_j$  among integers $s$ and $j$, for an exhausting sequence of compact sets $K_j$ of $\mathcal{P}’ \times \mathbb{R}^2$ when $r=\infty$. Note that for simplicity we denote often in the following  $(\mathcal{F}_p)_{p \in \mathcal{P}}$ this family since we are interested mainly in the dynamics when $p \in \mathcal{P}$ but keep in mind that it admits such an extension. 
 Let $(\mathcal{K}_p)_{p \in \mathcal{P}}$ be the (hyperbolic) continuation of  a  hyperbolic basic set (extending to $\mathcal{P}’$). Let  $\mathcal{E}^s_p$ and $\mathcal{E}^u_p$ be the one-dimensional  stable and unstable bundles of $\mathcal{K}_p$. 
\newline

\noindent
Our main result deals with jets of points inside local unstable manifolds of $\mathcal{K}_p$. 
 Let $(M_p)_p$ be a $C^r$-curve of points $M_p$ in the continuation of  one   local unstable     manifold  of $\mathcal{K}_p$.
  For any integer $s \le r$, one can consider the $s$-jet of $M_p$ at any $p_0 \in \mathcal{P}$:
$$\mathrm{J}_{p_0}^{s} M_p := (M_p, \partial_p M_p , \ldots, \partial^{s}_p M_p)_{     | p = p_0} \, .$$
An  interesting set is then the set $\mathrm{J }_{p_0}^s W^u_{\mathrm{loc}}( \mathcal{K}_p)$ of all the $s$-jets among such curves $(M_p)_p$. When this set has robustly a non-empty interior,  $(\mathcal{K}_p)_{p \in \mathcal{P}}$ is a  $C^{s}$-parablender at $p_0$. 
Let  $\delta_{ d, s} $  be the dimension of the set of jets in $d$ variables of order $s$ in one dimension, which is  that of the space  $\mathbb{R}_s[X_1, \cdots, X_d]$ of polynomials  in $d$ variables of degree at most  $s$. In particular, notice that  the space  of jets of order $s$ of maps from  $\mathcal{P}$ to $\mathbb{R}^2$ is of dimension $2\delta_{d,s} $. \newline 

\noindent Here is the counterpart for families of the definition of a prevalent $C^r$-perturbation:

  \vspace{0.3cm}
  
   \begin{Definition21}
 A property $(P)$ holds true for a \emph{prevalent $C^r$-perturbation} of a $C^r$-family $(\mathcal{F}_p)_{p \in \mathcal{P} }$ of endomorphisms of  $\mathbb{R}^2$ if there exists a  $C^r$-neighborhood $\mathcal{U}$ of  $(\mathcal{F}_p)_{ p \in \mathcal{P} }$ and a  continuous family  $(\Sigma_{q})_{q \in \mathcal{Q}}$ of $C^r$-families $\Sigma_q$ of endomorphisms of $\mathbb{R}^2$, for $q$  in a neighborhood $\mathcal{Q}$ of $0$ in $\mathbb{R}^m$ with $m >0$ and $\Sigma_0 = (0)_{p \in \mathcal{P}}$, s.t. for every fixed family   $(\mathcal{G}_p)_p \in \mathcal{U}$, the family $(\mathcal{G}_p)_p+\Sigma_q$ satisfies $(P)$ for $\mathrm{Leb}_m$ a.e. $q \in  \mathcal{Q}$.
 \end{Definition21}


 \vspace{0.2cm}

\noindent  In particular, property  $(P)$ holds true for an  arbitrary small perturbation of $(\mathcal{F}_p)_{p \in \mathcal{P}}$. \newline

\noindent 
 The following is an   analogous of $C^s$-parablenders, in a  measurable point of view: 
  \vspace{0.1cm}

\begin{Definition224}
The continuation $(\mathcal{K}_p)_{p \in \mathcal{P}}$  of a hyperbolic basic set for a  $C^r$-family $(\mathcal{F}_p)_{p \in \mathcal{P}}$ of endomorphisms  of  $\mathbb{R}^2$, with $r \ge 1$,  is an  {\bf  almost  $C^{r,s}$-parablender}, with $s$ an integer  s.t. $s  \le r$, if for a prevalent   $C^r$-perturbation $(\tilde{\mathcal{F}}_p)_{p  \in \mathcal{P}}$ of $(\mathcal{F}_p)_{p  \in \mathcal{P}}$, the  continuation $(\tilde{\mathcal{K}}_p)_{p \in \mathcal{P}}$ of  $(\mathcal{K}_p)_{p \in \mathcal{P}}$ satisfies:
$$\mathrm{Leb}_{ 2\delta_{d,s} }( \mathrm{J }_{p_0}^s W^u_{\mathrm{loc}}( \tilde{\mathcal{K}}_p)) >0 \text{  for } \mathrm{Leb}_{d} \text{  a.e. }  p_0  \in \mathcal{P}  \, . $$  
\end{Definition224}

\vspace{0.2cm}

\noindent Note that if $(\mathcal{K}_p)_{p \in \mathcal{P}}$ is an  almost $C^{r,s}$-parablender and $p$ is a  parameter in $\mathcal{P}$, the set $\mathcal{K}_p$  is an   almost $C^r$-blender. \newline

\noindent 
Here is our main result, which generalizes Theorem A in terms of jets:

  \vspace{0.3cm}

  \begin{theoremliloum} 
Let $(\mathcal{F}_p)_{p \in \mathcal{P}}$ be a $C^r$-family of local diffeomorphisms of  $\mathbb{R}^2$, with $2 \le r \le \infty$.  Let $(\mathcal{K}_p)_{p \in \mathcal{P}}$ be the continuation of a hyperbolic   basic set for $(\mathcal{F}_p)_{p \in \mathcal{P}}$.
 Take an integer $s \le r-2$ and suppose that the topological entropy $h_{\mathcal{F}_p}$ of $\mathcal{F}_p |  \mathcal{K}_p$  satisfies:
$$ (\star)  \hspace{0.3cm}  \text{ }  h_{\mathcal{F}_p}    > \delta_{d,s} \cdot     | \log m(D \mathcal{F}_p) |   \hspace{1cm} \forall p \in \mathcal{P}’
 \,  .     $$
\noindent  Then $(\mathcal{K}_p)_{p \in \mathcal{P}}$ is an almost  $C^{r,s}$-parablender. 
%
\end{theoremliloum}

  \vspace{0.3cm}
%
%
%

\noindent 
This second  result also goes  in the direction of Conjecture A, both in a measurable point of view and in terms of jets this time. Let us mention that both  Theorems A and B still hold true if we  assume that the maps involved are local diffeomorphisms only in a neighborhood of the basic sets. Last but not least, we hope to use Theorem B  to solve the conjecture of Pugh and Shub  \cite{ps} in the smooth $C^\infty$ case which is not handled by \cite{b9}. Finally, let us mention the following immediate question:

\begin{Question}
Is it possible to generalize Theorems A and B to the case where $\mathcal{M}$ is any surface (not necessarily equal to $\mathbb{R}^2$) and for the alternative  notion of prevalence defined by Kaloshin in this context ? We recall that this latter  one is defined as follows: a subset $E \subset C^r(\mathcal{M},\mathcal{M})$ is {\bf strictly $n$-prevalent} if there exists an open dense set of $n$-parameter families $(\mathcal{F}_p)_p$ s.t. $\mathcal{F}_p \in E$ for a.e. $p$ and if for every $\mathcal{F} \in C^r(\mathcal{M},\mathcal{M})$, there exists such a family with $\mathcal{F}_0 = \mathcal{F}$. A {\bf $n$-prevalent} set is a countable intersection of strictly $n$-prevalent sets.   \end{Question}

\noindent {\bf Acknowledgements:} The author would like to thank Pierre Berger for introducing him this topic, and also for many invaluable encouragements and suggestions which improved a lot this manuscript. The author is also grateful to Fran\c cois Ledrappier for helpful discussions.

\section*{Combinatorics and notations}  \label{combi} 
Let $\mathcal A$ be a  finite alphabet  of cardinality at least  $2$.  Let 
 $$\overrightarrow{ \mathcal A}= \mathcal A^{\mathbb{N}}   \text{ , } \overleftarrow{ \mathcal A}= \mathcal A^{\mathbb{Z}^*_-} \text{  and  } \overleftrightarrow{ \mathcal A}= \mathcal A^{\mathbb{Z}}$$
be the sets of infinite  forward, backward and bilateral  words with letters in $\mathcal A$. We consider the left full shift on $\overrightarrow{ \mathcal A}   $ or  $\overleftrightarrow{ \mathcal A}$:
$$ \sigma: \alpha= (\alpha_i)_i  \in \overrightarrow{ \mathcal A}    \sqcup \overleftrightarrow{ \mathcal A}   \mapsto \sigma (\alpha) =  (\alpha_{i+1})_i  \in \overrightarrow{ \mathcal A}    \sqcup \overleftrightarrow{ \mathcal A}  $$ and  the right full shift  on $\overleftarrow{ \mathcal A}$:
$$\sigma:  \alpha = (\alpha_i)_i  \in \overleftarrow{ \mathcal A}   \mapsto \sigma (\alpha) =  (\alpha_{i-1})_i \in    \overleftarrow{ \mathcal A} \, .$$
\noindent In particular  these full shifts are of positive entropy and topologically mixing. We also define $\mathcal A^*$ as the set of finite words  with letters in $\mathcal A$ and  denote by $e$  the empty word.
We endow $\overleftrightarrow{ \mathcal A}$ with the distance given by $d_\infty(\alpha,\beta) = D^{q} $ for every $\alpha = (\alpha_i)_{i}  \in  \overleftrightarrow{ \mathcal A}$ and $\beta= (\beta_i)_{i} \in \overleftrightarrow{ \mathcal A}$, where $D \in (0,1)$ is a fixed number and $q$ is the largest integer such that $\alpha_i = \beta_i$ for every $|i| < q$ if $\alpha \neq \beta$. We endow $\overrightarrow{ \mathcal A}$ with a metric defined similarly.    \newline

\noindent
For $\alpha \in \mathcal A^* \cup \overrightarrow{ \mathcal A} \cup \overleftarrow{ \mathcal A}  \cup \overleftrightarrow{ \mathcal A} $, let $| \alpha  | \in \mathbb{N} \cup \{+ \infty\}$ be  the number of letters in $\alpha$. When $| \alpha  | >n$ for some integer $n>0$, we call  $\alpha_i$ the $i^{th}$ letter of $\alpha$ and denote $\alpha_{|  n} :=(\alpha_0, \cdots ,\alpha_{n-1})$ when $\alpha \in \mathcal A^* \cup \overrightarrow{ \mathcal A}  $ and $\alpha_{|  n} :=(\alpha_{-n}, \cdots, \alpha_{-1})$  when $\alpha \in \overleftarrow{ \mathcal A} \cup \overleftrightarrow{ \mathcal A} $. Finally, for $\alpha = (\alpha_{-n}, \cdots, \alpha_{-1})  \in  \mathcal A^*$, let  $[\alpha]$ be the corresponding cylinder in $\overleftrightarrow{ \mathcal A} $: 
$$[\alpha] := \{    \beta  \in \overleftrightarrow{ \mathcal A}   :  \beta_i = \alpha_i \text{ } \forall  -n \le i \le  -1    \} \, .$$ 
\noindent We define similarly cylinders in  $\overleftarrow{ \mathcal A}$ and $\overrightarrow{ \mathcal A}$ and use the same notation.  {\bf  Greek (resp. gothic) letters will be used for finite or backward infinite (resp. forward infinite) words.}  For $\mathfrak{a}  \in \overrightarrow{ \mathcal A}$, $\alpha \in  \overleftarrow{ \mathcal A}$, $\beta \in \mathcal A^*$, we denote by $ \alpha \mathfrak{a}, \beta \mathfrak{a}, \alpha \beta$ their concatenations. The topological closure of a set relatively to the Euclidean distance is denoted with an overline. The notation $\preceq$ stands for the usual domination relation and $f \asymp g$ means that $f \preceq g$ and $g \preceq f$.

\section{Example} \label{FDRX}

We give here an application of our results. More precisely, we provide simple examples of almost blenders and parablenders. \newline 

\noindent Let us consider the segment $X = [-1,1]$. We pick three integers $n \ge 2$, $d \ge 1$ and $s \ge 0$. We choose $n’:=(n+1)^{\lceil \delta_{d,s} \rceil}  $ disjoint subsegments $X_j \Subset X$ and $n’$ numbers $0<r_j<1-1/n$, for $1 \le j \le n’$. Let $g_j$ be the affine preserving order map sending $X_j$ onto $X$. We pick a $C^\infty$-map $g : \mathbb{R} \rightarrow \mathbb{R}$ which is equal to $g_j$ on a small neighborhood of each interval $X_j$ and also a $C^\infty$-map $h : \mathbb{R} \rightarrow \mathbb{R}$ which is equal to $r_j$ on a small neighborhood of each interval $X_j$. The following $C^\infty$-endomorphism 
$$\mathcal{F} : (x,y) \in \mathbb{R}^2 \mapsto (g(x), \frac{y}{n}+  h(x))  \in \mathbb{R}^2 $$

\noindent  is a local diffeomorphism on a small neighborhood of $U:=\bigsqcup_{1 \le j \le n’} X_j \times X$. It is easy to verify that the set
\begin{equation} 
\mathcal{K}:= \bigcap_{ n \in \mathbb{Z}} \mathcal{F}^n(U) 
\end{equation}
   is a compact, hyperbolic, invariant, locally maximal set, with stable and unstable  dimensions equal to 1. 
   This remains true for any $C^\infty$-endomorphism $\tilde{\mathcal{F}}$ which is $C^\infty$-close to $\mathcal{F}$, with the  same formula.
  \begin{center}
  \includegraphics[width=6cm]{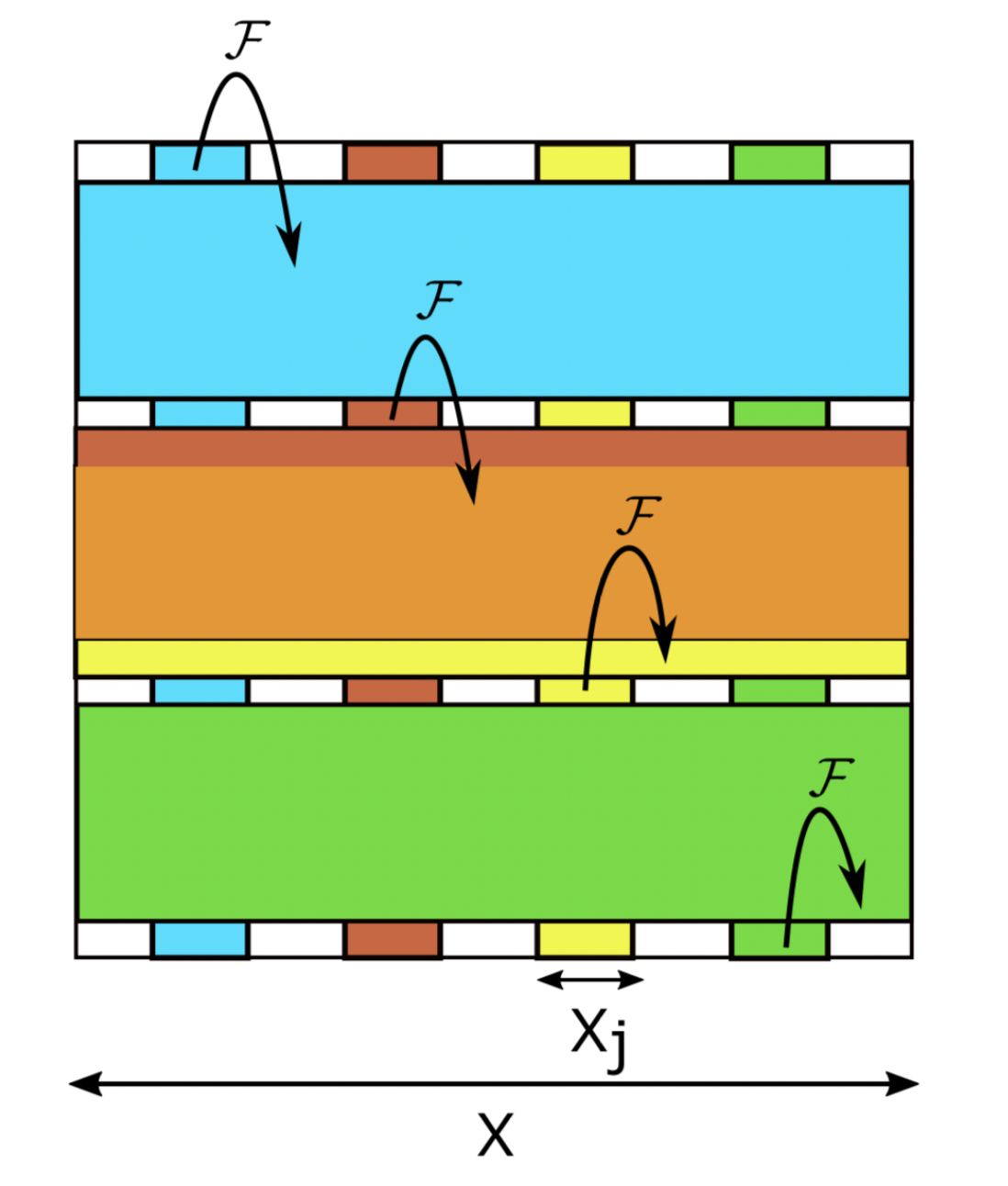}
  \end{center}
We set $\mathcal{A}:=\{1, \cdots, n’\}$ and call $\mathcal{F}_j$ the restriction of $\mathcal{F}$ on $X_j \times X$. For  $\mathfrak{a} = (\mathfrak{a}_i)_{i \ge 0} \in \overrightarrow{\mathcal A}$ and $\alpha=(\alpha_i)_{i<0} \in \overleftarrow{\mathcal A}$, the following are   local stable and unstable  manifolds of $\mathcal{K}$:
   \begin{equation}  \label{W}
    W^{\mathfrak{a}} = \bigcap_{ j \ge 0}   \mathrm{Dom}    (\mathcal{F}_{\mathfrak{a}_{j}}  \circ \cdots  \circ    \mathcal{F}_{\mathfrak{a}_0}) \text{ and  }      W^{\alpha} = \bigcap_{ j<0} \mathrm{Im}  (\mathcal{F}_{\alpha_{j}}  \circ \cdots  \circ    \mathcal{F}_{\alpha_{-1}})  \,  , 
    \end{equation}
    \noindent   where the domains $ \mathrm{Dom}   ( \mathcal{F}_{\mathfrak{a}_{j}}  \circ \cdots  \circ    \mathcal{F}_{\mathfrak{a}_0} ) $ and $ \mathrm{Dom}   (\mathcal{F}_{\alpha_{j}}  \circ \cdots  \circ    \mathcal{F}_{\alpha_{-1}})$
   of  $\mathcal{F}_{\mathfrak{a}_{j}}  \circ \cdots  \circ    \mathcal{F}_{\mathfrak{a}_0}$ and $\mathcal{F}_{\alpha_{j}}  \circ \cdots  \circ    \mathcal{F}_{\alpha_{-1}}$ are  $( g_{\mathfrak{a}_{j}}  \circ \cdots  \circ    g_{\mathfrak{a}_0})^{-1}(X)  \times X $ and $(g_{\alpha_{j}}  \circ \cdots  \circ    g_{\alpha_{-1}})^{-1}(X)  \times X$. 
    \newline

 \noindent     It is immediate that $W^{\mathfrak{a}}$ is a vertical segment of second coordinate projection $X$. 
       By hyperbolic continuation, for  every $C^\infty$-endomorphism $\tilde{\mathcal{F}}$ which is $C^\infty$-close to $\mathcal{F}$ and every $\mathfrak{a} \in \overrightarrow{\mathcal A} $, we can define a local stable manifold $\tilde{ W^{\mathfrak{a}} } $ which is a  vertical graph of class $C^\infty$ over $y \in X$ with small slope.
         We notice that these local stable manifolds are pairwise disjoint.    
    We have analogous properties for local unstable   manifolds (except their disjointness) and  $\tilde{W}^{\mathfrak{a}}$ and $\tilde{W}^{\alpha} $ intersect in exactly one point. 
     \newline

     \noindent  
     Let us now consider any $d$-unfolding $(\mathcal{F}_{p})_{p \in \mathcal{P}}$ of $\mathcal{F}$, i.e. a $C^\infty$-family of   endomorphisms $\mathcal{F}_p$ s.t. $\mathcal{F}_0=\mathcal{F}$ and $\mathcal{P} = (-1,1)^d$. Up to restricting and then rescaling the parameter space, this family leaves invariant the continuation $(\mathcal{K}_p)_{p \in \mathcal{P}}$ of the hyperbolic  set $\mathcal{K}$ and $\mathcal{F}_p$ is a local diffeomorphism on a neighborhood of $\mathcal{K}_p$.  We can define families of  local stable and unstable manifolds $ W_p^{\mathfrak{a}}$ and  $ W_p^{\alpha }$  as above, and we denote by $\phi_p$ the map sending $\beta = \alpha \mathfrak{a}  \in \overleftrightarrow{\mathcal{A}}$ (with $\alpha \in \overleftarrow{\mathcal{A}}$ and $ \mathfrak{a} \in \overrightarrow{\mathcal{A}}$)  to the unique intersection point $\phi_p(\alpha \mathfrak{
     a} ) \in \mathcal{K}_p$ between   $ W_p^{\mathfrak{a}}$ and  $ W_p^{\alpha }$. We denote:
     $$\Phi_p : \beta \in \overleftrightarrow{\mathcal{A}} \mapsto \big( \phi_p( \sigma^i(\beta) ) \big)_i \in   \overleftrightarrow{\mathcal{K}_p} \, .$$
 \noindent    We notice that  $\Phi_p$ conjugates the full shift $(\overleftrightarrow{\mathcal{A}},\sigma)$ to  the dynamics $(\overleftrightarrow{\mathcal{K}_p}, \overleftrightarrow{\mathcal{F}_p})$ on the inverse limit and so periodic points are dense in $\mathcal{K}_p$ and $\mathcal{K}_p$ is transitive, and  thus a hyperbolic basic set. 
     The entropy $h_{\mathcal{F}_p}$ of $\mathcal{F}_p   |  \mathcal{K}_p$ is equal to $\mathrm{log} (n’)$ and its stable contraction is close to $1/n$. We recall that:
     $$ \frac{ \mathrm{log} (n’) } { 
     |\mathrm{log} (1/n)| } =\lceil \delta_{d,s} \rceil \cdot   \frac{\mathrm{log}(n+1)}{\mathrm{log}(n)}>\delta_{d,s} \ge 1\, .$$
     \noindent Thus $h_{\mathcal{F}}>   |\mathrm{log} (1/n)|$ and assumption   $(\star)$ holds  true for the family $(\mathcal{F}_{p})_{p \in \mathcal{P}}$. By
      Theorems A and B, we conclude the following: 
     
     \begin{Proposition55}
        The set  $\mathcal{K}$ is an almost $C^\infty$-blender and for any $d$-unfolding $(\mathcal{F}_{p})_{p \in \mathcal{P}}$ of $\mathcal{F}$,  
    $(\mathcal{K}_p )_{p  \in \mathcal{P}}$ is an almost $C^{\infty,s}$-parablender, up to restricting and rescaling $\mathcal{P}$. 
     \end{Proposition55}

\section{ Skew-product   Formalism and Strategy}

\noindent 
Our method is based on a method introduced by Mihailescu, Simon, Solomyak and Urbański. Let us give some details. For IFS without overlaps, the Hausdorff dimension of the limit set is given by Bowen’s formula \cite{bo5}. In \cite{b3},  Simon, Solomyak and Urbański introduced a method to compute it even in the presence of overlaps. The key ingredient in their proofs (see Section \ref{simple}) is a transversality property (see also \cite{So} and \cite{PeSo} for more on transversality). This  also allows  to get parameters for which the limit set has positive measure, which is our interest. Later these results were extended by Mihailescu and Urbański to the case of hyperbolic and fiberwise conformal skew-products in \cite{b1}. Here we extend these to the setting of families of skew-products fiberwise unipotent.

\subsection{Skew-products}


We work with ($N$-dimensional) $C^r$-skew-products acting on $ \overrightarrow{ \mathcal A} \times [-1,1]^N$, where $N >0$. Here $\mathcal A$ is a   fixed finite alphabet  of cardinality at least  $2$.
For simplicity, we denote $X:= [-1,1]^N$. 
 The regularity $r$ of the maps is given by either an integer at least $2$ or $+\infty$.

  \vspace{0.3cm}

\begin{Definition2}
A {\bf pre-$C^r$-skew-product} is a map of the form:
$$ F : (\mathfrak{a},x) \in \overrightarrow{ \mathcal A} \times X \mapsto  (\sigma (\mathfrak{a}), f_{\mathfrak{a}}(x)) \in    \overrightarrow{ \mathcal A}   \times X $$
\noindent satisfying that  there exists an open set $X’ \subset \mathbb{R}^N$ independent of $\mathfrak{a}$ s.t. $X \Subset X’$ and s.t.   $f_{\mathfrak{a}}: X \rightarrow X$ extends to  a $C^r$-diffeomorphism from $X’$ to $f_{\mathfrak{a}}(X’) \Subset X$ for every $\mathfrak{a}$. \newline 

\noindent The map $F$ is a {\bf $C^r$-skew-product} if moreover the two following maps $$\mathfrak{a} \in \overrightarrow{ \mathcal A}  \mapsto   f_{\mathfrak{a}} \in C^0( X’, \mathbb{R}^N) \text{  and }   \mathfrak{a} \in \overrightarrow{ \mathcal A}  \mapsto Df_{\mathfrak{a}} \in C^0( X’ , \mathcal{L}(\mathbb{R}^N,\mathbb{R}^N))$$ 
\noindent  are H\"older with positive exponent, and the following third map is continuous:
 $$\mathfrak{a} \in \overrightarrow{ \mathcal A}  \mapsto D^2f_{\mathfrak{a}} \in C^0(X’, \mathcal{L}^2(\mathbb{R}^N,\mathbb{R}^N)) \, .$$
 \end{Definition2}

\noindent In the latter definition, $\overrightarrow{ \mathcal A}$ is endowed with its distance and the spaces of  $C^0$-maps from $X’$ to $\mathbb{R}^N$, to the space $\mathcal{L}(\mathbb{R}^N,\mathbb{R}^N)$ of linear maps from $\mathbb{R}^N$ to $\mathbb{R}^N$ and to the space $\mathcal{L}^2(\mathbb{R}^N,\mathbb{R}^N)$ of bilinear maps from $\mathbb{R}^N \times \mathbb{R}^N$ to $\mathbb{R}^N$ endowed with the uniform $C^0$-metric. In the following, we suppose that the extensions of the maps $f_{\mathfrak{a}}$ are fixed. \newline

\noindent We are even more interested in families $(F_p)_{p}$ of (pre-)$C^r$-skew-products, indexed by  $p$  varying in $ [-1,1]^d$, for $1 \le d< \infty$. 
We define such a family as a family of maps:
$$ F_p : (\mathfrak{a},x) \in \overrightarrow{ \mathcal A} \times X \mapsto  (\sigma (\mathfrak{a}), f_{p,\mathfrak{a}}(x)) \in    \overrightarrow{ \mathcal A}   \times X \, $$
s.t.  $\hat{F}: (\mathfrak{a}, (p,x)) \mapsto (\sigma (\mathfrak{a}), (p,f_{p,\mathfrak{a}}(x)))$ is a ($N+d$-dimensional) (pre-)$C^r$-skew-product. In particular, there exist open neighborhoods $X’$ and $\mathcal{P}’$ of $X$ and $\overline{\mathcal{P}}$ in $\mathbb{R}^N$ and $\mathbb{R}^d$ s.t.  $(p,x) \mapsto  (p,f_{p,\mathfrak{a}}(x))$ extends to a diffeomorphism from $\mathcal{P}’    \times    X’$ into  $    \mathcal{P}  \times    X$ for each $\mathfrak{a}$  and the map $F_p$ is a ($N$-dimensional) (pre-)$C^r$-skew-product for every $p \in \mathcal{P}’$. Again,  we denote  $(F_p)_{p \in \mathcal{P}}$ this family since we are interested mainly on the dynamics when $p \in \mathcal{P}$ but still keep in mind that it admits such an extension.\newline

%

\noindent We will say that such a family $(F_p)_{p \in \mathcal P}$ of (pre-)$C^r$-skew-products satisfies the {\bf Unipotent}    assumption  {\bf (U)}   when the following is satisfied: \newline

  \noindent {\bf (U)}: For any $p \in \mathcal{P}’ $, $\mathfrak{a} \in \overrightarrow{ \mathcal A}$ and $x \in X’$, the differential $D f_{p,\mathfrak{a}}(x)$ is inferior unipotent, that is an  inferior triangular matrix with all its diagonal coefficients equal the one other, and its  unique eigenvalue is strictly bounded between 0 and 1 in modulus.    \newline

\noindent We will adopt the following formalism. For every $p \in \mathcal{P}’$, $\mathfrak{a} \in \overrightarrow{ \mathcal A} $, $n >0$ and $\alpha = (\alpha_{-n}, \ldots, \alpha_{-1}) \in \mathcal{A}^*$, we set:
$$\forall x \in X’, \text{ } \psi_{ p,  \mathfrak{a} }^\alpha(x):=  f_{ p,  \alpha_{-1}  \mathfrak{a}   }   \circ \ldots \circ       f_{ p,    \alpha_{-n} \cdots \alpha_{-1}     \mathfrak{a}    } (x) \, .$$
We show below that a consequence of  {\bf (U)} is that $ \psi_{ p,  \mathfrak{a} }^\alpha$  is a $C^r$-contraction from $X’$ to $ \psi_{ p,  \mathfrak{a} }^\alpha(X’) \Subset X$ when $|\alpha |$ is large enough. If we now take an infinite backward sequence $\alpha = (\ldots, \alpha_{-n}, \ldots, \alpha_{-1}) \in \overleftarrow{ \mathcal A}$, we  see that the points $\psi_{p, \mathfrak{a} }^{\alpha_{ |n} }(0)$ converge to a point  $ \pi_{p, \mathfrak{a}} (\alpha) \in X$. This defines  a $C^0$-map $\pi_{p,\mathfrak{a}} : \overleftarrow{ \mathcal A} 
\rightarrow X$. 

  \vspace{0.3cm}
  
\begin{Definition2}
The {\bf limit set $K_{p,\mathfrak{a}}$} of the skew-product $F_p$ inside the $\mathfrak{a} $-fiber is:
$$K_{p,\mathfrak{a}}:= \pi_{p,\mathfrak{a}} (\overleftarrow{ \mathcal A} ) \, .$$
\end{Definition2} 
\noindent We will give conditions under which this set has positive measure. 

\vspace{0.3cm}

\noindent  
For a $C^1$-map $f : X \rightarrow X$, let $m(Df)$ and $M(Df)$ be the respective minimum and maximum of $||Df(x) \cdot u ||$ among $x \in X$ and $u \in \mathbb{R}^N$ s.t. $||u||=1$. We need the following thermodynamical  formalism:
\vspace{0.3cm}

\begin{Definition2} \label{pressiondef}
The pressure at the parameter $p$ in the $\mathfrak{a} $-fiber  is the map: 
\begin{equation} \Pi_{p, \mathfrak{a}}: s \in \mathbb{R}_+ \mapsto  \lim_{n \rightarrow + \infty} \frac{1}{n} \mathrm{log} \sum_{\alpha   \in   \mathcal A^{n}   }       M(D \psi_{p,\mathfrak{a} }^\alpha)^s  \,.    \nonumber
\end{equation} 
\noindent  When $ \Pi_{p,\mathfrak{a}} $ has a unique zero, we call it the {\bf similarity dimension}   inside the $\mathfrak{a} $-fiber.
\end{Definition2}  
 
\vspace{0.1cm}

\noindent In Proposition \ref{pression}, we show that both the pressure and the similarity dimension are well-defined and  independent of $\mathfrak{a}$.   In particular, we denote them by  $\Pi_{p}$ and $\Delta(p)$.
\newline 

\noindent  We  adopt the following terminology to denote perturbations with special properties:

\begin{Definition2}
Let $(F_p)_p$ be a family  of (pre-)$C^r$-skew-products and let us fix neighborhoods $X’$ and $\mathcal{P}’$ of $X$ and $\overline{\mathcal{P}}$ s.t.  $(p,x) \mapsto  (p,f_{p,\mathfrak{a}}(x))$ extends to a diffeomorphism from $\mathcal{P}’    \times    X’$ into  $    \mathcal{P}  \times    X$   for every $\mathfrak{a} \in \overrightarrow{\mathcal{A}}$. For any $\vartheta>0$, a  {\bf $\vartheta$-perturbation} of $(F_p)_p$ is a family of pre-$C^r$-skew-products $(\tilde{F}_p)_p$ s.t. the map $(p,x) \mapsto  (p,\tilde{f}_{p,\mathfrak{a}}(x))$ extends to a diffeomorphism from $\mathcal{P}’    \times    X’$ into  $    \mathcal{P}  \times    X$ for every $\mathfrak{a} \in \overrightarrow{\mathcal{A}}$ and:
$$\sup_{ \mathfrak{a} \in \overrightarrow{\mathcal{A}}  } || (p,x) \in  X’ \times \mathcal{P}’   \mapsto (f_{p,\mathfrak{a}}  - \tilde{f}_{p,\mathfrak{a}})(x) ||_{C^r} < \vartheta \, .$$
The family $(\tilde{F}_p)_p$ is a {\bf $\vartheta$-$\mathrm{U}$-perturbation} when it satisfies assumption $(\mathrm{U})$.

 \end{Definition2}
 
 \vspace{0.3cm}

\noindent  For $\vartheta$-$\mathrm{U}$-perturbations with small $\vartheta$, we will show that for any $\alpha \in \overleftarrow{\mathcal{A}}$, the points 
$$  \tilde{f}_{ p,  \alpha_{-1}  \mathfrak{a}   }   \circ \ldots \circ       \tilde{f}_{ p,    \alpha_{-n} \cdots \alpha_{-1}     \mathfrak{a}    }(0)$$
 converge to  $ \tilde{\pi}_{p, \mathfrak{a}} (\alpha) \in X$ s.t. $p \mapsto \tilde{\pi}_{p, \mathfrak{a}} (\alpha)$ is a $C^r$-map $C^r$-close to $p \mapsto \pi_{p, \mathfrak{a}} (\alpha)$. 
 \newline

\noindent We will consider also {\bf parameterized families of $\vartheta$-perturbations} $(\tilde{F}_{t,p})_p$:
$$ \mathbb{F}:= ((\tilde{F}_{t,p})_p)_{t \in \mathcal{T}} \, .$$ 
Here  $t$ varies in $\mathcal{T} :=(-1,1)^\tau$ with $\tau>0$ and   $(t,p,x) \mapsto  \tilde{f}_{t, p,\mathfrak{a}} (x)$ is  $C^r$ for every $\mathfrak{a} \in \overrightarrow{\mathcal{A}}$. When each $(\tilde{F}_{t,p})_p$ is a    $\vartheta$-$\mathrm{U}$-perturbation, we say that $ \mathbb{F}$ is a parameterized families of $\vartheta$-$\mathrm{U}$-perturbations. When $\vartheta$ is small enough, 
we will denote by $\tilde{\pi}_{t,p, \mathfrak{a}} (\alpha)$ the limit  point corresponding to any $\alpha \in \overleftarrow{\mathcal{A}}$
 and $\tilde{K}_{t,p,\mathfrak{a}}:= \tilde{\pi}_{t,p,\mathfrak{a}} (\overleftarrow{ \mathcal A} )$, and then the $C^r$-maps $p \mapsto   \tilde{\pi}_{t,p, \mathfrak{a}} (\alpha)   $ will be $C^r$-close to $p \mapsto \pi_{p, \mathfrak{a}} (\alpha)$, uniformly in $(t,\alpha)$. We will set conditions under which $\tilde{K}_{t,p,\mathfrak{a}}$ has positive Lebesgue measure for a.e. $t \in \mathcal{T}$.

\subsection{Strategy and Organization of the paper} \label{tz}

\noindent
 The strategy will be  to
 focus on the dynamics restricted to the local stable manifolds, which allows us to reduce the dynamics to that one of a  $C^r$-skew-product. 
 \newline  

\noindent 
Hence we forget for some time families of endomorphisms and we work with families  of (pre-)$C^r$-skew-products. We  say that such a family of (pre-)$C^r$-skew-products $(F_p)_{p \in \mathcal P}$ satisfying assumption {\bf (U)}  satisfies also the {\bf Transversality}   assumption  {\bf (T)}  when the following is satisfied:
  \newline

   \noindent  {\bf (T)}:
    There exists $C>0$ such that for every sequences $\mathfrak{a} \in \overrightarrow{ \mathcal A}$ and $\alpha, \beta \in \overleftarrow{ \mathcal A}$ satisfying $\alpha_{-1} \neq \beta_{-1}$, and for every $r >0$, we have:
$$ \mathrm{Leb}_d \{ p \in \mathcal{P}  : ||\pi_{p,\mathfrak{a}}(\alpha) - \pi_{p,\mathfrak{a}}(\beta)   || < r    \} \le Cr^N   \, ,$$
and moreover for every small $\vartheta>0$ and every family $\mathbb{F}$ of $\vartheta$-$\mathrm{U}$-perturbations, 
for any $t \in \mathcal{T}$, $\mathfrak{a} \in \overrightarrow{ \mathcal A}$,  $\alpha, \beta \in \overleftarrow{ \mathcal A}$ s.t. $\alpha_{-1} \neq \beta_{-1}$   and $r >0$ we have: 
$$\mathrm{Leb}_d \{ p \in \mathcal{P}  : ||\tilde{\pi}_{t,p,\mathfrak{a}}(\alpha) - \tilde{\pi}_{t,p,\mathfrak{a}}(\beta)   || < r    \} \le   C r^N \, .$$

\noindent 
 The main technical result to prove  Theorem B is the following. It  sets conditions under which a given family of skew-products intersects its fibers into a set of positive measure, up to perturbations.
  \vspace{0.1cm}
   \begin{theoremlilou} \label{Mmainintro}
Let $(F_p)_{p \in \mathcal{P}}$ be a family of  $C^r$-skew-products satisfying ${\bf(U)}$, ${\bf(T)}$ and  $\Delta(p)>N$ for any $p \in  \overline{\mathcal{P}}$. Then for every $\mathfrak{a} \in \overrightarrow{ \mathcal A}$, we have:
$$\mathrm{Leb}_N (K_{p,\mathfrak{a}}) >0 \text{  for } \mathrm{Leb}_{d} \text{  a.e. } p \in \mathcal{P}  \, ,$$  
\noindent 
 and for every family $\mathbb{F}$ of $\vartheta$-$\mathrm{U}$-perturbations of $(F_p)_{p \in \mathcal{P}}$  with small  $\vartheta$, it holds:
$$\mathrm{Leb}_N (\tilde{K}_{t,p,\mathfrak{a}}) >0 \text{  for }   \mathrm{Leb}_{d} \text{  a.e. } p \in \mathcal{P}   \text{ and }   \mathrm{Leb}_{\tau}  \text{  a.e. } t \in \mathcal{T}     \, .$$
\end{theoremlilou}

  \vspace{0.4cm}

\noindent  
Here is the strategy to prove Theorem C. For every parameter $p$ and $\mathfrak{a}$-fiber, we define a probability measure $\nu_{p,\mathfrak{a}}$ supported on the limit set $K_{p,\mathfrak{a}}$. To show that $K_{p,\mathfrak{a}}$ has positive measure, it is enough to show that $\nu_{p,\mathfrak{a}}$  is absolutely continuous relatively to the $N$-dimensional Lebesgue measure, and so to prove that its density is finite almost everywhere. We compute the integral of the density  relatively to the parameter and the phase space. The trick is to use the Fubini Theorem to integer first relatively to $p$. The finiteness of the integral is implied by the transversality assumption {\bf (T)} and the inequality $\Delta(p)>N$. The same method will give the same results for families of   $\vartheta$-$\mathrm{U}$-perturbations, with  additional integration relatively to $t$. 
\newline

\noindent To prove Theorem B, we go back to $C^r$-families $(\mathcal{F}_p)_{p}$ of endomorphisms and we    
 restrict the dynamics to the local stable manifolds $W^{\mathfrak{a}}_p$, which are tagged in exponent by infinite forward sequences $\mathfrak{a}$  in letters in an  alphabet $\mathcal{A}$. Since $\mathrm{dim}  \text{ }      \mathcal{E}_p^s=1$, we are led to study families $(F_p)_{p}$ of $C^r$-skew-products  acting on fibers which are segments. \newline

\noindent  We then look at the action induced by $(F_p)_{p}$ on $s$-jets (with $s \le r-2$) and this gives a new family  of $C^2$-skew-products $(G_{p_0})_{p_0}$   acting on fibers of dimension $\delta_{d,s}$. This  new family satisfies assumption  ${\bf(U)}$ (the assumption $\mathrm{dim}  \text{ }      \mathcal{E}_p^s=1$ is used here),  its similarity dimension is larger than $\delta_{d,s}$ by $(\star)$.   
We extend  $(\mathcal{F}_p)_{p}$ into a larger $C^r$-family of endomorphisms so that the associated extended family  of  $C^2$-skew-products  acting on $s$-jets satisfies moreover the transversality assumption {\bf (T)}. \newline 
%

\noindent  To conclude, we pick a  family $(\Gamma_{t,p})_{t,p}$ of parallel segments $\Gamma_{t,p}$ close to a local stable manifold $W^{\mathfrak{a}}_p$ and s.t. the set of $s$-jets at any $p_0 \in \mathcal{P}$ 
 of the projection of $\Gamma_{t,p}$  on a fixed direction transversal to $\Gamma_{t,p}$ when varying $t$ has positive measure. 
The local unstable set intersects each segment $\Gamma_{t,p} $ in a set which is the limit set of a $\vartheta$-$\mathrm{U}$-perturbation with small $\vartheta$, at every parameter $p$. We then  apply the second part of 
 Theorem C to get positive sets of $s$-jets for the intersection points between the local unstable set and $\Gamma_{t,p}$ for  a.e. $t$  at a.e. $p_0$: in other terms we get positive sets of $s$-jets   in the direction of $\Gamma_{t,p}$ for these values of $t$ and $p_0$. To conclude, we apply the Fubini Theorem  to find  positive sets of bidimensional $s$-jets for points inside local unstable manifolds at a.e. $p_0$. The same extension scheme works for $(\mathcal{G}_p)_{p}$ close to $(\mathcal{F}_p)_{p}$, which proves that $(\mathcal{K}_p)_{p}$ is an  almost $C^{r,s}$-parablender. 
 \newline 
 
 \noindent Finally, Theorem A is an immediate consequence of Theorem B together with Remark \ref{retour}, by taking the constant family $(\mathcal{F})_{p \in (-1,1)}$ and the order $s$ of the jets equal to 0 (remark in particular that $\delta_{1,0}=1$). 
 \newline


\noindent 
In Section \ref{simple}, we study a model given by families of IFS of affine maps on an interval. We simplify the proof of Simon, Solomyak and Urbański \cite{b3} in this context and introduce the strategy of the proof of  Theorem C here. In Section \ref{difficile}, we prove  Theorem C. In each fiber, the behavior of the dynamics looks like the model. Finally, we prove
 Theorem B in Section \ref{jet}.

\section{Model: IFS of affine maps on the interval} \label{simple}

\subsection{Setting and results}

In this Section, we simplify the proof of a result of Simon, Solomyak and Urbański about IFS on an interval. This can be seen as a model for the behavior of the dynamics inside the fibers of a skew-product, as we will see in Section \ref{difficile}. \newline 

\noindent Let us fix  $X:=[-1,1]$  and $\mathcal{P} :=(-1,1)$. We consider families $(\Psi_p)_{p \in \mathcal{P}}$, where, for every $p \in \mathcal{P}$, the IFS $\Psi_p$ is a finite family $\Psi_p = (\psi_p^a)_{ a \in \mathcal A}$ of affine contractions $\psi_p^a : X \rightarrow X$ s.t. $\psi_p^a(X) \Subset X$. The absolute value of the  linear coefficient of $\psi_p^a$ is denoted by $\Lambda_{p,a}$.  We suppose that for every $a \in \mathcal{A}$, the map $\psi_p^a$ depends continuously on $p$.  In fact, we even suppose that the affine contraction  $\psi_p^a$ is still defined for $p$ in some open neighborhood of $[-1,1]$ and still depends continuously on $p$. \newline

\noindent For every $ p \in \overline{\mathcal{P}}$ and  $\alpha = (\alpha_{-n}, \cdots, \alpha_{-1})  \in  \mathcal A^*$, we denote by $\psi_p^\alpha = \psi_p^{\alpha_{-1}} \circ \cdots \circ \psi_p^{\alpha_{-n}}$ the corresponding composition, which is an affine contraction of the segment $X$ into its own interior. When $\alpha=e$, the map $\psi_p^e$ is simply the identity. 
The absolute value of the linear coefficient of $\psi_p^\alpha$ is denoted by  $\Lambda_{p,\alpha}$. \newline

\noindent
By continuity of the derivative of $\psi_p^a$  relatively to $p \in  \overline{\mathcal{P}}$, there exist $0< \gamma' < \gamma  <1$ s.t. $\gamma' < \Lambda_{p,a}<\gamma$ for every $p \in \overline{\mathcal{P}}$ and $a \in \mathcal A$. Then $\Lambda_{p,\alpha} \le \gamma^{|\alpha|}$ for every $\alpha   \in  \mathcal A^*$. If we now take $\alpha \in \overleftarrow{ \mathcal A}$, the  sequence of points $\psi_p^{\alpha_{| n}} (0)$  tends to a point $ \pi_p(\alpha)$. This defines for every $p \in \overline{\mathcal{P}}$ a $C^0$-map $\pi_p :  \overleftarrow{ \mathcal A} \rightarrow X$. Since this convergence is uniform in $p$, the map $p  \mapsto \pi_p(\alpha)$ is $C^2$ for any $\alpha$. Moreover the map $p \mapsto \pi_p$ is continuous,
 the set of $C^0$-maps from $\overleftarrow{ \mathcal A}$ to $\mathbb{R}$ being endowed with the uniform $C^0$-norm. We set:
$$K_p :=\pi_p(\overleftarrow{ \mathcal A}) \, . $$

\noindent
We also suppose that  the following assumption ${\bf(T_{aff})}$ is  satisfied by  $(\Psi_p)_{p \in \mathcal{P}}$. \newline 

\noindent
${\bf(T_{aff})}$: There exists $C>0$ s.t. for every $\alpha, \beta \in \overleftarrow{ \mathcal A}$ satisfying $\alpha_{-1} \neq \beta_{-1}$, we have:
$$ \mathrm
{Leb}_1 \{ p \in \mathcal{P}  : | \pi_p(\alpha) - \pi_p(\beta)  | < r    \} \le Cr \text{ for any } r >0  \, .$$ 

\noindent   
It is immediate that for any $p \in \overline{\mathcal{P}}$, there exists exactly one number $\Delta(p) \ge 0$ s.t.:
 $$\sum_{a \in \mathcal{A}} \Lambda_{p,a}^{\Delta(p)} = 1 \, .$$ This is the similarity dimension of the IFS $\Psi_p$. We are now in position to state the following result, which is a direct consequence of Theorem 3.1 of
  \cite{b3}:

\begin{theoremliloumm} \label{main} (Simon, Solomyak, Urbański)
Let $(\Psi_p)_{p \in \mathcal{P}}$ be a family of IFS of affine contractions satisfying ${\bf(T_{aff})}$ and $\Delta(p)>1$ for any $p \in \overline{\mathcal{P}}$. Then it holds:
 $$\mathrm{Leb}_{1} (K_p) >0 \text{  for }    \mathrm{Leb}_{1} \text{  a.e. } p \in \mathcal{P} \, .$$  
\end{theoremliloumm}

     \subsection{Proof of Theorem D} 
       
  \begin{proof}[Proof of Theorem D] 
  
  Here is the strategy. 
   Let $p_0 \in \mathcal{P} $ such that $\Delta (p_0) >1+\epsilon$ for a small $\epsilon >0$. To prove the result, it is enough to show that there exists $\delta>0$ s.t. $\mathcal{B}:=(p_0-\delta,p_0+\delta)$ is included in $\mathcal{P}$ and $\mathrm{Leb}_1(K_p)>0$ for $\mathrm{Leb}_1$ a.e. $p \in \mathcal{B}$.    \newline  
   
 \noindent We define a probability measure $\mu$ on $\overleftarrow{ \mathcal A}$ by setting $\mu[\alpha] = \Lambda_{p_0,\alpha}^{\Delta(p_0)}$ for every cylinder defined by  $\alpha \in  \mathcal A^*$.
   For any $p \in \mathcal{P}$, let $\nu_p$ be the pushforward of $\mu$ by $\pi_p$, which is supported on $K_p$. To conclude, it is enough to show that there exists $\delta>0$ s.t.  for $\mathrm{Leb}_1$ a.e. $p \in \mathcal{B}$, the measure $\nu_p$ is absolutely continuous relatively to $\mathrm{Leb}_1$. We set   $$D(\nu_p,x) := \liminf_{r \rightarrow 0} \frac{\nu_p( x-r,x+r) }{2 r} $$  for every    $p \in \mathcal{P}$ and $x \in \mathbb{R}$, which is the lower density of the measure $\nu_p$ at $x$.

    \begin{Lemma} \label{mesureble}
   The map $(p,x) \in \mathcal{P}  \times \mathbb{R} \mapsto  D(\nu_p,x)$ is Borel measurable. 
    \end{Lemma}
 
    \begin{proof}
    Since  $p \in \mathcal{P}  \mapsto \pi_p$ is continuous, it is also the case for $p \in \mathcal{P}   \mapsto \nu_p$ (the set of probability measures being endowed with the weak-$*$ topology).  
    Using this and since $x \in \mathbb{R} \mapsto \nu_p(x-r,x+r)$ is Borel measurable for every $p \in \mathcal{P}$ and $r >0$, the map $(p,x) \in \mathcal{P}  \times \mathbb{R} \mapsto \nu_p(x-r,x+r)$ is  Borel measurable for every $r >0$. Since $r \mapsto \nu_p( x-r,x+r)$ is monotonic and $r \mapsto 2r$ continuous, the lower limit $D(\nu_p,x)$ does not change if $r$ is restricted to positive rationals. Thus the measurability of $(p,x)  \in \mathcal{P}  \times \mathbb{R} \mapsto  D(\nu_p,x) $ reduces to that of the lower limit of countably many measurable maps.
    \end{proof}
    \noindent
    We prove below:
    
    \begin{Proposition}  \label{integr} 
    There exists $\delta>0$ s.t.  $\mathcal{B} \subset \mathcal{P}$ and the following  is finite:
   $$\mathcal{I} := \int_{p \in \mathcal{B}} \int_{x \in \mathbb{R}} D(\nu_p,x)d\nu_p d \mathrm{Leb}_1<+ \infty \, .$$
   \end{Proposition} 
   \noindent
This is enough to show that for $\mathrm{Leb}_1$ a.e. $p \in \mathcal{B}$,  $\nu_p$ is absolutely continuous relatively to $\mathrm{Leb}_1$. Indeed, then, for $\mathrm{Leb}_1$ a.e. $p \in \mathcal{B}$, we will have $D(\nu_p,x)<+\infty$ for $\nu_p$ a.e.  $x \in \mathbb{R}$ and we apply the following result from the third item of Lemma 2.12 in \cite{b7}.
   
  \begin{Proposition}
  Let $\nu$ be a Radon measure on $\mathbb{R}^n$, where $n>0$, s.t. the density $D(\nu,x)$ of $\nu$ relatively to $\mathrm{Leb}_n$ is finite for $\nu$  a.e.  $x \in \mathbb{R}^n$. Then $\nu$ is absolutely continuous relatively to $\mathrm{Leb}_n$.  
  \end{Proposition} 
  \noindent
  This concludes the proof of Theorem D. 
  \end{proof}
   
   \begin{proof}[Proof of Proposition \ref{integr}] For $\delta$ small enough, the interval $\mathcal{B}:=(p_0-\delta,p_0+\delta)$ is included in $\mathcal{P}$. If necessary, we reduce $\delta$ so that $ \Lambda_{p_1,a}^{1+\epsilon/2} \le  \Lambda_{p_2,a}$ for every $p_1 ,p_2  \in \mathcal{B}$ and $a \in\mathcal{A}$. In particular, this implies the following:
   \begin{equation} \label{distorlilou} 
   \forall  p_1 ,p_2  \in \mathcal{B}, \text{ } \forall  \alpha \in\mathcal{A}^*, \text{ }  \Lambda_{p_1, \alpha
   }^{1+\frac{\epsilon}{2}} \le  \Lambda_{p_2,\alpha}  \, .
   \end{equation} 
   \noindent
   The strategy is to bound $\mathcal{I}$ by a new integral which will be easily shown to be finite using ${\bf(T_{aff})}$, for this specific choice of $\delta$. First, by Fatou's lemma, it holds:
   \begin{equation} \label{eqint}
    \mathcal{I} \le \liminf_{r \rightarrow 0}  \frac{1}{2 r}  \int_{p \in \mathcal{B}} \int_{x \in \mathbb{R}} \nu_p(  x-r,x+r) d\nu_p d  \mathrm{Leb}_1 \, .
    \end{equation}
    We can write $\nu_p( x-r,x+r)= \int_{y \in \mathbb{R}} 1_{ \{|x-y| < r \} } d\nu_p$ as the integral of the indicator function $1_{ \{|x-y| < r \} }$, equal to 1 if  $|x-y| < r$, and 0 if not. Using this and then the definition of $\nu_p$ as the pushforward of $\mu$ by $\pi_p
    $, we have:
 \begin{equation} \label{eqint11}
 \int_{x \in \mathbb{R}} \nu_p( x-r,x+r) d\nu_p  = \int_{(\alpha, \beta) \in \overleftarrow{ \mathcal A} \times \overleftarrow{ \mathcal A}}  1_{ \{ | \pi_p(\alpha)-\pi_p(\beta) | < r \} }     d \mu \times \mu \, ,
  \end{equation}
 where $1_{ \{ | \pi_p(\alpha)-\pi_p(\beta) | < r \} } $ is equal to 1 if $ | \pi_p(\alpha)-\pi_p(\beta)  | < r $ and 0 if not. Then, we inject Eq. (\ref{eqint11}) into Eq. (\ref{eqint}) and use Fubini's Theorem to reverse the order of integration:
  \begin{equation} \label{eqint111}
  \mathcal{I} \le \liminf_{r \rightarrow 0}  \frac{1}{2 r}   \int_{(\alpha, \beta) \in \overleftarrow{ \mathcal A} \times \overleftarrow{ \mathcal A}}     \mathrm{Leb}_1 \{ p \in   \mathcal{B} :   |\pi_p(\alpha)-\pi_p(\beta)   | < r \}   d \mu \times \mu  \, .
\end{equation}
   We are going to write the latter integral as a sum whose terms are all easier to bound. For every finite word $\rho \in  \mathcal A^*$, we denote by $\mathcal C_\rho$ the set of pairs $(\alpha, \beta) \in  \overleftarrow{ \mathcal A} \times \overleftarrow{ \mathcal A}$ such that $\alpha_{-|\rho |} \cdots \alpha_{-1} = \beta_{-|\rho |}  \cdots \beta_{-1} = \rho$  but $\alpha_{-|\rho |-1} \neq \beta_{-|\rho |-1}$. We notice that $\overleftarrow{ \mathcal A} \times \overleftarrow{ \mathcal A} = \bigsqcup_{n \ge 0} \bigsqcup_{\rho \in \mathcal{A}^n}   \mathcal C_\rho  $ and so by Eq. (\ref{eqint111}) we have:
      \begin{equation} \label{eqint11111}
   \mathcal{I} \le  \liminf_{r \rightarrow 0}  \frac{1}{2r} \sum_{n\ge 0}  \sum_{\rho \in \mathcal{A}^n }    \int_{(\alpha, \beta) \in  \mathcal C_\rho}     \mathrm{Leb}_1  \{   p \in \mathcal{B} :  | \pi_p(\alpha)-\pi_p(\beta) | < r \}  d \mu \times \mu \, . 
    \end{equation} 
   We show below that a consequence of the transversality assumption ${\bf(T_{aff})}$ is:
  \begin{Lemma} \label{borne}   For  every $n \ge 0$, $\rho \in \mathcal{A}^n$ and $(\alpha, \beta) \in  \mathcal C_\rho $, we have:
  $$  \mathrm{Leb}_1   \{   p \in \mathcal{B} :   | \pi_p(\alpha)-\pi_p(\beta) | < r \}  \preceq   r \cdot  \Lambda_{p_0,\rho}^{ -1- \epsilon/2  }  \, .$$
  \end{Lemma}
  \noindent
  We can inject the bound of Lemma $\ref{borne}
  $ into Eq. (\ref{eqint11111}):
  \begin{equation}   \label{eqint111113}
  \mathcal{I}  \preceq \sum_{n\ge 0}  \sum_{\rho \in \mathcal{A}^n }    \int_{(\alpha, \beta) \in  \mathcal C_\rho}   \Lambda_{p_0,\rho}^{ - 1 - \epsilon/2  }   d \mu \times \mu \, .
  \end{equation} 
 \noindent
    We use the equality $\mu[\rho] = \Lambda_{p_0,\rho}^{\Delta(p_0)}$, the inequality $(1+\frac{\epsilon}{2})/\Delta(p_0) < (1+\frac{\epsilon}{2})/(1+\epsilon) < 1-\frac{\epsilon}{3} $ and finally the inequality  $\gamma' < \Lambda_{p,a}<\gamma$ to get:
     
 \begin{equation}   \label{bound}
   \Lambda_{p_0,\rho}^{ -1 - \epsilon/2  } \asymp  \mu[\rho]^{ \frac{-1-\epsilon/2}{\Delta(p_0)}}  \le  \mu[\rho]^{ -(1-\frac{\epsilon}{3} ) }  \preceq   \gamma^{ \frac{n\epsilon}{3}  } \cdot   \mu [\rho]^{ -1 } \, .
 \end{equation}
 We now inject this bound into Eq. $(\ref{eqint111113})$ to find:
 $$\mathcal{I} \preceq   \sum_{n\ge 0} \gamma^{ \frac{n\epsilon}{3}  }    \sum_{\rho \in \mathcal{A}^n }  \frac{( \mu  \times \mu) (\mathcal{C}_\rho ) }{ \mu [\rho] } \le    \sum_{n\ge 0} \gamma^{ \frac{n\epsilon}{3}  }    \sum_{\rho \in \mathcal{A}^n }  \mu [\rho] =  \sum_{n\ge 0} \gamma^{ \frac{n\epsilon}{3 }  }  < + \infty \, ,$$ where we used the inequality $( \mu  \times \mu) (\mathcal{C}_\rho )  \le \mu [\rho]^2$ (coming from $\mathcal{C}_\rho \subset [\rho]^2$) to prove the second inequality. This concludes the proof of Proposition \ref{integr}.
  \end{proof}

   \begin{proof}[Proof of Lemma \ref{borne}]    
 
   For every $p \in \mathcal{P}$,
    $n \ge 0$, $\rho \in \mathcal{A}^n$ and $(\alpha, \beta) \in  \mathcal C_\rho$, it holds:
  \begin{equation}   \label{tor}
 | \pi_p(\alpha) - \pi_p(\beta) |  = \Lambda_{p,\rho} 
 \cdot |  \pi_p(\sigma^n(\alpha)) - \pi_p(\sigma^n(\beta))|  \, .
  \end{equation}
Indeed, the points $\pi_p(\alpha)$ and  $\pi_p(\beta)$ are the respective images of $ \pi_p(\sigma^n(\alpha))$ and $\pi_p(\sigma^n(\beta))$ by the map $\psi_p^\rho $ which is an affine contraction, and the absolute value of the linear coefficient of $\psi_p^\rho $ is $\Lambda_{p,\rho}$. Thus, using Eq. $(\ref{distorlilou})$, it holds: 
   \begin{equation} 
   \mathrm{Leb}_1   \{   p \in \mathcal{B} :  | \pi_p(\alpha)-\pi_p(\beta) | < r \}  =      
        \mathrm{Leb}_1   \{   p \in \mathcal{B} : | \pi_p(\sigma^n (\alpha))-\pi_p(\sigma^n(\beta)) | < \frac{ r}{ \Lambda_{p,\rho}} \}  \nonumber
    \end{equation} 
    \begin{equation}  \label{10}
  \mathrm{Leb}_1  \{   p \in \mathcal{B} :  | \pi_p(\alpha)-\pi_p(\beta)  |  < r \}  \le     \mathrm{Leb}_1   \{   p \in \mathcal{B} :  | \pi_p(\sigma^n (\alpha))-\pi_p(\sigma^n (\beta)) | <  \frac{r}{  \Lambda_{p_0,\rho}^{ 1+\frac{\epsilon}{2
  }  }  } \} \nonumber
      \end{equation}
      To conclude, by ${\bf(T_{aff})}$ and since $\mathcal{B}  \subset \mathcal{P}$ the right-hand term of the latter  is smaller than
     $  C \cdot   r \cdot  \Lambda_{p_0,\rho}^{ -1- \epsilon/2  }$.
   \end{proof}
   \noindent
   {\bf Example.} Let us give a simple example of application of Theorem D. Let $n \ge 2$ be an integer. We set $\mathcal{A}:=\{0,1,2, \ldots, n\}$, $X:=[0,1]$ and $\mathcal{P}:=(1/n, 1-1/n)$. Let $c<1/n$ be a real number close to $1/n$. For $a \in \mathcal{A}$, we put $\psi^a_p(x):= c x + \frac{1}{2} (1/n-c) + a/n $ if $0 \le a < n$ and $\psi^{n}_p(x):=  c x + p$ if $a= n$. Condition   ${\bf(T_{aff})}$ is clearly satisfied when $n$ is large. Moreover trivial computations show that the similarity dimension is $\Delta(p)= -\mathrm{log}(n+1) / \mathrm{log}(c)>1$ for any $p \in  \overline{\mathcal{P}}$. By Theorem D,  $K_p$ has positive one-dimensional Lebesgue measure for a.e. $p \in  \mathcal{P}$.

   \section{The unipotent case: Proof of Theorem C }  \label{difficile}

We now extend Theorem D to the case of families of fiberwise unipotent skew-products.  The fibers are indexed by $\mathfrak{a} \in \overrightarrow{ \mathcal A}$ and the dynamics on each fiber will look like the one of the model previously introduced. Here are some differences:
 \begin{itemize}

 \item We will not restrain ourselves to fibers of dimension 1 and we will not suppose that the dynamics on each fiber is conformal but we will suppose that its differentials are unipotent with contracting eigenvalue (assumption ({\bf U})).
  \item We will need  distortion results (Lemmas \ref{bdsk}, \ref{gfsk}, \ref{fisk} and \ref{poisson}) since the dynamics will not supposed to be affine this time.
   \end{itemize}

   \subsection{Notations and immediate facts}  \label{sub}

\noindent  
 We adopt from now the following  formalism in order to prove Theorem C. Let $(F_p)_{p \in \mathcal{P}}$ be a family of  $C^r$-skew-products satisfying ${\bf(U)}$ and  ${\bf(T)}$. We recall that there exist open neighborhoods $X’$ and $\mathcal{P}’$ of $X$ and $\overline{\mathcal{P}}$ in $\mathbb{R}^N$ and $\mathbb{R}^d$ s.t. each map $(p,x) \mapsto  (p,f_{p,\mathfrak{a}}(x))$ extends to a diffeomorphism from $X’ \times \mathcal{P}’ $ into  $X \times \mathcal{P} $ and the map $F_p$ is a ($N$-dimensional) $C^r$-skew-product for every $p \in \mathcal{P}’$. We set:
\begin{equation} \label{fpsi}
\forall p \in \mathcal{P}’ , \mathfrak{a} \in \overrightarrow{ \mathcal A}  , a \in \mathcal{A}, x \in X’, \text{ } \psi_{ p,  \mathfrak{a} }^a(x)
:= f_{p, a \mathfrak{a}}(x)
\end{equation}
   and notice that $\psi_{ p,  \mathfrak{a} }^a : X’ \rightarrow X$ is a $C^2$-map depending continuously on $(p,  \mathfrak{a})$. 
 The $C^2$-norm of $\psi_{ p,  \mathfrak{a} }^a$ on $X$ is then bounded independently of $p \in  \overline{\mathcal{P}}$, $
    \mathfrak{a} \in \overrightarrow{ \mathcal A}$ and  $a \in \mathcal{A}$. 
     We now define for any $p \in \mathcal{P}’$, $\mathfrak{a} \in \overrightarrow{ \mathcal A} $, $n >0$ and $\alpha = (\alpha_{-n}, \cdots, \alpha_{-1}) \in \mathcal{A}^*$:
$$\forall x \in X’, \text{ } \psi_{ p,  \mathfrak{a} }^\alpha(x):=   \psi_{ p,  \mathfrak{a} }^{\alpha_{-1}}   \circ \cdots \circ       \psi_{ p,       \alpha_{-n+1} \cdots \alpha_{-1}   \mathfrak{a}}^ {\alpha_{-n}}        (x) = f_{ p,  \alpha_{-1}  \mathfrak{a}   }   \circ \cdots \circ       f_{ p,    \alpha_{-n} \cdots \alpha_{-1}     \mathfrak{a}    } (x) \, . $$

\noindent
In particular, assumption {\bf (U)} implies that  for every $p \in \mathcal{P}’$, $\mathfrak{a} \in \overrightarrow{ \mathcal A} $, $\alpha \in \mathcal{A}^*$    and $x \in X’$, the differential $D\psi_{ p,  \mathfrak{a} }^\alpha(x)$ is unipotent inferior and thus has a unique eigenvalue.    This motivates the definition of the following contraction rate:


\begin{Definition}
For any $p \in \mathcal{P}’$, $\mathfrak{a} \in \overrightarrow{ \mathcal A} $, $\alpha \in \mathcal{A}^*$ and 
$x \in X’$,  let  $\lambda_{p,\mathfrak{a},\alpha}(x)$ be  the absolute value  of the unique eigenvalue of the differential $D\psi_{ p,  \mathfrak{a} }^\alpha(x)$ and: 

\begin{equation} \Lambda_{p,\mathfrak{a},\alpha}:=\mathrm{max}_{x \in X} \lambda_{ p,\mathfrak{a},\alpha}(x) \, .
\end{equation} 
\end{Definition}

\noindent 
For any  $p \in \mathcal{P}’$, $\mathfrak{a} \in \overrightarrow{ \mathcal A} $, $\alpha=(\alpha_{-n}, \ldots, \alpha_{-1}) \in \mathcal{A}^*$ and $x \in X’$, we then have:
\begin{equation} \label{produitchinne}
   \lambda_{p,\mathfrak{a},\alpha}(x)  = \prod_{k=1}^{n}       \lambda_{p, \mathfrak{a}_k, \alpha_{-k}} \big( \psi_{p, \mathfrak{a}_{k+1} }^{\alpha_{-k-1}} \circ \cdots \circ \psi_{p, \mathfrak{a}_n }^{\alpha_{-n}}(x) \big)  \text{ with }  \mathfrak{a}_k:=\alpha_{|  k-1} \mathfrak{a}   \, .
   \end{equation}

\noindent
By continuity of $D\psi_{ p,  \mathfrak{a} }^a(x)$  relatively to $p \in  \overline{\mathcal{P}}$, $ \mathfrak{a} \in   \overrightarrow{ \mathcal A}$ and   $x \in X$ and by compactness of $\overline{\mathcal{P}}$, $ \overrightarrow{ \mathcal A}$ and  $X$, there exist $0< \gamma' < \gamma  <1$ s.t. for any $p$,  $\mathfrak{a}$, $a$, it holds:
\begin{equation}
\forall x \in X, \text{ } \gamma' <  \lambda_{p,\mathfrak{a},a}(x)   < \gamma \, . 
\end{equation} 
\noindent In particular for every $\alpha \in \mathcal{A}^*$ it holds:
\begin{equation} \label{ineqpres}
\gamma’^{| \alpha|} <       \Lambda_{p,\mathfrak{a},\alpha}   < \gamma^{| \alpha|} \, . 
\end{equation} 
\noindent

\noindent


\noindent
We will need later the following, whose proof is in the Appendix:

\begin{Lemma} \label{polynome}
There exists a real polynomial $P$ positive on $\mathbb{R}_+$ s.t. for any $p \in  \overline{\mathcal{P}} $, $\mathfrak{a} \in  \overrightarrow{ \mathcal A}$,  $\alpha \in \mathcal{A}^*$, $x \in X$ and $(i,j) \in \{1, \cdots, N\}^2$ with $i>j$, the modulus  of the $(i,j)^{th}$ coefficient of the differential $D\psi_{p,\mathfrak{a}}^\alpha(x)$ is smaller than  $P(| \alpha| ) \cdot \lambda_{p,\mathfrak{a},\alpha}(x)$. 
\end{Lemma}

  \noindent
  In particular, for any $p \in \overline{\mathcal{P}}$, $\mathfrak{a} \in \overrightarrow{ \mathcal A} $ and $\alpha \in \mathcal{A}^*$ of length sufficiently large, the map $ \psi_{p,\mathfrak{a}}^{\alpha}$ is a contraction.  For $p \in \overline{\mathcal{P}}$, $\mathfrak{a} \in \overrightarrow{ \mathcal A} $ and  $\alpha \in  \overleftarrow{ \mathcal A}$,  the diameter of $ \psi_{p,\mathfrak{a}}^{\alpha_{| n}} (X)$  is then small when $n$ is large. Thus the sequence of points  $\psi_{p,\mathfrak{a}}^{\alpha_{| n}} (0)$ converges to $ \pi_{p, \mathfrak{a}} (\alpha) \in X$. This defines for every $p \in  \overline{\mathcal{P}}$ and $\mathfrak{a} \in \overrightarrow{ \mathcal A} $ 
 a $C^0$-map $\pi_{p,\mathfrak{a}} : \overleftarrow{ \mathcal A}  \rightarrow X$. The map $(p, \mathfrak{a}) \in  \overline{\mathcal{P}} \times  \overrightarrow{ \mathcal A}  \mapsto \pi_{p,\mathfrak{a}}$ is then continuous, the set of $C^0$-maps from $\overleftarrow{ \mathcal A}$ to $\mathbb{R}^N$ being endowed with the uniform $C^0$-norm. We set: $$
 K_{p,\mathfrak{a}} :=\pi_{p,\mathfrak{a}}( \overleftarrow{ \mathcal A}) \, .$$

\noindent For any  family $\mathbb{F} := ((\tilde{F}_{t,p})_p)_{t \in \mathcal{T}}$ of $\vartheta$-$\mathrm{U}$-perturbations with small $\vartheta>0$, the map $\psi_{t, p,  \mathfrak{a} }^\alpha:= \tilde{f}_{ t, p,  \alpha_{-1}  \mathfrak{a}   }   \circ \cdots \circ       \tilde{f}_{t,  p,    \alpha_{-n} \cdots \alpha_{-1}     \mathfrak{a}    }$ is also a contraction when $| \alpha|$ is large and so the sequence of points  $\psi_{t,p,\mathfrak{a}}^{\alpha_{| n}} (0)$ still converges to $ \tilde{\pi}_{t, p, \mathfrak{a}} (\alpha) \in X$. This allows us  to define a $C^0$-map $\tilde{\pi}_{t,p,\mathfrak{a}} : \overleftarrow{ \mathcal A}  \rightarrow X$ for any $t$ and  $p$, and its limit set $\tilde{K}_{t,p,\mathfrak{a}} :=\tilde{\pi}_{t,p,\mathfrak{a}}( \overleftarrow{ \mathcal A})$.

\begin{Lemma} \label{ljungq}
The map $p \mapsto \pi_{p,\mathfrak{a}}(\alpha)$ is $C^r$ for every $\mathfrak{a} \in  \overrightarrow{ \mathcal A}$ and $\alpha \in  \overleftarrow{ \mathcal A}$. Moreover, for every    family $\mathbb{F}$ of $\vartheta$-$\mathrm{U}$-perturbations with small $\vartheta>0$, the map $(t,p) \mapsto \tilde{\pi}_{t,p,\mathfrak{a}}(\alpha)$ is $C^r$ and $p \mapsto  \tilde{\pi}_{t,p,\mathfrak{a}}(\alpha)$ is $C^r$-close to $p \mapsto \tilde{\pi}_{p,\mathfrak{a}}(\alpha)$ uniformly in $t \in \mathcal{T}$.
\end{Lemma}

\begin{proof}
These maps are the respective  uniform limits of the following $C^r$-maps 
$$p \mapsto  \psi_{p,\mathfrak{a}}^{\alpha_{| n}} (0)  \text{ and } (t,p) \mapsto  \psi_{t,p,\mathfrak{a}}^{\alpha_{| n}} (0) \, .$$
\noindent 
 To conclude, we  remark that these convergences are both exponential and that the   maps $f_{p,\mathfrak{a}}$ and $\tilde{f}_{t,p,\mathfrak{a}}$     are (uniformly in $t$  and $\mathfrak{a}$) $C^r$-close when $\vartheta$ is small. 
\end{proof}

\noindent As an immediate consequence of Lemma \ref{polynome}, the pressure function $\Pi_{p,\mathfrak{a}}$ defined in Definition \ref{pressiondef} is equal to: 
\begin{equation} \label{pii}
 \Pi_{p,\mathfrak{a}}(s)=\lim_{n \rightarrow + \infty} \frac{1}{n} \mathrm{log} \sum_{\alpha \in \mathcal A^n}  \Lambda_{p,\mathfrak{a}
 ,\alpha}^s \text{ for any }  s \ge 0   \, .  
 \end{equation} 
Moreover this map satisfies the following nice properties:

 \begin{Proposition}  \label{pression}
The map $s \in \mathbb{R}_+ \mapsto  \Pi_{p,\mathfrak{a}}(s) \in \mathbb{R}$ is well-defined, strictly decreasing, continuous, independent of $\mathfrak{a}$, $\Pi_{p,\mathfrak{a}}(0)  >0$  and $ \mathrm{lim}_{ s \rightarrow + \infty} \Pi_{p,\mathfrak{a}}(s) = -\infty$. In particular, it has exactly one zero denoted by $\Delta (p)$, depending continuously on $p$. 
\end{Proposition}

   \noindent From now on we suppose that $\Delta(p)>N$ for any $p \in  \overline{\mathcal{P}}$.

      \subsection{Distortion lemmas}
   
We now state distortion results, whose proofs are given in the Appendix:
 
 \begin{Lemma} \label{bdsk} (Bounded distortion w.r.t. $x$) There exists $D_1>1$ s.t.  for every $p \in  \overline{\mathcal{P}}$,  $\mathfrak{a} \in  \overrightarrow{ \mathcal A}$, $\alpha \in  \mathcal A^*$ and $
 x,y \in X$, it holds:
$$1/D_1 < \frac{   \lambda_{p,\mathfrak{a},\alpha}(x) }{\lambda_{p,\mathfrak{a},\alpha}  (y) } < 
D_1 \, .$$
 \end{Lemma}


 \begin{Lemma} \label{gfsk} (Distortion w.r.t. $p$) For every $\eta>0$, there exists $\delta(\eta)>0$ and $D_2 =D_2(\eta)>1$ s.t.  for every $p_1, p_2 \in \overline{\mathcal{P}}$ and $\mathfrak{a} \in  \overrightarrow{ \mathcal A}$, it holds the following:
$$  || p_1-p_2 || \le \delta(\eta) \implies \forall \alpha \in \mathcal A^*, \text{ }  D_2^{-1}   e^{-\eta |\alpha| } < \frac{ \Lambda_{p_1,\mathfrak{a},\alpha} }{\Lambda_{p_2,\mathfrak{a},\alpha}} < D_2  e^{\eta |\alpha| }\, .$$
\end{Lemma}

\begin{Lemma} \label{fisk} (Bounded distortion w.r.t. $\mathfrak{a}$)
There exists $D_3>1$ s.t. for any  $p \in \overline{\mathcal{P}}$, $\mathfrak{a}, \mathfrak{a}' \in  \overrightarrow{ \mathcal A}$  and  $\alpha \in \mathcal A^*$, it holds:  
\begin{equation}
 1/D_3 <   \frac{ \Lambda_{p,\mathfrak{a},\alpha} }{ \Lambda_{p,\mathfrak{a}',\alpha}}<  D_3 \, . \nonumber
 \end{equation} 
 \end{Lemma} 

 \begin{Lemma}  (Distortion w.r.t. $\vartheta$-perturbations) \label{poisson}
 For every  $\epsilon’>1$, there exists  $D_4>1$ s.t. for  every family $\mathbb{F}$ of $\vartheta$-$\mathrm{U}$-perturbations with $\vartheta$  small enough, we have for every $t \in \mathcal{T}$, $\mathfrak{a} \in   \overrightarrow{ \mathcal A}$, $p \in \overline{\mathcal{P}}$ and $\alpha \in \mathcal{A}^*$:  $$ \Lambda_{p,\mathfrak{a},\alpha}^{\epsilon’}/D_4 <    \tilde{\Lambda}_{t,p,\mathfrak{a},\alpha}     <   D_4  \Lambda_{p,\mathfrak{a},\alpha}^{1/\epsilon’}$$
  where $ \tilde{\Lambda}_{t,p,\mathfrak{a},\alpha}  $ is the maximum  among $x \in X$  of the absolute value $ \tilde{ \lambda}_{t,p,\mathfrak{a},\alpha}(x)  $ of the unique eigenvalue  of the differential $D \psi^\alpha_{t,p,\mathfrak{a}}(x)$. 
 \end{Lemma} 


    \subsection{Choice of a probability measure $\mu$}  \label{mesurebuild}

We will first need the following result of Bowen (\cite{b5} thm 1.4 P7 and its proof P19) about the existence of a Gibbs measure. We recall that $\sigma$ is a  full shift. We state the result in this case but it remains true for subshifts of finite type topologically mixing. We suppose that a parameter $p_0$ has been fixed (the precise choice will be made in the next subsection). We fix an arbitrary $\mathfrak{a}_0  \in \overrightarrow{ \mathcal{A} }$.  

    \begin{theoremB} \label{gibbs} 
   Let $\phi: \overleftrightarrow{ \mathcal{A} } \rightarrow \mathbb{R}$ be a H\"older map with positive exponent (the set $\overleftrightarrow{ \mathcal{A}}$ being endowed with the distance $d_\infty$). Then there exists a unique $\sigma$-invariant measure $\mu$ on $ \overleftrightarrow{ \mathcal{A}}$ s.t. for every $A \in  \overleftrightarrow{ \mathcal{A}}$, it holds:    
   $$   \mu [A_{|  n}]   \asymp  \mathrm{exp} \big(     -    \Pi  n  + \sum_{k=0}^{n-1} \phi (\sigma^k (A)) \big) $$
\noindent where $\Pi = \Pi(\phi) = \mathrm{lim}_{ n \rightarrow \infty} \frac{1}{n} \mathrm{log} Z_n(\phi)$, with
  $$Z_n(\phi):= \sum_{x  \in \mathcal{A}^n } \mathrm{exp} (S_x) \text{  and }  S_x= \mathrm{sup} \{  \sum_{k=0}^{n-1} \phi (\sigma^k (y))  : y \in [ x ]  \} \, .$$
 \end{theoremB}
   
\noindent
Writing $A = \alpha \mathfrak{a} \in \overleftrightarrow{ \mathcal{A}} $ as the concatenation of $\alpha  \in \overleftarrow{ \mathcal{A}} $ and $\mathfrak{a}  \in \overrightarrow{ \mathcal{A}} $, we can apply the previous result with the map 
\begin{equation} 
\phi : A  \in  \overleftrightarrow{ \mathcal{A}}       \mapsto  \Delta(p_0) \cdot   \mathrm{log} \lambda_{p_0, \mathfrak{a}, \alpha_{-1}}  \big( \pi_{p_0, \alpha_{-1} \mathfrak{a}   }(\sigma (\alpha))  \big)  \,  . \nonumber   
\end{equation}

 \noindent We show in the Appendix:
 \begin{Lemma}   \label{hol}
 The map $\phi$ is H\"older with positive exponent.
 \end{Lemma}  
 \noindent 
Using Lemma \ref{bdsk} and \ref{fisk}, we note that $\Pi = \Pi(\phi)$ coincides with $ \Pi_{p_0}(\Delta(p_0))$ and thus  vanishes by definition of $\Delta(p_0) $ (see Proposition \ref{pression}). Moreover,  by Eq. $(\ref{produitchinne})$,  
for any $A \in \overleftrightarrow{ \mathcal{A}} $ the sum $\sum_{k=0}^{n-1} \phi (\sigma^k(A))$ is equal to:
\begin{equation} \Delta(p_0) \cdot  \mathrm{log} \prod_{k=1}^n  \lambda_{p_0, \mathfrak{a}_k, \alpha_{-k}}  \big( \pi_{p_0, \mathfrak{a}_{k+1} }(\sigma^k (\alpha) ) \big)  = \Delta(p_0) \cdot     \mathrm{log} \lambda_{ p_0,  \mathfrak{a}, \alpha_{|  n}} \big( \pi_{p_0,  \alpha_{|  n} \mathfrak{a}  }(\sigma^n (\alpha) ) \big)  
\nonumber   
\end{equation}   \noindent with $ \mathfrak{a}_k:=  \alpha_{|  k-1} \mathfrak{a}$.  By Theorem E, this gives us a $\sigma$-invariant measure $\mu$ on $ \overleftrightarrow{ \mathcal{A}}    $ such that for every $A = \alpha \mathfrak{a} 
 \in  \overleftrightarrow{ \mathcal{A}}    $, we have:
$$\mu [A_{|  n}] \asymp               \lambda_{ p_0,  \mathfrak{a}, \alpha_{|  n}}^{\Delta(p_0)} \big( \pi_{p_0, \alpha_{|  n} \mathfrak{a}   }(\sigma^n( \alpha))  \big) \text{ when } n \rightarrow + \infty \, .$$ 
Using successively Lemmas \ref{bdsk} and \ref{fisk}, for any $A = \alpha \mathfrak{a} \in  \overleftrightarrow{ \mathcal{A}}  $, it holds:
 \begin{equation}   \label{helppppp}
 \mu [A_{|  n}]   \asymp   \Lambda_{p_0,\mathfrak{a} , \alpha_{|  n} }  ^{\Delta(p_0)}  \asymp   \Lambda_{p_0, \mathfrak{a}_0   , \alpha_{|  n} }  ^{\Delta(p_0)} \text{ when } n \rightarrow + \infty \, .
   \end{equation}   
  We then  define a $\sigma$-invariant probability measure on $\overleftarrow{ \mathcal{A}}$, still  denoted $\mu$,  by giving  to each cylinder in $\overleftarrow{ \mathcal{A}}  $ the same measure than the corresponding one in $\overleftrightarrow{ \mathcal{A}}$. Then:

\begin{equation}   \label{help}
 \mu [\rho]   \asymp    \Lambda_{p_0,  \mathfrak{a}_0   , \rho }^{\Delta(p_0)} \text{ when }  \rho \in \mathcal{A}^n      \text{ and  }  n \rightarrow + \infty \, .
   \end{equation}  

\begin{Remark} 
We will not need the $\sigma$-invariance property of $\mu$ in the following but only the estimation from Eq. $(\ref{help})$. 
\end{Remark}

   \subsection{Proof of Theorem C }

The strategy is the same as for the proof of Theorem D. Let us consider  $p_0 \in \mathcal{P} $ and $\mathfrak{a} \in   \overrightarrow{ \mathcal A}$.  We have $\Delta(p_0) >N+\epsilon$, where $\epsilon:= \frac{1}{2}( \min_{p \in \overline{\mathcal{P}} }\Delta(p)  -N) >0$. To prove the result, we show that there exists $\delta>0$ s.t.  the $d$-dimensional ball  $\mathcal{B}$ of center $p_0$ of radius $\delta$ is included in $\mathcal{P}$ with  $\mathrm{Leb}_{N}(K_{p,\mathfrak{a}}) >0$ and $\mathrm{Leb}_{N}( \tilde{K}_{t,p,\mathfrak{a}}) >0$ for $\mathrm{Leb}_d$ a.e.  $p \in \mathcal{B}$ and  $\mathrm{Leb}_{\tau} \text{  a.e. } t \in \mathcal{T}$, for every family $\mathbb{F}$ of $\vartheta$-$\mathrm{U}$-perturbations with small $\vartheta$. 
  \newline

\noindent We endow $\overleftarrow{ \mathcal A}$ with the probability measure $\mu$ defined in Subsection \ref{mesurebuild}. For any $p \in \mathcal{P}$ and $ t \in \mathcal{T}$, let $\nu_{p,\mathfrak{a}}$ and $\nu_{t,p,\mathfrak{a}}$ be the images of $\mu$ by the maps $\pi_{p,\mathfrak{a}}$ and  $\tilde{\pi}_{t,p,\mathfrak{a}}$.  \newline

   \noindent 
  As Proposition \ref{integr} implies Theorem D, Theorem C is a   consequence of:

  \begin{Proposition}  \label{lilou}
There exists $\delta>0$ s.t.  the  ball $\mathcal{B}$ of center $p_0$ and radius $\delta$ is included in $\mathcal{P}$ and the two following integrals are finite:
   $$ \mathcal{I} :=    \int_{p \in \mathcal{B}}   \int_{x \in \mathbb{R}^{N}} \liminf_{r \rightarrow 0}  \frac{\nu_{p,\mathfrak{a}}( x + B(r)) } {c_{N} r^{N}}    d\nu_{p,\mathfrak{a}}   d  \mathrm{Leb}_d<+ \infty  \, ,$$
    $$ \mathcal{I}’ :=     \int_{p \in \mathcal{B}}  \int_{t \in \mathcal{T}}     \int_{x \in \mathbb{R}^{N}} \liminf_{r \rightarrow 0}  \frac{\nu_{t,p,\mathfrak{a}}( x + B(r)) } {c_{N} r^{N}}    d\nu_{t,p,\mathfrak{a}}  d  \mathrm{Leb}_\tau   d  \mathrm{Leb}_d   
        <+ \infty  \, ,$$
 for any    family $\mathbb{F}$ of $\vartheta$-$\mathrm{U}$-perturbations with small $\vartheta$, where $B(r) \subset \mathbb{R}^{N}$ is the ball of center 0 and radius $r$ and the constant $c_{N}$ is defined by $c_{N} r^{N} := \mathrm{Leb}_{N}(B(r))$. 
 \end{Proposition} 
   
\begin{proof}
 Let us take a small $\delta$ s.t. $\mathcal{B} \subset \mathcal{P}$. The radius $\delta$ will be reduced one time so   that  $\mathcal{I}$ and $\mathcal{I}’$ are  finite.  We first begin by bounding $\mathcal{I}$.
 The proof is similar to the one of Proposition \ref{integr}: we begin by using  Fatou's Lemma, 
the definition of $\nu_{p,\mathfrak{a}}$ and the Fubini-Tonelli Theorem to find the following bound:  
$$\mathcal{I} \le \liminf_{r \rightarrow 0}                \frac{1}{c_{N} r^{N}   }  \int_{(\alpha,\beta)  \in \overleftarrow{ \mathcal A}   \times   \overleftarrow{ \mathcal A}     }   \mathrm{Leb}_d  \{ p \in   \mathcal{B} :   ||       \pi_{p,\mathfrak{a}}(\alpha) - \pi_{p,\mathfrak{a}}(\beta)    || < r \}     d \mu \times \mu  
  \, .$$
\noindent
We write $\overleftarrow{ \mathcal A}   \times \overleftarrow{ \mathcal A} 
 = \bigsqcup_{n \ge 0} \bigsqcup_{\rho \in \mathcal{A}^n}   \mathcal C_\rho  $ where $\mathcal C_\rho$ is the set of pairs $(\alpha, \beta) \in  \overleftarrow{ \mathcal A} \times \overleftarrow{ \mathcal A}$ s.t. $\alpha_{-|\rho |} \cdots \alpha_{-1} = \beta_{-|\rho |}  \cdots \beta_{-1} = \rho$  but $\alpha_{-|\rho |-1} \neq \beta_{-|\rho |-1}$. Thus 
 $\mathcal{I}$ is smaller than:
      \begin{equation} \label{eqint11111'} 
   \liminf_{r \rightarrow 0}   \frac{1}{c_{N} r^{N}   }           \sum_{n\ge 0}  \sum_{\rho \in \mathcal{A}^n }    \int_{   (\alpha,\beta)  \in   \mathcal C_\rho     } \mathrm{Leb}_d  \{ p \in   \mathcal{B} :   ||       \pi_{p,\mathfrak{a}}(\alpha) - \pi_{p,\mathfrak{a}}(\beta)    || <
    r \}       d \mu \times \mu     \, . 
    \end{equation} 
 We show below that a  consequence of the transversality assumption ${\bf (T)}$ is:
 \begin{Lemma} \label{bornelilou}  We  fix  $\eta:=\frac{-\epsilon \mathrm{log} \gamma}{2N+\epsilon}$ and reduce $\delta$ if necessary s.t.  $\delta < \delta(\eta)$ (where $\delta(\eta)$ is defined in Lemma $\ref{gfsk}$).  Then
 for any $n \ge 0$, $\rho \in \mathcal{A}^n$, $(\alpha, \beta) \in  \mathcal C_\rho$, we have:
  $$  \mathrm{Leb}_d   \{   p \in \mathcal{B} :   ||       \pi_{p,\mathfrak{a}}(\alpha) - \pi_{p,\mathfrak{a}}(\beta)           || < r \} \preceq  r^N   \cdot   \Lambda_{p_0, \mathfrak{a}_0,\rho}^{ -N- 2\epsilon/3  }  \, .$$
  \noindent For any family $\mathbb{F}$ of $\vartheta$-$\mathrm{U}$-perturbations of $(F_p)_p$ with small $\vartheta$, for any $t \in \mathcal{T}$, we have:
 $$  \mathrm{Leb}_d   \{   p \in \mathcal{B} :   ||       \tilde{\pi}_{t,p,\mathfrak{a}}(\alpha) - \tilde{\pi}_{t,p,\mathfrak{a}}(\beta)           || < r \} \preceq  r^N   \cdot   \Lambda_{p_0, \mathfrak{a}_0,\rho}^{ -N- 2\epsilon/3  }  \, .$$
  \end{Lemma}
  \noindent Notice that when the similarity dimension is close to the dimension $N$ of the fibers (and so $\epsilon$ small), we need to work with a ball  of small radius     $\eta$. 
  We can inject the first bound of Lemma $\ref{bornelilou}
  $  into Eq. (\ref{eqint11111'}):
  \begin{equation}   \label{eqint111113'}
  \mathcal{I}  \preceq \sum_{n\ge 0}  \sum_{\rho \in \mathcal{A}^n }    \int_{(\alpha, \beta) \in  \mathcal C_\rho} \Lambda_{p_0, \mathfrak{a}_0,\rho}^{  -N - 2\epsilon/3  }   d \mu \times \mu   \, .
  \end{equation} 
 \noindent
    We use successively Eq. (\ref{help}), the inequality $( N+\frac{2\epsilon}{3})/\Delta(p_0) < ( N+\frac{2\epsilon}{3})/(N+\epsilon) < 1-\frac{\epsilon}{4 N} $ and finally Eq. $(\ref{ineqpres})$ (which gives $\mu[\rho]  \preceq  \gamma^n$) to get:
      \begin{equation}   \label{bound}
   \Lambda_{p_0,\mathfrak{a}_0 ,\rho}^{ - N-2\epsilon/3  } \asymp  \mu[\rho]^{ \frac{- N-2\epsilon/3}{\Delta(p_0)}}  \le  \mu[\rho]^{ -(1-\frac{\epsilon}{4N} ) }  \preceq   \gamma^{ \frac{n\epsilon}{4 N}  } \cdot   \mu [\rho]^{ -1 } \, .
 \end{equation}
 We now inject this bound into Eq. $(\ref{eqint111113'})$ to find:
 $$\mathcal{I} \preceq   \sum_{n\ge 0} \gamma^{ \frac{n\epsilon}{4 N}  }    \sum_{\rho \in \mathcal{A}^n }  \frac{( \mu  \times \mu) (\mathcal{C}_\rho ) }{ \mu [\rho] } \le    \sum_{n\ge 0} \gamma^{ \frac{n\epsilon}{4 N}  }    \sum_{\rho \in \mathcal{A}^n }  \mu [\rho] =  \sum_{n\ge 0} \gamma^{ \frac{n\epsilon}{4 N }  }  < + \infty \, ,$$ where we used the inequality $( \mu  \times \mu) (\mathcal{C}_\rho )  \le \mu [\rho]^2$ (coming from  $\mathcal{C}_\rho \subset [\rho]^2$) to prove the second inequality. 
 To bound $\mathcal{I}’$ for every family $\mathbb{F}$ of $\vartheta$-$\mathrm{U}$-perturbations of $(F_p)_p$ with small $\vartheta$, we just remark that the same proof works when $\vartheta$ is small enough, with an additional integration relatively to $t \in \mathcal{T}$.
  \end{proof}

        \begin{proof}[Proof of Lemma \ref{bornelilou}] Let us begin with the following distortion lemma:

   \begin{Lemma} \label{lemauxlilou}
   There exists $D_5>0$ s.t. 
   for every   $p_1 ,p_2  \in \overline{\mathcal{P}}$, $\mathfrak{a}_1 ,\mathfrak{a}_2  \in  \overrightarrow{ \mathcal A}$ and   $\rho \in\mathcal{A}^*$, the following holds true:
   $$    || p_1 - p_2 || < \delta \implies  \Lambda_{p_1,\mathfrak{a}_{1},  \rho}^{1+\frac{\epsilon}{2 N}} \le  D_5 \cdot \Lambda_{p_2, \mathfrak{a}_2,\rho}  \, .  $$
   \end{Lemma}

   \begin{proof}If $    || p_1 - p_2 || < \delta $, then $    || p_1 - p_2 || < \delta (\eta)$.
   By Lemma \ref{gfsk}, it holds:
   $$  \Lambda_{p_1, \mathfrak{a}_{1  }, \rho}^{1+\epsilon_0} \le D_2^{1+\epsilon_0}   e^{|\rho| \eta  ( 1 +\epsilon_0)} \Lambda_{p_2, \mathfrak{a}_{1}, \rho}^{  1 +\epsilon_0  } \le   D_2^{1+\epsilon_0}  e^{|\rho| \eta  (1 +\epsilon_0)} \gamma^{|\rho| \epsilon_0}  \Lambda_{p_2, \mathfrak{a}_{1}, \rho} \, ,$$ with  $\epsilon_0 := \epsilon / (2N)$. 
 By Lemma \ref{fisk}, $ \Lambda_{p_2, \mathfrak{a}_{1},\rho} \asymp  \Lambda_{p_2, \mathfrak{a}_2,\rho}  $  when $|\rho| \rightarrow + \infty$, with bounds independent of $p_2$. The result follows since, by definition of $\eta$, it holds: 

 \begin{equation}
   e^{|\rho| \eta  ( 1+\epsilon_0)} \gamma^{|\rho| \epsilon_0} = e^{|\rho|  (\eta ( 1 +\epsilon_0 )+\epsilon_0 \mathrm{log} \gamma            )  } \text{ with } \eta ( 1 +\epsilon_0 )+\epsilon_0 \mathrm{log} \gamma          = 0 \, .
   \nonumber \end{equation} 
   \end{proof}

      \begin{Lemma} \label{xy} There exists a real polynomial $R$ positive on $\mathbb{R}_+$ s.t.  for every $p \in \overline{\mathcal{P}}$, $\mathfrak{a} \in  \overrightarrow{ \mathcal A}$, $n \ge 0$, $\rho \in \mathcal{A}^n$ and $(\alpha, \beta) \in  \mathcal C_\rho$, it holds:
  \begin{equation}   \label{tor}
 || \pi_{p,\mathfrak{a}} (\alpha) - \pi_{p,\mathfrak{a}} (\beta) ||  \ge \frac{\Lambda_{p,\mathfrak{a},\rho}}{R(n)} \cdot ||  \pi_{p,\rho  \mathfrak{a}} (\sigma^n(\alpha)) - \pi_{p,\rho \mathfrak{a}}(\sigma^n(\beta)) | | \, .
  \end{equation}
  Moreover for any  $\epsilon’>1$,    for  every family $\mathbb{F}$ of $\vartheta$-$\mathrm{U}$-perturbations of $(F_p)_p$  with $\vartheta>0$ small enough we have for every $t \in \mathcal{T}$, $p \in \overline{\mathcal{P}}$, $\mathfrak{a} \in  \overrightarrow{ \mathcal A}$, $n \ge 0$, $\rho \in \mathcal{A}^n$ and $(\alpha, \beta) \in  \mathcal C_\rho$:
 \begin{equation}   \label{torp}
 || \tilde{\pi}_{t,p,\mathfrak{a}} (\alpha) - \tilde{\pi}_{t,p,\mathfrak{a}} (\beta) ||  \ge \frac{\Lambda^{\epsilon’}_{p,\mathfrak{a},\rho}}{R(n)} \cdot ||  \tilde{\pi}_{t,p,\rho  \mathfrak{a}} (\sigma^n(\alpha)) - \tilde{\pi}_{t,p,\rho \mathfrak{a}}(\sigma^n(\beta)) | | \, .
  \end{equation}

   \end{Lemma} 
   \noindent
 The proof is  in the Appendix. 
    Given  $\mathfrak{a} \in  \overrightarrow{ \mathcal A}$, $n \ge 0$, $\rho \in \mathcal{A}^n$, $(\alpha, \beta) \in  \mathcal C_\rho$, using Eq. (\ref{tor}) and then Lemma \ref{lemauxlilou} together with the fact that $\mathcal{B}$ is the ball of center $p_0$ and radius $\delta$, it holds:

    \footnotesize
   \begin{equation} 
   \mathrm{Leb}_d   \{   p \in \mathcal{B}:||  \pi_{p,\mathfrak{a}} (\alpha) - \pi_{p,\mathfrak{a}} (\beta) || < r \}  \le     
        \mathrm{Leb}_d   \{   p \in \mathcal{B} : || \pi_{p, \rho \mathfrak{a}} (\sigma^n (\alpha)) - \pi_{p,\rho \mathfrak{a}} (\sigma^n (\beta)) || < \frac{R(n) r}{ \Lambda_{p,\mathfrak{a} ,\rho}} \} 
        \hspace{4 cm} \le  \mathrm{Leb}_d   \{   p \in \mathcal{B} :  ||  \pi_{p,\rho \mathfrak{a}} (\sigma^n(\alpha))- \pi_{p,\rho \mathfrak{a}} (\sigma^n(\beta)) ||  < \frac{D_5 R(n) r}{  \Lambda_{p_0,    \mathfrak{a}_0    ,\rho}^{1+\frac{\epsilon}{2 N}  }  } \} \nonumber
      \end{equation}
         \small

      \noindent
      To conclude, by ${\bf(T)}$ and since $\mathcal{B}  \subset \mathcal{P}$ the latter  is smaller than:
       $$ r^N \cdot Q(n) \cdot   \Lambda_{p_0, \mathfrak{a}_0 ,\rho}^{ - N- \epsilon/2  } \text{ with } Q(n):= C \cdot D^{ N}_5  \cdot (R(n))^N    \, .$$ 
       \noindent The result follows since $ \Lambda_{p_0, \mathfrak{a}_0 ,\rho}$ decreases exponentially with $n$, and $\epsilon/2<2\epsilon/3$. The proof of the second item is similar, by taking $\epsilon’$ close to 1 in  Eq. (\ref{torp}). 
   \end{proof}

   \section{Jets: Proof of Theorem B} \label{jet}

   \noindent We now prove  Theorem B. 
    The strategy is to study the dynamics of the family $(\mathcal{F}_p)_{p}$ inside the local stable manifolds to reduce the problem to the dynamics of a family $(F_p)_{p}$ of $C^r$-skew-products (Step 1) with one-dimensional fibers, from which we construct a family of $C^2$-skew-products $(G_{p_0})_{p_0}$ acting on $s$-jets (Step 2). Then we extend the latter one into a larger family $(G_{q_0})_{q_0}$ to satisfy the transversality assumption {\bf (T)} (Step 3). Finally, we look at the intersection between the unstable set and a family of curves all close to a stable manifold. In each curve, this intersection is equal to the limit set of a perturbation of the skew-product. We then apply successively Theorem C and the Fubini Theorem to conclude to a set of jets of positive measure at   a.e. parameter, which gives the parablender property (Step 4).  \newline 
    
\noindent {\bf Step 1: Dynamically defined family of skew-products.} We first need to define local stable and unstable manifolds. Let us fix a small $\varepsilon>0$   and an arbitrary parameter in $\overline{\mathcal{P}}$, taken arbitrarily equal to 0 for simplicity.    We recall that $\mathcal{K}_0$ is a hyperbolic basic set for $\mathcal{F}_0$.  Up to a change of metric on the stable (resp. unstable)  bundles of $\mathcal{K}_0$, we suppose that $D\mathcal{F}_0$ strictly contracts (resp. expands) the stable (resp. unstable) bundle by a factor $\lambda<1$ (resp. $1/\lambda $) uniform  over $z \in \mathcal{K}_0$. \newline  
 
\noindent  It has been shown by Qian and Zhang (see section 4 of \cite{qz}) that the limit inverse $\overleftrightarrow{\mathcal{K}_0}$ can be endowed with a map $[ \cdot ]$ defined on a subset of $\overleftrightarrow{\mathcal{K}_0} \times \overleftrightarrow{\mathcal{K}_0}$ with values in  $\overleftrightarrow{\mathcal{K}_0}$ so that for every sufficiently  closed orbits $x, y \in \overleftrightarrow{\mathcal{K}_0}$, the orbit $z=[x,y]$ is well-defined and   the 0-coordinate projection    $\pi(z) \in \mathcal{K}_0$ of $z$ is the intersection of the local unstable manifold of $x$ and the local stable manifold of $y_0$. This map endows $\overleftrightarrow{\mathcal{K}_0}$ with a structure of Smale space (see Ruelle \cite{ruell}, Chapter 7, for the definition of a Smale space). This implies that  $\overleftrightarrow{\mathcal{K}_0} $ admits Markov partitions of arbitrarily small diameter (see again Ruelle \cite{ruell}, Chapter 7, for this result and the definition and properties  of a Markov partition for a Smale space). We pick such a partition of $\overleftrightarrow{\mathcal{K}_0}$ by a finite number of compact rectangles $\overline{\mathcal{R}}_1, \ldots, \overline{\mathcal{R}}_M$ of diameter small compared to $\varepsilon$.  Up to reducing $\varepsilon$, we can suppose that $\mathcal{F}_0$ is a local diffeomorphism when restricted to the 0-coordinate projection $\mathcal{R}_i:= \pi(\overline{\mathcal{R}}_i) \subset \mathcal{K}_0$ of   any  rectangle $\overline{\mathcal{R}}_i$.  \newline 
 
\noindent The topological entropy of $\mathcal{F}_0 | \mathcal{K}_0$ is larger than $\delta_{d,s} \cdot     | \log m(D \mathcal{F}_0) |$ by $(\star)$. Then for      large $\ell$ we can pick $N_{\varepsilon,\ell}$  orbits $x_a=(x^i_a)_{0 \le i \le \ell-1}$ of length $\ell$ (with $a \in \mathcal{A}:=\{1,2, \ldots, N_{\varepsilon,\ell}\}$) under $\mathcal{F}_0$ which are $(\ell,\varepsilon)$-separated (in the sense of Bowen) with:
$$ (\diamond) \hspace{1cm}  \frac{ \mathrm{log} N_{\varepsilon,\ell} } { \ell } > \delta_{d,s} \cdot     |  \log m( D \mathcal{F}_0) |  \, .$$
In particular  the cardinality $N_{\varepsilon,\ell}$ of $\mathcal{A}$ is at least 2. We can extend each of these orbits  $x_a$ of length $\ell$   into an infinite orbit in $\overleftrightarrow{\mathcal{K}_0} $, still denoted $x_a$. Up to slightly perturbating $x_a$, we can suppose that for every $0 \le  k \le   \ell-1$, the orbit  $\overleftrightarrow{\mathcal{F}_0}^k (x_a)$ is in the interior of some rectangle $\overline{\mathcal{R}}_{i}$ that we denote $\overline{\mathcal{R}}_{a,k}$. In particular $x_a^k = \mathcal{F}_0^k(x^0_a)$ is in $\mathcal{R}_{a,k} = \pi ( \overline{\mathcal{R}}_{a,k} )$. Since the rectangles $\overline{\mathcal{R}}_i$ are in finite number $M$ independent of $\ell$, up to modifying $N_{\varepsilon,\ell}$ by a multiplicative constant (independent of $\ell$ large), we can suppose that all     orbits $x_a$ ($a \in \mathcal{A}$) begin in the same $\mathcal{R}_i$ and end in the same $\mathcal{R}_j$, with $(\diamond)$ still true. 
 For any $a \in \mathcal{A}$, let us define:   
 $$R_a :=                        \mathcal{R}_{a,0}  \cap \mathcal{F}^{-1}_0 ( \mathcal{R}_{a,1} ) \cap \cdots \cap \mathcal{F}^{-(\ell-1)}_0 ( \mathcal{R}_{a,\ell-1})     \, .$$ 
 which is not empty by assumption.
Since the $\ell$-orbits $(x^i_a)_{0 \le i \le \ell-1}$ (with $a \in \mathcal{A}$) under $\mathcal{F}_0$ are $(\ell,\varepsilon)$-separated and since the diameter of each rectangle of the partition is small compared to $\varepsilon$,  the sets $R_a$ ($a \in \mathcal{A}$)  are pairwise disjoint compact subsets of $\mathcal{R}_i$ having their images by $\mathcal{F}^\ell_0$  included in $\mathcal{R}_j$. Up to adding a constant (independent of $\ell$ large) to $\ell$,  using the properties of  Markov partitions, we can suppose moreover that   $\mathcal{R}_i$ and $\mathcal{R}_j$ are equal. We denote this set by $R$ in the following and thus the sets $R_a$ are non empty pairwise disjoint compact  subsets of $R$ with images under $\mathcal{F}^\ell_0$  included in $R$. This does not change the number $N_{\varepsilon,\ell}$ of such sets $R_a$ and so inequality $(\diamond)$ is still true.
 For every non empty finite word $\beta= \beta_0 \cdots \beta_p \in \mathcal{A}^{p+1}$, we now set
 $$  R_\beta :=               R_{\beta_0  }  \cap \mathcal{F}^{-\ell}_0 (R_{\beta_1} ) \cap \cdots \cap \mathcal{F}^{-\ell p}_0 (R_{\beta_p})       \, . $$ which is a non empty compact set by the properties of  Markov partitions,  and the sets $R_\beta$ among $\beta \in  \mathcal{A}^{p+1}$ are pairwise disjoint subsets of $R$ for a fixed value of $p$. Finally, for $\mathfrak{a} \in \overrightarrow{\mathcal{A}}$, we define:
 $$R_{\mathfrak{a}} := \bigcap_{n >0} R _{\mathfrak{a}_{| n}} $$
 which is also a non empty compact subset of $R$, included in $\mathcal{K}_0$, and  the sets $R_\mathfrak{a}$ ($\mathfrak{a}  \in  \overrightarrow{\mathcal{A}}$) are pairwise disjoint. Finally, we notice that since the partition admits a continuation in a neighborhood of 0 in $\mathcal{P}$, the sets $R_a$, $R_\beta$ and $R_\mathfrak{a}$ admit continuations $R_{p,a}$, $R_{p,\beta}$ and $R_{p,\mathfrak{a}}$ with the same dynamical properties when $p$ varies  in a neighborhood of 0 in the parameter space. Up to taking a finite covering of $\overline{\mathcal{P}}$ by such neighborhoods and extending a finite number of times $(\mathcal{F}_p)_p$ in Steps 2-3-4, we can suppose that $R_{p,a}$, $R_{p,\beta}$ and $R_{p,\mathfrak{a}}$ vary continuously in a neighborhood of $\overline{\mathcal{P}}$, that we can suppose equal to $\mathcal{P}’$ up to reducing it, and  
 inequality $(\diamond)$  is satisfied in $\overline{\mathcal{P}}$.  \newline

 \noindent  Let us take a small $\varepsilon’>0$. 
 For any infinite forward sequence  $\mathfrak{a} \in \overrightarrow{ \mathcal A}$ and any $p \in \mathcal{P}’$, we notice that all the points  in $R_{p, \mathfrak{a}}$ are asymptotic (using the hyperbolicity) and thus belong to a same stable manifold $W$ of $\mathcal{K}_p$ (which is one-dimensional). 
We define $W^{\mathfrak{a}}_p$ as the $\varepsilon’$-neighborhood in $W$ of the maximal arc of $W$ bounded by  points of $R_{p, \mathfrak{a}}$. We parameterize $W^{\mathfrak{a}}_p$ with $X:=[-1,1]$ via a $C^r$-map $s_p^\mathfrak{a}$ s.t. $(s_p^\mathfrak{a})_p$ depends  H\"older  in the $C^{r-1}$-topology  on $\mathfrak{a}$ by 
  Remark \ref{epicerie}. \newline
 
 \noindent Since $\mathcal{F}_p$ is a local diffeomorphism and since $W^{\mathfrak{a}}_p$ is injectively immersed, up to decreasing $\varepsilon’$,  the restriction of $\mathcal{F}_p^\ell$ to each $W_p^\mathfrak{a}$ is a diffeomorphism satisfying:
 $$\forall p \in \mathcal{P}’, \forall  \mathfrak{a} \in \overrightarrow{ \mathcal A}, \text{ }  \mathcal{F}^\ell_p (W_p^\mathfrak{a}) \subset \mathring{W}_p^{\sigma (\mathfrak{a})} :=s_p^{\sigma(\mathfrak{a})} ((-1,1)) $$ by hyperbolicity. We define   the $C^r$-diffeomorphism $f_{p,\mathfrak{a}}:=                ( s_p^{\sigma (\mathfrak{a})}  )^{-1}  \circ   \mathcal{F}^\ell_p     \circ   s_p^\mathfrak{a} $ on a small neighborhood $X’$ of $X$ independent of $(p,\mathfrak{a})$. Its image 
     is s.t. $f_{p,\mathfrak{a}}(X’) \Subset X$. 
The differential $D\mathcal{F}^\ell_p$ strictly contracts (resp. expands) $\mathcal{E}^s_{p,z}$ (resp. $\mathcal{E}^u_{p,z}$) by a factor $\lambda^\ell <\lambda<1$ (resp. $1/\lambda$) uniform over $p \in \overline{\mathcal{P}}$ and $z \in \mathcal{K}_p$. Up to modifying $\mathcal{F}_p^\ell$ outside a neighborhood of $\mathcal{K}_p$ (this operation does not modify the local stable/unstable sets of $\mathcal{K}_p$ in a neighborhood of $\mathcal{K}_p$ nor jets of points inside it), we can suppose that $\mathcal{F}_p^\ell$ contracts strictly each $W_p^\mathfrak{a}$ by $\lambda$ and so each map   $f_{p,\mathfrak{a}}$  contracts by  $<\lambda$. Moreover the map $(p,x) \mapsto f_{p,\mathfrak{a}}(x)$ is $C^r$. It extends on  $\mathcal{P}’ \times X’$, up to reducing $\mathcal{P}’$ and $X’$. Easy computations show that the map $\mathfrak{a}  \mapsto ((p,x) \mapsto f_{p,\mathfrak{a}} (x) )$ is H\"older for the $C^{r-1}$-topology and continuous for the $C^r$-topology.
 We set
$$ F_{ p} : (\mathfrak{a},x) \in \overrightarrow{ \mathcal A} \times X \mapsto  (\sigma (\mathfrak{a}), f_{p,\mathfrak{a}}(x)) \in    \overrightarrow{ \mathcal A}   \times X  \, , $$
and $(F_p)_p$  is a     family of fiberwise $\lambda$-contracting $C^r$-skew-products satisfying  the preliminary assumptions of Theorem C. We now show:

\vspace{0.2cm}

\begin{Lemma} \label{sim} 
The similarity dimension $\Delta(p)$ is larger than $\delta_{d,s}$ for any $p \in \overline{\mathcal{P}}$.
\end{Lemma}

\begin{proof}
 The fact that the similarity dimension does not depend on the fiber is a  consequence of Proposition \ref{pression}. Moreover for a fixed fiber  $\mathfrak{a}$,  there are $N_{\varepsilon,\ell}$ contractions $f_{p,\mathfrak{a}}$ which contract by at most $m( D \mathcal{F}_p)^\ell$. Thus the result follows by the definition of the similarity dimension and inequality $(\diamond)$. 
\end{proof} 

\vspace{0.2cm}


\vspace{0.2cm}

\noindent Using backward sequences instead of forward sequences, we can define similarly families $(W^\alpha_p)_p$ of local unstable manifolds of $\mathcal{K}_p$ parameterized by maps $s^\alpha_p$ s.t.  the map $ \mathcal{F}_p^\ell$  restricted to some subset of each $\mathring{W}_p^{\sigma (\alpha )} := s^{\sigma(\alpha)}_p((-1,1))$ is a diffeomorphism onto $W_p^{\alpha}$
which expands strictly by $1/\lambda$. Moreover the local unstable manifold $W_p^{\alpha}$ intersects the local stable manifold $W_p^\mathfrak{a}$ at a unique point.

\begin{Remark} \label{remutil}
For any $C^r$-family $(\mathcal{F}_q)_{q \in \mathcal{Q}}$ of endomorphisms which extends $(\mathcal{F}_p)_{p \in \mathcal{P}}$ (by this, we mean that $\mathcal{Q}$ is a neighborhood of $\mathcal{P}$), by hyperbolic continuation, we can extend $(\mathcal{K}_p)_{p \in \mathcal{P}}$ into           the   continuation  $(\mathcal{K}_q)_{q \in \mathcal{Q}}$ of a  hyperbolic basic set of stable dimension 1, up to reducing $\mathcal{Q}$. Thus we can also extend the local stable and  unstable manifolds (using    Theorem F  in the Appendix) and so   the family $(F_p)_p$ into a      family $(F_q)_q$ of fiberwise $\lambda$-contracting $C^r$-skew-products. 
\end{Remark}

\noindent In the following, we prove Theorem B in the case $\ell=1$. The proof is the same when $\ell>1$, with a variant  in the proof of Proposition \ref{keylemme3} described in Remark \ref{keyrem3}.

\vspace{0.2cm}

\noindent   {\bf Step 2: Maps acting on jets.} 
  From $(F_p)_p$, we define a family of $C^2$-skew-products $(G_{p_0})_{p_0}$ acting on $s$-jets at any $p_0$. We begin by  performing this Step for $d=1$ for the sake of simplicity concerning the notations, and then treat the general case  $d \ge 2$.  \newline 
   
   \noindent 
   For any $p_0 \in \mathcal{P}’$ and $\mathfrak{a} \in \overrightarrow{\mathcal A}$, we define from $(f_{p, \mathfrak{a}})_p$ the following map $g_{p_0, \mathfrak{a}}$ acting on $s$-jets:
   $$g_{p_0, \mathfrak{a}} : (x_p, \partial_p x_p , \ldots, \partial^s_p x_p)_{ | p = p_0}  \mapsto   (f_{p, \mathfrak{a}} (x_p), \partial_p (f_{p, \mathfrak{a}} (x_p)) , \ldots, \partial^s_p (f_{p, \mathfrak{a}} ( x_p)))_{     | p = p_0} \, , $$
when defined.     Since  $(f_{p, \mathfrak{a}})_p$ is a $C^r$-family  of maps with $s \le r-2$, the family $(g_{p_0, \mathfrak{a}})_{p_0}$  is itself a  $C^2$-family  of maps. It depends continuously on $\mathfrak{a}$. 
  \noindent
   Notice that 
   $$ \partial_p (f_{p, \mathfrak{a}} (x_p))_{ | p = p_0}   = Df_{p_0} (x_{p_0}  ) \cdot  \partial_p (x_p)_{ | p = p_0} +  \partial_p (   f_p(   x_{p_0}    )      )_{ | p = p_0} \, .$$
   We notice that  $\partial_p (f_{p, \mathfrak{a}} (x_p))_{ | p = p_0}  $ has possibly a non zero derivative relatively to $x_{p_0}$ (independent of $\partial_p (x_p)_{ | p = p_0} $),  a derivative $Df_{p_0} (x_{p_0}  )$ relatively to $\partial_p (x_p)_{ | p = p_0} $ and  derivatives equal to zero relatively to each $ \partial^j_p (x_p)_{ | p = p_0} $ for $2 \le j \le s$.  \newline 
   
\noindent When $1 < k \le s$, the term $\partial^k_p (f_{p, \mathfrak{a}} (x_p))_{ | p = p_0}  $  has  possibly a non zero derivative relatively to $ \partial^j_p (x_p)_{ | p = p_0} $ for $0 \le j \le k-1$ (independent of $\partial^k_p (x_p)_{ | p = p_0} $),  a  derivative $Df_{p_0} (x_{p_0}  )$ relatively to $\partial^k_p (x_p)_{ | p = p_0} $ and derivatives equal to zero relatively to each $ \partial^j_p (x_p)_{ | p = p_0} $ for $k< j \le s$. To see this, it is enough to write the map $p \mapsto f_{p, \mathfrak{a}} (x_p)$ as the composition of $p \mapsto (p,x_p)$ and $(p,x) \mapsto f_{p, \mathfrak{a}} (x)$ and to apply the multidimensional Fa\`a di Bruno formula (see for example Theorem 2.1 in \cite{formule}).\newline 

\noindent We recall that by assumption we have $0<|Df_{p_0} (x_{p_0})|<\lambda<1$. Then    for any  $p_0 \in 
 \mathcal{P}’$ and $\mathfrak{a} \in \overrightarrow{\mathcal A}$, the map $g_{p_0, \mathfrak{a}}$ has inferior unipotent differentials with eigenvalues  uniformly bounded between 0 and 1 in modulus. \newline

\noindent 
We recall that  $\delta_{1,s} = s+1$. 
We define a set of $s$-jets $Y$, identified with a subset of  $\mathbb{R}^{s+1}$ as follows (the term of degree $i$ corresponding to the $i^{th}$-coordinate). We set: $$Y:= X \times [-R_1,R_1] \times \cdots \times [-R_s, R_s] \, .$$ 
\noindent We can choose $R_1$ large compared to the diameter of $X$ and $R_{i+1}$ large compared to $R_i$ for $1 \le i \le s-1$. 
Since its differentials are unipotent with non zero contracting  diagonal coefficients, an immediate induction shows that $g_{p_0, \mathfrak{a}}$ is a $C^2$-diffeomorphism from a small neighborhood $Y’$ of $Y$ onto  $g_{p_0, \mathfrak{a}}(Y’) \Subset Y$. Up to rescaling, we can suppose that $Y:=[-1,1]^{s+1}$.
The map $\mathfrak{a} \mapsto ((p,x) \mapsto f_{p,\mathfrak{a}} (x) )$ is H\"older for the $C^{r-1}$-topology and continuous for the $C^r$-topology
 and we have  $s \le r-2$. This implies that the map $\mathfrak{a} \mapsto ((p_0,y) \mapsto g_{p_0,\mathfrak{a}} (y))$ is H\"older for the $C^1$-topology and continuous for the $C^2$-topology. To summarize, we just prove that $(G_{ p_0})_{p_0}$, with
$$ G_{ p_0} : (\mathfrak{a},y)   \in    \overrightarrow{ \mathcal A}   \times Y  \mapsto  (\sigma (\mathfrak{a}), g_{p_0,\mathfrak{a}}(y)) \in    \overrightarrow{ \mathcal A}   \times Y \, , $$
is a    family of $C^2$-skew-products satisfying  the preliminary assumptions of Theorem C,   ${\bf(U)}$ and  $\Delta(p_0)>s+1=\delta_{1,s}$ for any $p_0 \in \overline{\mathcal{P}}$. \newline

\noindent   As already  said, we get the same result with $d \ge 2$ but with painful notations. Indeed, it is enough to remark that the action of $(f_{p, \mathfrak{a}})_p$ on jets of multiorder $(k_1,\ldots , k_d)$ with $\sum_i k_i \le s$ only depends on jets of the same multiorder, with a linear coefficient $Df_{p_0} (x_{p_0}  )$, and on jets  of multiorder $(k’_1,\ldots , k’_d)$ with $\sum_i k’_i <  \sum_i k_i$. In particular, its differentials still satisfy assumption {\bf (U)}.  

\vspace{0.2cm}

%
%
%
%
%
%
%
%
\noindent {\bf Step 3: Extending the family.}  
We now extend $(G_{p_0})_{p_0}$ to satisfy the transversality assumption {\bf (T)} inside a larger family.   But we have to ensure that this extension comes from an extension  of $(F_p)_p$, itself coming from an extension of $(\mathcal{F}_p)_{p}$.

   \begin{Proposition} \label{keylemme3}  There exists a  family  of $C^2$-skew-products $(G_{q_0})_{q_0 \in \mathcal{Q}}$:
    \begin{equation}  \label{exit3}
   G_{q_0} : (\mathfrak{a},y) \in \overrightarrow{ \mathcal A} \times  Y \mapsto  (\sigma (\mathfrak{a}), g_{q_0,\mathfrak{a}}(y)) \in    \overrightarrow{ \mathcal A}   \times Y  \, , \nonumber
    \end{equation}
    with  $\mathcal{Q}:=\mathcal{P}  \times (-1,1)^{m}$ for some  $m>0$,  
   satisfying  the preliminary assumptions of Theorem C,  {\bf (U)}, {\bf (T)}, $\Delta(q_0)> \delta_{d,s}$ for $q_0 \in \overline{\mathcal{Q}}$  and:
   $$G_{(p_0,0)} = G_{p_0} \text{ and } g_{(p_0 ,0), \mathfrak{a}} = g_{p_0  ,  \mathfrak{a}} \text{ for every } p_0  \in  \overline{\mathcal{P}}   \text{ and } \mathfrak{a } \in \overrightarrow{\mathcal{A}} \, .$$
   Moreover there exists a $C^r$-family $(\mathcal{F}_q)_{q \in \mathcal{Q}}$ of  local diffeomorphisms extending $(\mathcal{F}_p)_{p \in \mathcal{P}}$ s.t. if $ (F_q)_q=(f_{q,\mathfrak{a}})_q$ is its associated  family of $C^r$-skew-products, we have $f_{(p_0,0) ,  \mathfrak{a} } = f_{p_0 ,  \mathfrak{a}}$ and $g_{q_0,\mathfrak{a}}$ is the map  acting on the $s$-jets at $p_0$ derived from $(f_{q,\mathfrak{a}})_q$ (here  the jets at $p_0$ are taken varying $p$ for fixed $q’_0$ where $q_0:=(p_0,q’_0)$). The family  $(\mathcal{F}_q)_{q \in \mathcal{Q}}$ is   of the form $\mathcal{F}_q = \mathcal{F}_{(p,q’)} = \mathcal{F}_p + \Sigma_{p,q’} $ where $(\Sigma_{p,q’})_{(p,q’)}$ is a $C^r$-family s.t. $\Sigma_{p,0} = 0$. 
     \end{Proposition} \vspace{0.3cm}

\noindent      We now finish the proof of Theorem B. We postpone the proof of Proposition \ref{keylemme3} after it  since it is technical.

\vspace{0.2cm}

\noindent   {\bf Step 4: Conclusion.}  
 Let us pick the family $(G_{q_0})_{q_0}$ of  $C^2$-skew-products   satisfying  the preliminary assumptions of Theorem C,  {\bf (U)}, {\bf (T)} and $\Delta(q_0)>\delta_{d,s}$ for every $q_0 \in \overline{\mathcal{Q}}$ given by Proposition  \ref{keylemme3}. We will pick  well-chosen $\vartheta$-$\mathrm{U}$-perturbations and apply the second part of Theorem C. We will conclude by using the Fubini Theorem. \newline

\noindent 
We pick the continuation $(k_q)_q$ of a point $k_q \in \mathcal{K}_q$ at the intersection of  local stable and unstable manifolds  $W_q^{\mathfrak{a}}$ and $W_q^\alpha$ for arbitrary $\mathfrak{a} \in \overrightarrow{\mathcal{A}}$ and $\alpha \in \overleftarrow{\mathcal{A}}$. 
Up to working locally in the parameter space $\mathcal{Q}$, we can pick  a  $C^r$-family $(\Gamma_{t,q})_{t,q}$ of segments of same direction $u_\Gamma$ (these segments remain all parallel when varying $t$ and $q$), which do not intersect $W^s(\mathcal{K}_q)$ and which  intersect $(W_q^\alpha)_{t,q}$ in a curve $(z_{t,q})_{t,q}$ $C^0$-close to $(k_q)_{t,q}$. Moreover we choose $(\Gamma_{t,q})_{t,q}$ s.t. the $v_\Gamma$-coordinate $\mathcal{X}(q,t)$ in the basis $(u_\Gamma,v_\Gamma)$ of $\Gamma_{t,q}$,   
for a  fixed direction $v_\Gamma$ transversal to $\Gamma_{t,q}$,  is of the form $$\mathcal{X}(q,t):= \sum_{  | i |  \le s} t_i  p^i \text{  for  } t=(t_1, \cdots, t_{\delta_{d,s} } ) \in \mathcal{T}:= (-1,1)^{\delta_{d,s} }\,  $$
which does not depend on $q’$ but only on $p$ for $q=(p,q’)$. In particular, for any $q_0 \in \mathcal{Q}$, the map $t \in \mathcal{T} \mapsto  \mathrm{J }_{q_0}^s \mathcal{X}(q,t)$ sends  diffeomorphically $\mathcal{T}$  to a non empty set of $s$-jets in $p$  (and thus a set of $s$-jets in $p$ of  positive $\delta_{d,s}$-dimensional measure). \newline

\noindent By the parametric inclination Lemma \ref{refff} (in the Appendix), up to taking an inverse iterate, for any $t \in \mathcal{T}$, we can suppose that $(\Gamma_{t,q})_{q}$ is (uniformly in $t \in \mathcal{T}$) close to the continuation $(W^{\mathfrak{a}}_q)_q$ of a local stable manifold for some $\mathfrak{a} \in \overrightarrow{ \mathcal A}$. We can parameterize it  by    the segment  $X$ with a $C^r$-family of charts $C^r$-close to $(s^{\mathfrak{a}}_q)_q$. Still iterating backwards, for every $a \in \mathcal{A}$,   there is a family of  submanifolds $(\Gamma^a_{t,q})_{q}$ close to $(W^{a\mathfrak{a}}_q)_q$ (uniformly in $t \in \mathcal{T}$) s.t. $\Gamma^a_{t,q}$ is sent by $\mathcal{F}_q$ into the interior of $\Gamma_{t,q}$. We can parameterize $(\Gamma^a_{t,q})_{q}$ by $X$ with a family of charts  close to $(s^{a\mathfrak{a}}_q)_q$ and  in these charts the restriction of $\mathcal{F}_q$ from $\Gamma^a_{t,q}$ into $\Gamma_{t,q}$ defines  a family of maps $(\tilde{f}_{t,q,a \mathfrak{a} })_{q}$  $C^r$-close to $(f_{q,a \mathfrak{a} })_{q}$. Iterating backwards by induction, we can define  for every $\alpha \in \mathcal{A}^*$ families of curves $(\Gamma^\alpha_{t,q})_{q}$ which are (uniformly in $\alpha$) $C^r$-close to $(W^{\alpha \mathfrak{a}}_q)_q$ by Lemma \ref{refff} and   s.t. each  $\Gamma^{a \alpha}_{t,q}$ is sent into the interior of $\Gamma^\alpha_{t,q}$. This defines families of maps  $(\tilde{f}_{t,q,\alpha \mathfrak{a}} )_{q}$ $C^r$-close to $(f_{q, \alpha \mathfrak{a} })_{q}$ (uniformly in $t \in \mathcal{T}$ and $\alpha \in \mathcal{A}^*$). We set $\tilde{f}_{t,q,\mathfrak{b} } := f_{q, \mathfrak{b} }$ when 
  $\mathfrak{b}$ is not of the form $\alpha  \mathfrak{a}$ with $\alpha$ a non empty word. When looking at the action of the families of maps $(\tilde{f}_{t,q,\mathfrak{a} })_{q}$ on $s$-jets (w.r.t. $p$) for fixed values of $t$ and $\mathfrak{a}$,  one obtains a family of maps $(\tilde{g}_{t,q_0, \mathfrak{a}} )_{q_0}$ from $Y$ into itself $C^2$-close to $(g_{q_0, \mathfrak{a} })_{q_0}$ (uniformly in $t$ and $\mathfrak{a}$).
This defines a family  (indexed by $t$) of $\vartheta$-$\mathrm{U}$-perturbations of the family $(G_{q_0})_{q_0}$ of $C^2$-skew-products, with arbitrarily small $\vartheta$. \newline 

\noindent 
By  the second part of Theorem C, we have 
$$\mathrm{Leb}_{\delta_{d,s} }(  \tilde{K}_{t,q_0,\mathfrak{a}} ) >0 \text{ for }    \mathrm{Leb}_{d+m} \text{ a.e. } q_0 
  \in \mathcal{Q}          \text{ and  }     \mathrm{Leb}_{\delta_{d, s}} \text{  a.e. } t \in   \mathcal{T}       \, ,$$ where the limit set $ \tilde{K}_{t,q_0,\mathfrak{a}}$ is formed by
   jets at $p_0$ taken while varying $p$ for fixed $q’_0$ where $q_0 := (p_0,q’_0)$.      By the Fubini Theorem, for a.e. $q’_0$ we have:
   $$\mathrm{Leb}_{\delta_{d,s} }(   \tilde{K}_{t,(p_0,q’_0),\mathfrak{a}} ) >0 \text{ for }      \mathrm{Leb}_d \text{ a.e. } p_0 \in \mathcal{P}  \text{ and  for }  \mathrm{Leb}_{\delta_{d,s}} \text{  a.e. } t \in   \mathcal{T}     \, .$$ 
  We  fix such a $q’_0$. 
  We consider  $(\tilde{\mathcal{F}}_{p})_p := (\mathcal{F}_{(p,q’_0)})_p$ and $\tilde{\Gamma}_{t,p}:= \Gamma_{t,(p,q’_0)}$. We notice that $\tilde{K}_{t,(p_0,q’_0),\mathfrak{a}}$ is (in the charts) the set of the $s$-jets  at $p_0$   of the 
    intersection points between  the local unstable set of $\tilde{\mathcal{K}}_{p}$ and $\tilde{\Gamma}_{t,p}$. Moreover the   set of $s$-jets at any $p_0$  of    the $v_\Gamma$-coordinate  of $\tilde{\Gamma}_{t,p}$  in the basis $(u_\Gamma,v_\Gamma)$ when varying $t$ in $\mathcal{T}$ has positive $\delta_{d,s}$-dimensional Lebesgue measure. Since we have a  positive set of one-dimensional $s$-jets  in the direction of $u_\Gamma$ for a.e. $t \in \mathcal{T}$, we just have to use the Fubini Theorem  to conclude to a set of bidimensional $s$-jets of positive measure for a.e. $p_0 \in \mathcal{P}$. \newline
    
      \noindent
     The same proof works for every family $(\mathcal{G}_p)_p$ which is $C^r$-close to $(\mathcal{F}_p)_p$ with the extension $(\mathcal{G}_q)_{q \in \mathcal{Q}}$ given by $\mathcal{G}_q = \mathcal{G}_p + \Sigma_{p,q’} $ where $(\Sigma_{p,q’})_{(p,q’)}$ is the $C^r$-family given by     Proposition \ref{keylemme3}. Indeed, the preliminary conditions of Theorem C, {\bf (U)}, $\Delta(q_0)> \delta_{d,s}$ are open conditions and the extension in Proposition \ref{keylemme3} in order to get {\bf (T)} works for nearby families with the same additive perturbation since having a relative positive speed is an open property. Thus we can apply Theorem C to the family of $C^2$-skew-products derived from $(\mathcal{G}_q)_{q }$.
     This  achieves  the proof that  $(\mathcal{K}_{p})_p$ is an almost  $C^{r,s}$-parablender  and thus shows     Theorem B.  
    %
    %
    %
   %
%
 %
\begin{Remark} \label{retour}
When the family $(\mathcal{F}_p)_p$ is of the form $(\mathcal{F})_p$ with $\mathcal{F}$ independent of $p$ and when the order $s$ of the jets is equal to 0, the extension $(\mathcal{F}_q)_q$ satisfying the conclusions of Theorem B can be taken of the form $(\mathcal{F}+\Sigma_{q’})_{(p,q’)}$  with $(\Sigma_{q’})_{q’}$  independent of $p$. This will help to prove  Theorem A from  Theorem B. 
\end{Remark} 

   \begin{proof}[Proof of Proposition \ref{keylemme3}]
   The proof is divided in two Steps. The goal and main difficulty of the extension is about satisfying {\bf (T)}: we want to extend the family $(\mathcal{F}_p)_{p \in \mathcal{P}}$ by adding finitely many new parameters,  which will give a positive relative speed to pairs of limit points with different combinatorics, inside each fiber. We set
   $$   \Omega := \{ (\mathfrak{a},\alpha, \beta) \in 
 \overrightarrow{\mathcal{A}}  \times \overleftarrow{\mathcal{A}} \times \overleftarrow{\mathcal{A}}: \alpha_{-1} \neq \beta_{-1}  \}    $$  
 and by introducing new parameters, we will give a positive relative speed to the points coded by $\alpha$ and $\beta$   inside the fiber coded by $\mathfrak{a}$ for any  $(\mathfrak{a},\alpha, \beta) \in   \Omega$, 
 and this will give the transversality assumption {\bf (T)} inside the extended family. \newline 
 
\noindent  We begin by choosing a covering of  $ \Omega$ by small products of cylinders. The reason why we are going to work with a covering by small products of cylinders is that this will allow us to control precisely the relative movement of the two limit points. We will then extend  iteratively the original family  $(\mathcal{F}_p)_{p \in \mathcal{P}}$ by adding new parameters for each set of the covering.\newline 

\noindent We first pick an arbitrary number $h>0$. The extended family  $(\mathcal{F}_q)_{q \in \mathcal{Q}}$ will be taken s.t. the families  $(\mathcal{F}_{(p,q’)})_{p \in \mathcal{P}}$ are (uniformly in $q’$)  $h$-$C^r$-close to  $(\mathcal{F}_p)_{p \in \mathcal{P}}$. We pick  an integer $M$ large enough so that $ \lambda^M/(1-\lambda)$ is small. \newline

   %
 \noindent   {\bf Step 1: Working locally.}
 %
 %
 Let us pick any $(\mathfrak{a},\alpha, \beta) \in  \Omega$. 
 If $\mathfrak{a}$ is periodic of minimal period $\mathfrak{p} \in \mathcal{A}^*$, since $\alpha_{-1} \neq \beta_{-1}$, we can choose $\beta$ s.t. $\beta_{-1}$ is different from the last letter of $\mathfrak{p}$. 
 
\begin{Fact} \label{f12}
The sequences $\sigma^k(\mathfrak{a})$ where $0 \le k \le M$ are all distinct from $\beta_{-1}  \mathfrak{a}$. 
\end{Fact}

\begin{Fact}  \label{f15}
The sequences $\beta_{| k}  \mathfrak{a}$ where $2 \le k \le M$ are all distinct from $\beta_{-1}  \mathfrak{a}$. 
\end{Fact}

\begin{proof}[Proof of Facts \ref{f12} and  \ref{f15}]
If  $\mathfrak{a}$ is not periodic, this is immediate. If $\mathfrak{a}$ is periodic, this is due to the fact that $\beta_{-1}$ is different from the last letter of its minimal period $\mathfrak{p}$.
\end{proof}

\noindent From now,  we  distinguish two disjoint Cases which cover all possible tuples in $ \Omega$: \newline 

\noindent  {\bf Case 1: the sequences $\alpha_{| k}  \mathfrak{a}$ where $1 \le k \le M$ are all distinct from $\beta_{-1}  \mathfrak{a}$.} In this case, we consider the two cylinders $[\rho_\alpha]$ and $[\rho_\beta]$ of $ \overleftarrow{\mathcal{A}}$ of length $M$ defined by $\rho_\alpha := \alpha_{|M}$ and $\rho_\beta := \beta_{|M}$. We   pick a small cylinder $[\rho_{\mathfrak{a}}]$ of $\overrightarrow{\mathcal{A}}$ s.t.  $\ \mathfrak{a} \in [\rho_{\mathfrak{a}}]$ and:
\begin{itemize}
\item The cylinders $[\alpha_{|k} \rho_{\mathfrak{a}}]$ where $1 \le k \le M$ are all disjoint from $[\beta_{-1} \rho_{\mathfrak{a}}]$.  
\item The cylinders $[ \beta_{|k}  \rho_{\mathfrak{a}}]$ where $2 \le k \le M$ are all disjoint from $[\beta_{-1} \rho_{\mathfrak{a}}]$.
\end{itemize}

\noindent  {\bf Case 2: $\mathfrak{a}$ is periodic of minimal period $\mathfrak{p}$ and there exists $f>0$ s.t. the last letters of $\alpha$ are of the form $\beta_{-1}  \cdot \mathfrak{p}^f$.} We consider two small cylinders $[\rho_\alpha]$ and $[\rho_\beta]$ of $ \overleftarrow{\mathcal{A}}$ containing $\alpha$ and $\beta$, and a small cylinder $[\rho_{\mathfrak{a}}]$ of $\overrightarrow{\mathcal{A}}$ s.t.  $\ \mathfrak{a} \in [\rho_{\mathfrak{a}}]$ and:
\begin{itemize}
\item The cylinders $\sigma^k( [\rho_{\mathfrak{a}}])$ where $0 \le k \le M$ are all disjoint from $[\beta_{-1} \rho_{\mathfrak{a}}]$.
\item The cylinders $[ \beta_{|k}  \rho_{\mathfrak{a}}]$ where $2 \le k \le M$ are all disjoint from $[\beta_{-1} \rho_{\mathfrak{a}}]$. 
\end{itemize}

\noindent We take these cylinders so that their lengths are large compared to $M$.  \newline

\noindent
We notice that the compact set $ \Omega$ is covered by the union of the (open) products of cylinders $[\rho_{\mathfrak{a}}] \times [\rho_\alpha ] \times [\rho_\beta ]$ associated to any $(\mathfrak{a},\alpha, \beta) \in  \Omega$. We can find then a finite 
covering of $ \Omega$ by such products. 
  \newline

 \noindent   {\bf Step 2: Extension.} We now construct an extension of  $(\mathcal{F}_p)_p$ obtained by extending successively a finite number of times the family, adding at each step $\delta_{d,s}$ parameters corresponding to a product  of cylinders $[\rho] \times [\rho'] \times [\rho''] \subset  \Omega$ in the finite covering defined in Step 1. Here $[\rho]$ is a cylinder in $\overrightarrow{\mathcal{A}}$ and $[\rho']$ and $[\rho'']$ are cylinders in $ \overleftarrow{\mathcal{A}}$ s.t. $\rho'_{-1} \neq \rho''_{-1}$. The $\delta_{d,s}$ parameters are intended to move the $s$-jet of the limit point associated to $\alpha$ relatively to the one corresponding to $\beta$, inside the fiber encoded by $\mathfrak{a}$, where $(\mathfrak{a},\alpha, \beta)$ is any tuple in $[\rho] \times [\rho'] \times [\rho'']$. \newline 
 
 \noindent The union of the points of $R_{p,\mathfrak{a}}$ on  the  local stable manifolds $W^\mathfrak{a}_p$ for $\mathfrak{a} \in [ \rho''_{-1} \rho]$ is disjoint from the union of the points of  $R_{p,\mathfrak{a}}$  on the local stable manifolds $W^\mathfrak{a}_p$ for $\mathfrak{a} \notin [ \rho''_{-1} \rho]$  at every parameter $p \in \overline{\mathcal{P}}$. 
  We then pick a $C^r$-family $(h_p)_p$ of bump functions $h_p$ equal to $h$ in a neighborhood of the first ones and equal to 0 in a neighborhood of the second ones  for every $p \in \overline{\mathcal{P}}$. We also pick a $C^r$-family of maps $v^s_p$ s.t. $v^s_p(z)$ is close to  the stable  direction of $\mathcal{K}_p$ at $\mathcal{F}_p(z)$  for any point $z \in \mathcal{K}_p$. \newline

 \noindent   We now extend the family by setting for every $q:=(p,q’)$ with $p \in \mathcal{P}’$ and $q’=(q_i)_i \in \mathbb{R}^{\delta_{d,s}}$ small:
    
   \begin{equation} \label{pertureb} 
   \mathcal{F}_q(z) = \mathcal{F}_p(z)+  h_p(z)   \cdot  \big( \sum_{ i } q_{i}   \cdot   p^i  \big) \cdot v^s_p(z) \hspace{0.5cm} \forall z \in \mathbb{R}^2 \, , 
    \end{equation}      
    
    \noindent where we sum over $i=(i_1, \ldots , i_d )$  s.t. $\sum_k i_k \le s$ with  $p^i=p_1^{i_1} \cdots p_d^{i_d}$. \newline 
    
    \noindent This defines a $C^r$-family  of  endomorphisms $(\mathcal{F}_q )_q$. For small values of $q’$, let us say for $q$ in some open neighborhood  of  $\mathcal{P} \times \{0\}$,  these are still local diffeomorphisms. Moreover, the family of  hyperbolic basic sets admits a continuation as a family  $(\mathcal{K}_q)_q$ (see the Appendix).
\newline    
  
  \noindent 
  By Remark \ref{remutil}, we can define the  associated  family of $C^r$-skew-products $(F_{q})_q$ which extends $(F_p)_p$ and then the   associated map $(G_{q_0})_{q_0}$   acting on the $s$-jets derived from  $(F_{q})_q$, defined as in Step 2 of the proof of Theorem B (let us recall that  the jets are taken while varying only $p$ for fixed $q’$ if we set $q:=(p,q’)$). In particular, $(G_{q_0})_{q_0}$ is a    family of $C^2$-skew-products satisfying  the preliminary assumptions of Theorem C and also    ${\bf(U)}$. By Proposition \ref{pression}, the map $q \mapsto \Delta (q)$ is continuous. Thus, up to restricting the parameter space,  we have  $\Delta(q_0)>\delta_{d,s}$ for any $q_0$. \newline

\noindent We now prove that $(G_{q_0})_{q_0}$ satisfies the property {\bf (T)} restricted to any $(\mathfrak{a},\alpha, \beta) \in  [\rho] \times [\rho'] \times [\rho'']$.  More precisely we show below the following technical lemma:

\begin{Lemma} \label{cochon}
Up to reducing $\mathcal{Q}$, for every $\mathfrak{a} \in [\rho]$, $\alpha \in [\rho']$, $\beta \in [\rho'']$, $p_0  \in \mathcal{P}$ and $r>0$, with $q:=(p,q’)$, the set of $q’_0$ s.t. $\mathrm{J}_{p_0}^s \pi_{q,\mathfrak{a}}(\alpha)$ and    $\mathrm{J}_{p_0}^s \pi_{q,\mathfrak{a}}(\beta)$ are $r$-close is of Lebesgue measure dominated by $r^{\delta_{d,s}}$, with an independent constant. Moreover for every  family of $\vartheta$-perturbations  of $(G_{q_0})_{q_0}$ with $\vartheta$ small enough, $t \in \mathcal{T}$, $\mathfrak{a} \in [\rho]$, $\alpha \in [\rho']$, $\beta \in [\rho'']$, $p_0  \in \mathcal{P}$ and $r>0$, the set of $q’_0$ s.t. $\mathrm{J}_{p_0}^s \tilde{\pi}_{t,q,\mathfrak{a}}(\alpha)$ and    $\mathrm{J}_{p_0}^s    \tilde{\pi}_{t,q,\mathfrak{a}}(\beta)$ are $r$-close is of Lebesgue measure dominated by $r^{\delta_{d,s}}$, with the same constant.
\end{Lemma}

\noindent We recall that the $s$-jets at $p_0$ are taken while varying $p$ around $p_0$ for a fixed value of $q’$ equal to $q’_0 $. In particular, by the Fubini Theorem,  the set of $(p_0,q’_0)$ s.t. $\mathrm{J}_{p_0}^s \pi_{q,\mathfrak{a}}(\alpha)$ and    $\mathrm{J}_{p_0}^s \pi_{q,\mathfrak{a}}(\beta)$ are $r$-close is also of  measure dominated by $r^{\delta_{d,s}}$ and the same holds for $\vartheta$-perturbations. The proof of Lemma \ref{cochon} is below. We first finish the proof of Proposition \ref{keylemme3}. \newline

\noindent We extend iteratively the initial family $(\mathcal{F}_p)_p$ a finite number of times by the same method. At each step, we extend by adding $\delta_{d,s}$ new parameters corresponding to a new product  of cylinders $[\rho] \times [\rho'] \times [\rho'']$ in the finite covering of $ \Omega$  defined in Step 1. The adaptation of Lemma \ref{cochon} is straightforward. 
 This proves property {\bf (T)} and so     $(G_{q_0})_{q_0 \in \mathcal{Q}}$ is a  family  of $C^2$-skew-products, with  $\mathcal{Q}:=\mathcal{P}  \times (-1,1)^{m}$ for some  $m>0$ (up to rescaling). This   concludes the proof. \end{proof}

\begin{proof} [Proof of Lemma \ref{cochon}]
We first prove the result for the Case 1 in Step 1  of the proof of Proposition \ref{keylemme3}. 
Let $\Pi_{q,\mathfrak{a}}(\alpha)$ and    $\Pi_{q,\mathfrak{a}}(\beta)$ be the points of the phase  space $\mathbb{R}^2$ equal to $\pi_{q,\mathfrak{a}}(\alpha)$ and    $\pi_{q,\mathfrak{a}}(\beta)$ in the parametrization of $W^{\mathfrak{a}}_q$. Both belong to $W^{\mathfrak{a}}_q$. \newline 

\noindent 
We pick local coordinates  in a    neighborhood of $\Pi_{0,\mathfrak{a}}(\alpha)$ centered at  $\Pi_{0,\mathfrak{a}}(\alpha) $ with a basis given by $\mathbb{R} \cdot e_0^s+\mathbb{R} \cdot e_0^u$, where we denoted by $\mathbb{R} \cdot e_0^s$ and $\mathbb{R} \cdot e_0^u$ the stable and unstable directions of $\overleftrightarrow{\mathcal{K}_0}$ at   $\Pi_{0,\mathfrak{a}}(\alpha)$. We write:
$$ \mathrm{J}_{p_0}^s \Pi_{q,\mathfrak{a}}(\alpha) =: \mathcal{J}_{p_0}^s \Pi_{q,\mathfrak{a}}(\alpha) \cdot e_0^s + \mathcal{J}_{p_0}^u \Pi_{q,\mathfrak{a}}(\alpha)  \cdot e_0^u \, . $$ 
\noindent Take care that the s on the left hand term is the order of the jet and the s on the right hand term  just means  « stable ». 
We can proceed similarly for $\Pi_{0,\mathfrak{a}}(\beta)$. \newline



\noindent
{\bf 1. Easy case: 0-jets.} We first perform the proof for 0-jets to show the general idea. We fix the parameter $p_0=0$.
We begin by studying $\Pi_{q,\mathfrak{a}}(\alpha)$, more precisely the variations of the  0-jet of $\Pi_{(0,q’),\mathfrak{a}}(\alpha)$, i.e.   $\Pi_{(0,q’),\mathfrak{a}}(\alpha)$  itself, when moving $q’$. In order to do this, we set 
$$\Pi_0(q’):=\Pi_{(0,q’),\mathfrak{a}}(\alpha)$$ and denote by $\Pi_k(q’)$ its preimage by $\mathcal{F}^k_{(0,q’)}$ on $W^{\alpha_{ |  k }  \mathfrak{a}}_{(0,q’)}$ (this is the point equal to $\pi_{(0,q’),  \alpha_{ |  k }  \mathfrak{a}}(\sigma^k ( \alpha ))$ in the parametrization of the stable manifold).  We pick local coordinates centered at each $P_k:=\Pi_k(0)$ with a basis given by the corresponding preimages $\mathbb{R} \cdot e_k^s$ and $\mathbb{R} \cdot e_k^u$ of $\mathbb{R} \cdot e_0^s$ and $\mathbb{R} \cdot e_0^u$ by $\mathcal{F}^k_0$ for $k>0$. These are the stable and unstable directions of $\overleftrightarrow{\mathcal{K}_0}$ at   $P_k$. In the decomposition $\mathbb{R} \cdot e_k^s+\mathbb{R} \cdot e_k^u$, we write
 $$            \Pi_k(q’)  =:      \Pi^s_k(q’)   \cdot e_k^s +    \Pi^u_k(q’)   \cdot e_k^u  \, .  $$ \noindent  In the coordinates given by   $P_{k+1} + \mathbb{R} \cdot e_{k+1}^s + \mathbb{R} \cdot e_{k+1}^u$ and $P_k + \mathbb{R} \cdot e_k^s+\mathbb{R} \cdot e_{k}^u$ with $P_{k+1}=P_k=(0,0)$, the map $\mathcal{F}_0$ restricted to a neighborhood of $P_{k+1} = (0,0)$ sends $P_{k+1}=(0,0)$ to $P_k=(0,0)$ and is $C^1$-close to its differential which is diagonal. In particular, the $(1,1)$-coefficient is a real number $\lambda_k$ s.t. $| \lambda_k| < \lambda<1$. \newline

\noindent By hyperbolic  continuation, there exists    $B>0$ independent of $k$ and $M$ s.t.:
$$   C_k := \left| \frac{d  \Pi^s_k }{dq’} (0)    \right| < B  \, .$$
\noindent We recall that $\Pi_{k+1} (q’)$ is sent onto $\Pi_k(q’)$ by $\mathcal{F}_{(0,q’)}$. Moreover $\mathcal{F}_{(0,q’)} (\Pi_{k+1} (q’))$ is the sum of $\mathcal{F}_{0} (\Pi_{k+1} (q’))$ and a term $\Sigma_{k}(q’) = \Sigma^s_k(q’) \cdot e_k^s+       \Sigma^u_k(q’)  \cdot            e_k^u$  coming from Eq. $(\ref{pertureb})$. In particular, we notice that  we have: $$  \left| \frac{d \Sigma^s_k}{dq’} (0) \right|  < 2h \, .$$
\noindent We have: $$ \frac{d\Pi^s_{k}}{dq’} (0)= \lambda_k  \cdot \frac{d \Pi^s_{k+1}}{d q’} (0) + \frac{d  \Sigma^s_k}{d q’} (0) \, .$$
$$ C_k < \lambda_k \cdot C_{k+1} + 2h  \, .$$ 
\noindent We also notice that $\Sigma_k(q’)$ is equal to  zero when $0 \le k \le  M-1$ by the first item of Case 1. This gives for every $M’>M$:
$$ C_0 < B \cdot \lambda^{M’} + 2h \cdot (\lambda^{M’-1} + \cdots +\lambda^M) \, .$$
\noindent By taking $M’$ large, we get:
$$ C_0 \le 2  h \frac{ \lambda^{M}} {1-  \lambda} \, ,$$
\noindent which is small compared to $h$ by assumption. We finally get that the derivative of $ \Pi_{(0,q’),\mathfrak{a}}(\alpha) = \Pi_0(q’)$ at $q’=0$  in the stable direction $\mathbb{R} \cdot e_0^s$ is small compared to $h$:
$$\frac{d}{dq’} \Pi^s_{(0,q’),\mathfrak{a}}(\alpha) \text{  is small compared to  }  h \, . $$
The same holds true  when     replacing $\Pi_{(0,q’),\mathfrak{a}}(\alpha)$ by the preimage $\Pi’_1(q’)$ of $\Pi_{(0,q’),\mathfrak{a}}(\beta)$ by $\mathcal{F}_{(0,q’)}$ on $W^{\beta_{ -1 }  \mathfrak{a}}_{(0,q’)}$, using the second item of Case 1. By  Eq. $(\ref{pertureb})$ and since $ e_0^s$ is close to $ v_0^s(\Pi’_1(0))$, the derivative $\frac{d}{dq’} \Pi^s_{(0,q’),\mathfrak{a}}(\beta)$ is then  close to $h$. Thus:
$$ \frac{d}{dq’}  \big( \Pi^s_{(0,q’),\mathfrak{a}}(\alpha) - \Pi^s_{(0,q’),\mathfrak{a}}(\beta) \big) \text{   is  bounded away from }(0,0) \, .$$ 
\noindent In particular, up to reducing $\mathcal{Q}$,  the set of $q’_0$ s.t. $\mathrm{J}_{0}^0 \Pi_{q, \mathfrak{a}}(\alpha)$ and    $\mathrm{J}_{0}^0 \Pi_{q,\mathfrak{a}}(\beta)$ are $r$-close is of Lebesgue measure dominated by $r$ with an independent constant. We can proceed the same way for any $p_0$, and the domination constant is independent. Integrating using the Fubini Theorem, the set $(p_0,q’_0)$ s.t. $\mathrm{J}_{p_0}^0 \Pi_{q, \mathfrak{a}}(\alpha)$ and    $\mathrm{J}_{p_0}^0 \Pi_{q,\mathfrak{a}}(\beta)$ are $r$-close is of Lebesgue measure dominated by $r$, and thus it is also the case for $\mathrm{J}_{p_0}^0 \pi_{q, \mathfrak{a}}(\alpha)$ and    $\mathrm{J}_{p_0}^0 \pi_{q,\mathfrak{a}}(\beta)$. 
The statement about $\vartheta$-perturbations follows easily with the same arguments when $\vartheta$ is small. 
\newline

\noindent  {\bf 2. General Case: $s$-jets.} Now we go to the general but more difficult case of $s$-jets. This time, we vary also the parameter $p$. We fix $p_0=0$. 
We  investigate the differentials of $\mathrm{J}_{0}^s \Pi_{q,\mathfrak{a}}(\alpha)$ and    $\mathrm{J}_{0}^s \Pi_{q,\mathfrak{a}}(\beta)$ when derivating relatively to $q’$. 
 \newline  

\noindent   We keep the same coordinates for each $k$: we keep expressing the point $\Pi_k(p,q’)$ (depending also on $p$ this time) in the coordinates $P_k+\mathbb{R} \cdot e_k^s+\mathbb{R} \cdot e_k^u$ (independent of $(p,q’)$ small).  Let  $\mathcal{J}_{0}^s \Pi_k(q’)$ and $\mathcal{J}_{0}^u \Pi_k(q’)$ be the components in this basis  of the  $s$-jet of $\Pi_k(p,q’)$ at $p_0=0$ for $q’$ fixed. Note that by hyperbolic  continuation, there exists $B’>0$ independent of $k$ and $M$  s.t.:
$$   C’_k := \left| \frac{d  \mathcal{J}_{0}^s \Pi_k}{dq’} (0) \right|< B’  \, .$$

\noindent For $q’$ fixed  the family $(\Pi_k(p,q’))_p$ is the image of  $(\Pi_{k+1}(p,q’))_p$ by the family of maps $(\mathcal{F}_{(p,q’)})_p$. The family $(\mathcal{F}_{(p,q’)} (\Pi_{k+1} (p,q’)))_p$ is the sum of $(\mathcal{F}_{p} (\Pi_{k+1} (p,q’)))_p$ and of  the family $(\Sigma_{p,k})_p$ of perturbations coming from Eq. $(\ref{pertureb})$.  In particular: $$  \left| \frac{d \mathcal{J}_{0}^s \Sigma_{k,p}}{dq’} (0) \right| \preceq 2 \delta_{d,s} h  \, .$$

\noindent  Again $\Sigma_{k,p}(q’)$ is equal to  zero when $0 \le k \le  M-1$ by the second item of Case 1. However this time the action of $(\mathcal{F}_p)_p$ on $s$-jets is more complicated than on 0-jets. For each $k$, we write the Taylor expansion of $\mathcal{F}_p$ in $P_{k+1}$ of order $s$. Then we replace each of its coefficients by its $s$-jet in $p$, and the variables by $ \mathcal{J}_{0}^s \Pi_{k+1}(q’)$ and $\mathcal{J}_{0}^u \Pi_{k+1}(q’)$, and we expand this expression. This shows that
 $$ \mathcal{J}_{0}^s \Pi_{k}(q’) = \mathcal{M}_k \cdot \mathcal{J}_{0}^s \Pi_{k+1}(q’)+\mathcal{M}_k’ \cdot \mathcal{J}_{0}^u \Pi_{k+1}(q’)+ \mathcal{J}_{0}^s \Sigma_{k,p}(q’)\, .$$  
 
 \noindent The term $\mathcal{J}_{0}^s \Sigma_{k,p}(q’)$ has  its derivative at $q’=0$ bounded by $2 \delta_{d,s} h$, and its 
  coefficients in $p^i$ depends only on $q_i$. The matrix $\mathcal{M}_k$  is inferior triangular with all its diagonal coefficients equal to $\lambda_k$ with $|\lambda_k|<\lambda<1$. On the other hand, the matrix $\mathcal{M}_k’$  is inferior triangular with all its diagonal coefficients equal to $0$. Iterating, using the first  item of Step 1 and then taking the derivative at $q’=0$, we see that 
$$ \frac{d}{dq’} \mathcal{J}_{0}^s \Pi_{(p ,q’),\mathfrak{a}}(\alpha)$$

\noindent  is inferior triangular with all its diagonal coefficients bounded by 
$$ 2 \delta_{d,s} h   \frac{ \lambda^{M}} {1-  \lambda} \, ,$$
\noindent  which is small compared to $h$ by assumption on $M$. Thus for every $i = (i_1, \ldots , i_d)$ s.t. $\sum_k i_k \le  s$,  the  coordinate in $p^i$  of the  stable   component  of the $s$-jet $\mathrm{J}_{0}^s \Pi_{q,\mathfrak{a}}(\alpha)$   has  small derivative compared to $h$  while moving the parameter $q_i$ around 0. \newline

\noindent The same holds for the preimage $\Pi’_1(p,q’)$ of $\Pi_{q,\mathfrak{a}}(\beta)$ by $\mathcal{F}_{q}$. By  Eq. $(\ref{pertureb})$ and since $\mathbb{R} \cdot e_0^s$ is close to $\mathbb{R} \cdot v_0^s(\Pi’_1(0,0))$, for every $i = (i_1, \ldots , i_d)$ s.t. $\sum_k i_k \le  s$,  the  coordinate in $p^i$  of the  stable  component of the $s$-jet $\mathrm{J}_{0}^s \Pi_{q,\mathfrak{a}}(\beta)$   has a derivative close to $h$ (up to a non zero independent multiplicative constant) while moving the parameter $q_i$. Thus the  coordinate in $p^i$  of  $\mathrm{J}_{0}^s \Pi_{q,\mathfrak{a}}(\alpha)-\mathrm{J}_{0}^s \Pi_{q,\mathfrak{a}}(\beta)$ has a non zero derivative while moving $q_i$ around 0.  
The same still holds for other values of  $p_0$.  Then, up to reducing $\mathcal{Q}$, the set  of $(p_0,q’_0)$ s.t. $\mathrm{J}_{p_0}^s \Pi_{q, \mathfrak{a}}(\alpha)$ and    $\mathrm{J}_{p_0}^s \Pi_{q,\mathfrak{a}}(\beta)$ are $r$-close is of Lebesgue measure dominated by $r^{\delta_{d,s}}$, and thus it is also the case for  $\mathrm{J}_{p_0}^s \pi_{q, \mathfrak{a}}(\alpha)$ and    $\mathrm{J}_{p_0}^s \pi_{q,\mathfrak{a}}(\beta)$. This gives the result, using the Fubini Theorem. The statement about $\vartheta$-perturbations follows easily with similar  arguments. This ends the proof in Case 1.\newline

\noindent In Case 2, the proof is more simple. First take the original sequences $\alpha$, $\beta$ and $\mathfrak{a}$ of Case 2 around which the cylinders $[\rho_\alpha ]$, $[\rho_\beta ]$ and $[\rho_{\mathfrak{a}}]$ are centered. The  coordinate in $p^i$  of the  stable  component of the $s$-jet $\mathrm{J}_{0}^s \Pi_{q,\mathfrak{a}}(\beta)$   has still  a non zero derivative    while moving the parameter $q_i$ (this only needs the second  item of Case 2 and not the first item of Case 1 not present in Case 2).  
On the other hand, $\Pi_{q,\mathfrak{a}}(\alpha)$ is the image of $\Pi_{q,\mathfrak{a}}(\beta)$ by some iterate of $\mathcal{F}_q$ restricted on the $M$ first images of $W^{\mathfrak{a}}_q$ since $\mathfrak{a}$ is periodic of period $\mathfrak{p}$ and the last letters of $\alpha$ are of the form $\beta_{-1}  \cdot \mathfrak{p}^f$ for some  $f>0$. These local stable manifolds do not depend on  $q’$ by the first item of Case 2. Moreover the action of $(\mathcal{F}_q)_q$ restricted to these stable manifolds on stable components of jets  is inferior triangular with diagonal coefficients between 0 and 1. Thus the  coordinate
 in $p^i$  of  $\mathrm{J}_{0}^s \Pi_{q,\mathfrak{a}}(\alpha)-\mathrm{J}_{0}^s \Pi_{q,\mathfrak{a}}(\beta)$ has again a non zero derivative while moving $q_i$. Since the lengths of the cylinders are large, this remains true for any $(\alpha, \beta, \mathfrak{a}) $ in $[\rho_\alpha ] \times [\rho_\beta ] \times [\rho_{\mathfrak{a}}]$.
 We conclude as in Case 1, which ends the proof.
\end{proof}

\begin{Remark} \label{keyrem3}
When $\ell =1$, to give a relative movement  inside the $\mathfrak{a}$-fiber   to the points encoded by $\alpha$ and $\beta$, we  considered the preimages of these two points by $\mathcal{F}_p$ (respectively on  $W^{\alpha_{-1} \mathfrak{a}}_p$ and $W^{\beta_{-1} \mathfrak{a}}_p$). The second one was not periodic and distinct from the first one. We then perturbed $\mathcal{F}_p$ in a neighborhood of this second preimage. \newline

\noindent
In the case when $\ell >1$, we shall look at the $\ell$  respective successive preimages of these two points by $\mathcal{F}_p$ (respectively on $W^{\alpha_{-1} \mathfrak{a}}_p$ and $W^{\beta_{-1} \mathfrak{a}}_p$ and their $\ell-1$ successive images  by $\mathcal{F}_p$) and take the first pair of preimages which are distinct. One of them is not periodic, and we perform the same perturbation as before in a neighborhood of this point. Then the proof is the same with the same  distinction in two cases  whether the $M$ successive  preimages of the other preimage intersect this neighborhood or not. 
\end{Remark}

   \section{Appendix}

   \subsection*{Proofs of Intermediate Lemmas \ref{polynome}, \ref{hol} and \ref{xy}}

     \begin{proof}[Proof of Lemma \ref{polynome}]
Let us fix $p \in  \overline{\mathcal{P}}$, $\mathfrak{a} \in  \overrightarrow{ \mathcal A}$,  $x \in X$ and $\alpha = (\alpha_{-n}, \cdots, \alpha_{-1}) \in \mathcal{A}^n$ for some $n >0$. We notice that the differential  $D\psi_{p,\mathfrak{a}}^\alpha(x)$ can be written as the product of $n$ factors:
\begin{equation} \label{lemme1eq11}
D\psi_{p,\mathfrak{a}}^\alpha(x)=    \prod_{k=1}^n  D\psi_{p, \mathfrak{a}_k}^{\alpha_{-k}} \big( \psi_{p, \mathfrak{a}_{k+1} }^{\alpha_{-k-1}} \circ \cdots \circ \psi_{p, \mathfrak{a}_n }^{\alpha_{-n}}(x) \big)  \text{ with }   \mathfrak{a}_k:=\alpha_{|  k-1} \mathfrak{a} 
\end{equation}
By ${\bf (U)}$, each of these $n$ factors is unipotent inferior  and then can be written as a sum of $N$ terms:
\begin{equation} \label{lemme1eq111}
D\psi_{p, \mathfrak{a}_k}^{\alpha_{-k}} \big( \psi_{p, \mathfrak{a}_{k+1} }^{\alpha_{-k-1}} \circ \cdots \circ \psi_{p, \mathfrak{a}_n }^{\alpha_{-n}}(x) \big) = M_{0}  + M_1 + \cdots + M_{N-1} \, .
\end{equation} 
Here $M_{0}$ is a diagonal matrix which has all its diagonal coefficients of absolute value equal to the coefficient  $\lambda_{p, \mathfrak{a}_k,\alpha_{-k}} \big( \psi_{p, \mathfrak{a}_{k+1} }^{\alpha_{-k-1}} \circ \cdots \circ \psi_{p, \mathfrak{a}_n }^{\alpha_{-n}}(x) \big)$. On the other side, for every $1 \le k \le N-1$, all the coefficients of $M_k$ are equal to 0 except possibly on the $k^{th}$ (small) diagonal  line below the (great) diagonal line. \medskip

\noindent   We can write each of the $n$ factors of Eq. $(\ref{lemme1eq11})$ as in Eq. $(\ref{lemme1eq111})$ and then expand $D\psi_{p,\mathfrak{a}}^\alpha(x)$ as a sum of $N^n$ factors of $n$ terms. Among them, any such product with $\ge N$ matrices having zero coefficients on and above  the great diagonal line vanishes. Thus, between the $N^n$ terms whose sum equals  $D\psi_{p,\mathfrak{a}}^\alpha(x)$, there are at most 
   \begin{equation} \label{lemme1eq11111}
  \tilde{P}(n):= {n \choose N-1} \cdot (N-1)^{N-1} + {n \choose N-2} \cdot (N-1)^{N-2}  + \cdots {n \choose 1}  \cdot (N-1) + 1
  \end{equation}     
  which are non zero, with $\tilde{P}$ polynomial. Each of these $\le \tilde{P}(n)$ terms $\mathcal{M}_1 \cdots \mathcal{M}_n$ is a product of $n$ factors $\mathcal{M}_k$. \medskip
  
\noindent  Let us consider such a product $\mathcal{M}_1 \cdots \mathcal{M}_n$. At most $N-1$ of the $\mathcal{M}_k$  have all their coefficients equal to 0 except possibly on one of the small diagonal lines below  the (great) diagonal line. The $\ge n-(N-1)$ other factors are all diagonal matrices. Each coefficient of the resulting product $\mathcal{M}_1 \cdots \mathcal{M}_n$ is then either zero or equal to the product of $n$  non zero coefficients $c_k$, with $c_k$ a non zero coefficient of $\mathcal{M}_k$. If $\mathcal{M}_k$ is diagonal, $|c_k|$ is equal to $\lambda_{p, \mathfrak{a}_k,\alpha_{-k}} \big( \psi_{p, \mathfrak{a}_{k+1} }^{\alpha_{-k-1}} \circ \cdots \circ \psi_{p, \mathfrak{a}_n }^{\alpha_{-n}}(x) \big)$. If not, $|c_k|$ is bounded by some independent constant $C_1$ since for every $a \in \mathcal{A}$, the $C^1$-bounded map $\psi_{p, \mathfrak{a} }^a$ depends continuously   on $p \in \overline{\mathcal{P}}$ and $\mathfrak{a} \in  \overrightarrow{ \mathcal A}$.  By Eq. $(\ref{produitchinne})$, we have:
  \begin{equation} 
   \lambda_{p,\mathfrak{a},\alpha}(x)  = \prod_{k=1}^{n}       \lambda_{p, \mathfrak{a}_k, \alpha_{-k}} \big( \psi_{p, \mathfrak{a}_{k+1} }^{\alpha_{-k-1}} \circ \cdots \circ \psi_{p, \mathfrak{a}_n }^{\alpha_{-n}}(x) \big)  \text{ with }  \mathfrak{a}_k:=\alpha_{|  k-1} \mathfrak{a}   \, .
   \end{equation} 

\noindent   Moreover, each of the coefficients $  \lambda_{p, \mathfrak{a}_k, \alpha_{-k}} \big( \psi_{p, \mathfrak{a}_{k+1} }^{\alpha_{-k-1}} \circ \cdots \circ \psi_{p, \mathfrak{a}_n }^{\alpha_{-n}}(x) \big) $ is larger than $\gamma'$. 
 Thus the resulting coefficient of  $\mathcal{M}_1 \cdots \mathcal{M}_n$ is smaller than $\lambda_{p,\mathfrak{a},\alpha}(x) \cdot (\gamma’)^{-N+1} \cdot C^{N-1}_{1}$ in modulus. Thus  any coefficient of $D\psi_{p,\mathfrak{a}}^\alpha(x)$ is bounded by $P(n) \cdot \lambda_{p,\mathfrak{a},\alpha}(x)$, where $P(n):= \tilde{P}(n) \cdot (\gamma’)^{-N-1}  \cdot C^{N-1}_{1}$ is a positive polynomial on $\mathbb{R}_+$.  \end{proof}

   \begin{proof}[Proof of Lemma \ref{xy}] 
For every $p \in \overline{\mathcal{P}}$, $\mathfrak{a} \in \overrightarrow{\mathcal A}$, $n \ge 0$, $\rho \in \mathcal{A}^n$ and $(\alpha, \beta) \in  \mathcal C_\rho$, the points $ \pi_{p,\mathfrak{a}} (\alpha)$ and  $\pi_{p,\mathfrak{a}} (\beta)$ are the respective images of $\pi_{p,\rho  \mathfrak{a}} (\sigma^n (\alpha))$ and $\pi_{p,\rho \mathfrak{a}}(\sigma^n(\beta))$ by $\psi_{p,\mathfrak a}^\rho$. Let us denote by $v = (v_1, \cdots , v_N) $ the vector $v:=\pi_{p, \rho  \mathfrak{a}} (\sigma^n(\alpha))-\pi_{p,\rho  
\mathfrak{a}}(\sigma^n(\beta))$. We denote by $j \in \{1, \cdots, N\}$ the maximal index s.t.  $|v_j | > 2N \cdot P(n) \cdot |v_i |$ for every $i<j$, where the polynomial $P$ was defined in Lemma \ref{polynome}. Using this, it follows:
  \begin{equation} \label{lemme1eq3}
  |v_j|  \ge F \cdot  (2N P(n) )^{-N} \cdot ||  v ||   
  \end{equation}
  for some positive constant $F$ (depending only on $N$). The segment between the two points $\pi_{p, \rho \mathfrak{a}} (\sigma^n(\alpha))$ and $\pi_{p,\rho \mathfrak{a}}(\sigma^n(\beta))$ is fully included in $X$ since $X$ is convex. Let $\chi: [0,1] \rightarrow \mathbb{R}$ be the $C^1$-map which sends $x \in [0,1]$ to the $j^{th}$ coordinate of $\psi_{p,\mathfrak a}^\rho(\pi_{p,\rho \mathfrak{a}}(\sigma^n(\beta))+xv )$.  By the mean value equality, there exists $x \in (0,1)$ s.t.:
   \begin{equation} \label{lemme1eq57}
  \chi(1) -\chi(0) = \chi'(x) = \sum_{i=1}^{ j}  a_{j,i} \cdot v_i 
  \end{equation} 
where $a_{j,i}$ is the coefficient of index $(j,i)$ of the differential $D \psi_{p,\mathfrak a}^\rho (y) $ with $y:=  \pi_{p,\rho \mathfrak{a}}(\sigma^n(\beta)) + xv $. The right-hand equality is due to the fact that   $D\psi_{p,\mathfrak a}^\rho( y )$ is unipotent inferior by ${\bf (U)}$. We notice that $| a_{j,j}  |= \lambda_{p,\mathfrak a,\rho} (y) $.  By Lemma \ref{polynome},  $|  a_{j,i}  |$ is smaller than $P(n) \cdot \lambda_{p,\mathfrak a,\rho} (y)$. By Eq. $(\ref{lemme1eq57})$, we then have:
   \begin{equation} \label{lemme1eq577}
  |\chi(1) -\chi(0)|  \ge     \lambda_{p,\mathfrak a,\rho} (y)   \cdot   |v_j |  -    \sum_{i<j}  P(n) \cdot \lambda_{p,\mathfrak a,\rho} (y) \cdot   \frac{| v_j |  }{ 2N \cdot P(n)}  \ge     \frac{ \lambda_{p,\mathfrak a,\rho} (y) \cdot | v_j |    }{2} \,  .
  \end{equation} 
Noticing that $|| \pi_{p,\mathfrak{a}} (\alpha)-\pi_{p,\mathfrak{a}} (\beta) ||  \ge   |\chi(1) -\chi(0)| $ and injecting Eq. $(\ref{lemme1eq3})$ in Eq. $(\ref{lemme1eq577})$, it then holds:
   \begin{equation} \label{lemme1eq31}
 || \pi_{p,\mathfrak{a}} (\alpha)-\pi_{p,\mathfrak{a}} (\beta) ||  \ge   \frac{F}{2} \cdot  (2N P(n) )^{-N} \cdot \lambda_{p,\mathfrak a,\rho } (y) \cdot  ||   \pi_{p,\rho  \mathfrak{a}} (\sigma^n(\alpha)) - \pi_{p,\rho \mathfrak{a}}(\sigma^n(\beta)) ||   \, . 
   \end{equation}
  By Lemma \ref{bdsk} and noting $R(n):= \frac{2D_1}{F}   \cdot (2N P(n) )^{N} $ which is positive on $\mathbb{R}_+$, we have:
   \begin{equation} \label{lemme1eq317}
 || \pi_{p,\mathfrak{a}} (\alpha) - \pi_{p,\mathfrak{a}} (\beta) ||  \ge      \frac{\Lambda_{p,\mathfrak a,\rho}}{R(n)}   \cdot  ||   \pi_{p,\rho  \mathfrak{a}} (\sigma^n(\alpha)) -   \pi_{p,\rho \mathfrak{a}}(\sigma^n(\beta) )||   \, . 
   \end{equation}
   \noindent The proof of the second item is similar and we apply Lemma \ref{poisson} to conclude. 
   \end{proof}

   \begin{proof}[Proof of Lemma \ref{hol}] We notice that the coefficient $  \lambda_{p, \mathfrak{a}, \alpha_{-1}}  \big( \pi_{p, \alpha_{-1} \mathfrak{a}   }(\sigma (\alpha))  \big)$ is positive and uniformly distant from 0 and $+\infty$. Since $\log$ is $C^1$ on $]0,+\infty[$, we just have to show that the following map  is  H\"older with positive exponent on its domain:
$$ (\alpha, \mathfrak{a} )     \mapsto  \lambda_{p, \mathfrak{a}, \alpha_{-1}}  \big(   \pi_{p, \alpha_{-1} \mathfrak{a}   }(\sigma (\alpha))  \big) \, .$$ Let us recall that the latter is the $(1,1)$ coefficient of $   D\psi_{p, \mathfrak{a} }^{ \alpha_{-1}}( \pi_{p, \alpha_{-1} \mathfrak{a}   }(\sigma (\alpha)) ) $ (up to the sign). By assumption the map $\mathfrak{a} \in \overrightarrow{ \mathcal A} \mapsto Df_{p,\mathfrak{a}}$ is H\"older for the $C^0$-topology and so it is enough to show that the map $(\alpha, \mathfrak{a} ) \in   \overleftarrow{ \mathcal A} \times \overrightarrow{ \mathcal A}     \mapsto \pi_{p,\mathfrak{a}} (\alpha) \in X$ is itself H\"older. By hyperbolicity,  the map $\alpha \in   \overleftarrow{ \mathcal A}     \mapsto \pi_{p,\mathfrak{a}} (\alpha) \in X$ is H\"older for any $\mathfrak{a}  \in  \overrightarrow{ \mathcal A}$, with exponent and constant independent of $ \mathfrak{a}   $. Thus it is enough to show that  the  map  $\mathfrak{a}  \in  \overrightarrow{ \mathcal A}    \mapsto \pi_{p,\mathfrak{a}} (\alpha) \in X$ is H\"older for any $\alpha \in  \overleftarrow{ \mathcal A}$, with independent constants. But using again both the hyperbolicity and that $\mathfrak{a}  \mapsto  f_{p,\mathfrak{a}}$ and  $\mathfrak{a}  \mapsto D f_{p,\mathfrak{a}}$ are H\"older for the $C^0$-topology, we see that the maps $\mathfrak{a} \mapsto \psi_{p,\mathfrak{a}}^{\alpha_n}(0)$ are H\"older, with exponent and constant independent of $p$, $\alpha$ and $n$. But this sequence converges uniformly to the map $\mathfrak{a}   \mapsto \pi_{p,\mathfrak{a}} (\alpha)$, which concludes the proof. 
 \end{proof}

   \subsection*{Proof of Distortion Lemmas \ref{bdsk}, \ref{gfsk}, \ref{fisk} and \ref{poisson}}

  \begin{proof}[Proofs of Lemma \ref{bdsk}]
  The Lemma will follow easily from the two following:
  
  \begin{Sublemma}\label{sub11}
  There exists $A>0$ s.t. for any  $p \in \overline{\mathcal{P}}$, $\mathfrak{a} \in \overrightarrow{ \mathcal A}$, $a \in \mathcal{A}$, it holds: 
  $$  1-A |x-y|    \le    \frac{   \lambda_{p,\mathfrak{a},a}(x) }{\lambda_{p,\mathfrak{a}, a}  (y) } \le 1 + A  |x-y|  \hspace{0.5cm} \forall \text{ } x,y \in X   \, . $$  
  \end{Sublemma} 
  \begin{proof} The non zero number $  \lambda_{p,\mathfrak{a},a}(x)$ is the $(1,1)$ coefficient of the differential $   D\psi_{p, \mathfrak{a} }^a(x) $ (up to the sign). Since the maps $\psi_{p,\mathfrak{a} }^a$ are uniformly (in $p$, $\mathfrak{a} $ and $a$) $C^2$ bounded, denoting by $\tilde{A}$   a uniform bound of the second differential of $\psi_{p,\mathfrak{a} }^a$ on $X$ among $p \in \overline{\mathcal{P}}$, $\mathfrak{a} \in \overrightarrow{ \mathcal A}$, $a \in \mathcal{A}$,  the coefficient $  \lambda_{p,\mathfrak{a},a}(x)$ is between $     \lambda_{p,\mathfrak{a},a}(y) -\tilde{A} |x-y|$ and $    \lambda_{p,\mathfrak{a},a}(y) + \tilde{A} |x-y|$. We notice that $\gamma' <   \lambda_{p,\mathfrak{a},a}(y) $. Denoting $A:= \tilde{A} /\gamma'$ and taking the quotient, we get the desired inequality.
  \end{proof}

 \begin{Sublemma}\label{sub12} 
  There exists $A'>0$ such that for every $p \in \overline {\mathcal{P}}$, $
  \mathfrak{a} \in \overrightarrow{ \mathcal A}$, $n \ge 0$, $\alpha \in \mathcal{A}^n$, $x,y \in X$, the points $\psi_{p, \mathfrak{a} }^\alpha(x)$ and $\psi_{p, \mathfrak{a} }^\alpha(y)$ are $A' \cdot \gamma^{n/2}$ distant. 
  \end{Sublemma}
  
  \begin{proof}
  It is an immediate consequence of the inequality $\Lambda_{ p,  \mathfrak{a} , \alpha}    < \gamma^{|\alpha|}$   and of Lemma \ref{polynome} that the diameter of $\psi_{p, \mathfrak{a} }^\alpha(X)$ is dominated by $\gamma^{|\alpha|/2}$. 
  \end{proof} 
  
  \noindent
  We can now conclude.  Using Eq. $(\ref{produitchinne})$ we write both $ \lambda_{p,\mathfrak{a},\alpha}(x)$ and $ \lambda_{p,\mathfrak{a},\alpha}(y)$ as products of $n$ factors and thus their quotients as:
  $$ \frac{   \lambda_{p,\mathfrak{a},\alpha}(x) }{\lambda_{p,\mathfrak{a},\alpha}  (y) } = \prod_{k=1}^{n}    \frac{   \lambda_{p, \mathfrak{a}_k, \alpha_{-k}} \big( \psi_{p, \mathfrak{a}_{k+1} }^{\alpha_{-k-1}} \circ \cdots \circ \psi_{p, \mathfrak{a}_n }^{\alpha_{-n}}(x) \big) }{   \lambda_{p, \mathfrak{a}_k, \alpha_{-k}} \big( \psi_{p, \mathfrak{a}_{k+1} }^{\alpha_{-k-1}} \circ \cdots \circ \psi_{p, \mathfrak{a}_n }^{\alpha_{-n}}(y) \big)  }    $$ 
  where we set again $\mathfrak{a}_k:=\alpha_{|  k-1}  \mathfrak{a}   \in \overrightarrow{ \mathcal A}$.
 Using the two Sublemmas, the previous quotient is between $\prod_{k=1}^{n}  (1-AA' \cdot \gamma^{n/2})$ and $\prod_{k=1}^{n}  (1+AA' \cdot \gamma^{n/2})$. Since $0<\gamma<1$, the infinite products $\prod  (1 \pm AA' \cdot \gamma^{n/2})$  converge and their limits are respectively in $(0,1)$ and $(1,+\infty)$, which concludes the proof.  \end{proof}

   \begin{proof}[Proof of Lemma \ref{gfsk}] 
   Let us fix $\eta>0$. By Sublemmas $(\ref{sub11})$ and  $(\ref{sub12})$, there exist $n_0 \in \mathbb{N}$ s.t. for every $p \in \overline{\mathcal{P}}$, $\mathfrak{a} \in \overrightarrow{ \mathcal A}$, $n > n_0$, $\alpha \in \mathcal{A}^n$, $a \in \mathcal{A}$ and $x,y \in \psi_{p, \mathfrak{a} }^\alpha(X)$, we have $e^{-\eta}<  \lambda_{p,\mathfrak{a},a}(x) /     \lambda_{p,\mathfrak{a}, a}  (y)< e^{\eta}$. We recall that the map $\psi_{p, \mathfrak{a} }^a$ depends continuously in the $C^2$-norm on $\mathfrak{a} \in \overrightarrow{ \mathcal A}$ and $p$ in $\overline{\mathcal{P}}$ (both compact sets). Thus there exists $\delta(\eta)>0$ such that for every $\mathfrak{a} \in \overrightarrow{ \mathcal A}$ and $p_1, p_2 \in \mathcal{P}$ with $||p_1-p_2|| < \delta(\eta)$, we have $e^{-\eta}<\lambda_{p_1,\mathfrak{a},a}(x) /     \lambda_{p_2,\mathfrak{a}, a}  (y)< e^{\eta}$ for every $n > n_0$, $\alpha \in \mathcal{A}^n$, $a \in \mathcal{A}$, $x \in \psi_{p_1,\mathfrak{a}}^\alpha(X)$ and $y \in \psi_{p_2,\mathfrak{a}}^\alpha(X)$. We denote by $D_2>0$ the maximum of the quotients $\lambda_{p_1,\mathfrak{a},\alpha}(x) /     \lambda_{p_2,\mathfrak{a},\alpha}  (y)$ among $p_1,p_2 \in \overline{\mathcal{P}}$,  $\mathfrak{a} \in \overrightarrow{ \mathcal A}$, $n \le n_0$,  $\alpha \in \mathcal{A}^{n}$ and $x,y \in X$. We conclude by noticing that for every $n>n_0$ and $x,y \in X
   $, we have:
   $$\frac{\lambda_{p_1,\mathfrak{a},\alpha}(x)}{ \lambda_{p_2,\mathfrak{a},\alpha}  (y)}    =       \frac{  \lambda_{p_1,\mathfrak{a}, \cdots \alpha_{-n_0}    \cdots \alpha_{-1}  }  (x)   }{  \lambda_{p_2,\mathfrak{a},  \alpha_{-n_0}   \cdots  \alpha_{-1}   }  (y) }  \cdot \prod_{k=n_0+1}^{n}    \frac{   \lambda_{p_1, \mathfrak{a}_k, \alpha_{-k}} \big( \psi_{p, \mathfrak{a}_{k+1} }^{\alpha_{-k-1}} \circ \cdots \circ \psi_{p, \mathfrak{a}_n }^{\alpha_{-n}}(x) \big)    }{   \lambda_{p_2, \mathfrak{a}_k, \alpha_{-k}} \big( \psi_{p, \mathfrak{a}_{k+1} }^{\alpha_{-k-1}} \circ \cdots \circ \psi_{p, \mathfrak{a}_n }^{\alpha_{-n}}(y) \big)  }   \, .$$    \end{proof}


\begin{proof}[Proof of Lemma \ref{fisk}] Let us fix $p \in \overline{\mathcal{P}}$. 
We saw in the proof of Lemma \ref{hol} that
$$ ( \mathfrak{a}  , \alpha )    \mapsto  \lambda_{p, \mathfrak{a}, \alpha_{-1}}  \big( \pi_{p, \alpha_{-1} \mathfrak{a}   }(\sigma (\alpha))  \big)$$ is  H\"older with positive exponent on its domain. We now fix an arbitrary $\beta \in \overleftarrow{ \mathcal A}$. Let us take $\mathfrak{a}, \mathfrak{a}’ \in  \overrightarrow{ \mathcal A} $. By Eq. $(\ref{produitchinne})$ and proceeding as in the proof of Lemma     \ref{bdsk}, we see that  the quotient $$ \frac{ \lambda_{p,\mathfrak{a},\alpha}(\pi_{p,\alpha \mathfrak{a}}(\beta)) } 
{\lambda_{p,\mathfrak{a}’,\alpha}  (\pi_{p,\alpha \mathfrak{a}’}(\beta))}$$ is bounded between 0 and $+\infty$, with constants independent of $p$, $\mathfrak{a}$, $\mathfrak{a}’$ and $\alpha$. To conclude, we just have to apply Lemma \ref{bdsk}. \end{proof}

    \begin{proof}[Proof of Lemma \ref{poisson}]

    By Sublemma  $(\ref{sub12})$, there exists $n_0 \in \mathbb{N}$ s.t. for every $p \in \overline{\mathcal{P}}$,  $n > n_0$, $\rho \in \mathcal{A}^n$, $a \in \mathcal{A}$ and $x,y \in \psi_{p, \mathfrak{a} }^\rho(X)$, we have $$ \lambda_{p,\mathfrak{a}, a}  (y)^{\epsilon’}<  \lambda_{p,\mathfrak{a},a}(x) <    \lambda_{p,\mathfrak{a}, a}  (y)^{1/\epsilon’}\, .$$
    \noindent 
     Then up to taking  $\vartheta$-perturbations for small $\vartheta$, for every $t \in \mathcal{T}$, $p \in \overline{\mathcal{P}}$,  $\mathfrak{a} \in \overrightarrow{ \mathcal A}$,  $n > n_0$, $\rho \in \mathcal{A}^n$, $a \in \mathcal{A}$, $x \in \psi_{t,p,\mathfrak{a}}^\rho(X)$ and $y \in \psi_{p,\mathfrak{a}}^\rho(X)$, we have
     $$ \lambda_{p,\mathfrak{a}, a}  (y)^{\epsilon’}<  \tilde{\lambda}_{t,p,\mathfrak{a},a}(x) <    \lambda_{p,\mathfrak{a}, a}  (y)^{1/\epsilon’}\, .$$
     \noindent Let $D_4>0$ be  the maximum of $\tilde{\lambda}_{t,p,\mathfrak{a},\rho}(x) /     \lambda^{\epsilon’}_{p,\mathfrak{a},\rho}  (y)$ and $  \lambda^{1/\epsilon’}_{p,\mathfrak{a},\rho}  (y)/\tilde{\lambda}_{t,p,\mathfrak{a},\rho}(x)$ among $t \in \mathcal{T}$, $ \mathfrak{a} \in \overrightarrow{ \mathcal A}$, $p \in \overline{\mathcal{P}}$,  $n \le n_0$,  $\rho  \in \mathcal{A}^{n}$, $x,y \in X$. We conclude as for Lemma \ref{gfsk}.  
\end{proof}

   \subsection*{Pressure function: Proof of Proposition \ref{pression}}

 \noindent  For simplicity, we fix $p \in  \overline{\mathcal{P}}$ and denote for every $ s \ge 0$: $$Z_n( s ):=\sum_{\alpha \in \mathcal A^n}  \Lambda_{p,\mathfrak{a}, \alpha}^s >0\, .$$
   For every $s \ge 0$, the sequence $n \in \mathbb{N}_+ \mapsto \mathrm{log} Z_n(s)$ is subadditive and so the limit $\Pi_{p,\mathfrak{a}}(s) = \mathrm{lim}_{n \rightarrow + \infty} \frac{1}{n} \mathrm{log} Z_n(s)$ exists and is finite by Fekete's Lemma.  
   We notice that $\Pi_{p,\mathfrak{a}} (0)$  is equal to the topological  entropy of the shift $\sigma$ which is positive since $\mathcal{A}$ has at least two letters. An immediate consequence of Lemma \ref{fisk} is that the pressure $\Pi_{p,\mathfrak{a}}$  only depends on $p$.  We denote it by $\Pi_p$.    We notice that $s \in \mathbb{R}_+ \mapsto Z_n(s)$ is log convex. Thus the map  $s \in \mathbb{R}_+ \mapsto   \mathrm{log} Z_n(s)$ is convex. The limit map $s \mapsto \Pi_{p    }(s)$  is then also convex and thus continuous. We remark that for $s,s’ \ge 0$, it holds: 
   $$  Z_n(s+s’)=   \sum_{\alpha \in \mathcal A^n}  \Lambda_{p,  \mathfrak{a}  ,  \alpha}^{s+s’} \le \sum_{\alpha \in \mathcal A^n }   \Lambda_{p,  \mathfrak{a}, \alpha}^{s} \cdot \gamma^{ns’} $$
   and then $\Pi_{p}(s+s’) \le   \Pi_{p}(s) + s’ \cdot   \mathrm{log} \gamma$. Since $ \mathrm{log} \gamma < 0$,  the map $ s   \in \mathbb{R}_+  \mapsto \Pi_{p}(s)$ is strictly decreasing. Tending $s \rightarrow + \infty$, we see that $\Pi_{p}(s)$ tends to $-\infty$. Finally, by the intermediate value theorem, the map $s \in \mathbb{R}_+ \mapsto \Pi_{p}(s)$ has a unique zero $\Delta(p)$.
   It remains to prove the continuity of $p \mapsto \Delta(p)$. By Eq. $(\ref{pii})$, we have:
   \begin{equation} \label{pir}
 \Pi_p(s)=\lim_{n \rightarrow + \infty} \frac{1}{n} \mathrm{log} \sum_{\alpha \in \mathcal A^n}  \Lambda_{p,\mathfrak{a},\alpha}^s \text{ for any  }  s\ge 0   \, .  
 \end{equation} By Lemma \ref{lemauxlilou} (whose proof does not need the continuity of $p \mapsto \Delta(p)$), we see that for every $p$ and $\epsilon’ >1$, there exists a neighborhood $\mathcal{U}_{p,\epsilon’}$ of $p$ and a constant $D_5>0$ s.t. for every $p’ \in \mathcal{U}_{p,\epsilon’}$ and $\alpha \in \mathcal{A}^*$,  the term $\Lambda_{p’,\mathfrak{a},\alpha} $ is bounded between $\Lambda_{p,\mathfrak{a},\alpha}^{\epsilon’} / D_5$ and   $D_5 \Lambda_{p,\mathfrak{a},\alpha}^{1/\epsilon’} $. Injecting this in Eq $(\ref{pir})$, this implies that $\Pi_{p}(s \epsilon’) \le \Pi_{p’}(s) \le \Pi_{p}(s/\epsilon’)$ for every $s \ge 0$ and $p’ \in \mathcal{U}_{p,\epsilon’}$ and so $\Delta(p) /\epsilon’ \le \Delta(p’) \le \Delta(p)\epsilon’$ for every $p’ \in \mathcal{U}_{p,\epsilon’}$. This proves the continuity of $p \mapsto  \Delta(p)$.

  

      \subsection*{Hyperbolicity  theory}

  We recall here some background on hyperbolic compact sets of $C^r$-endomorphisms. This subsection is mainly taken from Section 1    of the article \cite{b9} of Berger and from  Appendix C of the article \cite{bbb137} of Berger and the author. \newline 
  
  \noindent Let $\mathcal{M}$ be a manifold. A subset $\mathcal{K} \subset \mathcal{M}$ is left invariant by a  $C^1$-endomorphim $\mathcal{F}$ from $\mathcal{M}$ into $\mathcal{M}$ if $\mathcal{F} (\mathcal{K} ) =\mathcal{K} $. When $\mathcal{F}$ is a diffeomorphism, the invariant compact set $\mathcal{K} \subset \mathcal{M}$ is hyperbolic if there exists a $D\mathcal{F}$-invariant splitting $T\mathcal{M} | \mathcal{K} = \mathcal{E}^s \bigoplus  \mathcal{E}^u$ so that $\mathcal{E}^s$ is contracted by $D\mathcal{F}$ and $\mathcal{E}^u$ is expanded by $D\mathcal{F}$:
   \begin{equation} 
  \exists \lambda < 1, \text{ }  C>0, \text{  }  \forall k \in \mathcal{K}, \text{ } \forall n \ge 0, \text{ }   || D\mathcal{F}^n | \mathcal{E}^s_k  || \le C \lambda^n \text{ and }  || (D\mathcal{F}^n  |  \mathcal{E}^u_k)^{-1}  ||  \le C \lambda^n \, .    \nonumber 
   \end{equation} 
   When $\mathcal{F}$ is a local diffeomorphism, we shall study the inverse limit $ \overleftrightarrow{\mathcal{K}}_\mathcal{F}$ of $ \mathcal{K} $:
\begin{equation}
\overleftrightarrow{\mathcal{K}}_\mathcal{F} :=\{ (k_i)_i \in \mathcal{K}^{\mathbb{Z}} : \mathcal{F}(k_i) = k_{i+1} , \forall i \in \mathbb{Z} \} \, . \nonumber
\nonumber 
\end{equation} 
It is a compact space for the topology induced by the product one of $\mathcal{K}^{\mathbb{Z}}$. The dynamics induced by $\mathcal{F}$ on $\overleftrightarrow{\mathcal{K}}_\mathcal{F}$ is the shift $\overleftrightarrow{\mathcal{F}}$ and is invertible. Let $\pi  : \overleftrightarrow{\mathcal{K}}_\mathcal{F} \rightarrow \mathcal{K}$ be the zero-coordinate projection.  Let $\pi^* T\mathcal{M}$ be the bundle over $\overleftrightarrow{\mathcal{K}}_\mathcal{F}$ whose fiber at $\underline{k}$ is $T_{\pi(\underline{k})} \mathcal{M}$. The map $D\mathcal{F}$ acts canonically on $\pi^* T\mathcal{M}$ as $\overleftrightarrow{\mathcal{F}}$ on the basis and as the linear map $D_{\pi(\underline{k})} \mathcal{F}$ on the fiber of $\underline{k} \in  \overleftrightarrow{\mathcal{K}}_\mathcal{F}$. \newline
 
\noindent The compact set $\mathcal{K}$ (or $\overleftrightarrow{\mathcal{K}}_\mathcal{F}$) is hyperbolic if there exists a $D\mathcal{F}$-invariant splitting $\pi^* T\mathcal{M} = \mathcal{E}^s \bigoplus  \mathcal{E}^u$  s.t. $\mathcal{E}^s_{\underline{k}}$ is contracted by  $D_{\pi(\underline{k})} \mathcal{F}$ and $\mathcal{E}^u_{\underline{k}}$ is expanded by  $D_{\pi(\underline{k})} \mathcal{F}$:
       \begin{equation} 
  \exists \lambda < 1, \text{ }  C>0, \text{  }  \forall \underline{k} \in  \overleftrightarrow{\mathcal{K}}_\mathcal{F}, \text{ } \forall n \ge 0, \text{ }   || D\mathcal{F}^n | \mathcal{E}^s_{\underline{k}}  || \le C \lambda^n \text{ and }  || (D\mathcal{F}^n  |  \mathcal{E}^u_{\underline{k}})^{-1}  ||  \le C \lambda^n \, .    \nonumber 
   \end{equation} 
   
   \noindent Actually the definition of hyperbolicity for local diffeomorphisms is consistent with the definition of hyperbolicity for diffeomorphisms when the dynamics is invertible. Here is a useful  result about structural stability:
   
   \begin{theoremcontinu}
   Let $\mathcal{K}$ be a hyperbolic set for a $C^1$-local diffeomorphism $\mathcal{F}$ of $\mathcal{M}$. Then for every $C^1$-local diffeomorphism  $\mathcal{F}’$ which is $C^1$-close to $\mathcal{F}$, there exists a continuous map $i_{\mathcal{F}’} : \overleftrightarrow{\mathcal{K}}_\mathcal{F} \rightarrow \mathcal{M}$ which is $C^0$-close to $\pi$ and so that:
   \begin{enumerate} 
   \item $i_{\mathcal{F}’} \circ \overleftrightarrow{\mathcal{F}}= \mathcal{F}’ \circ i_{\mathcal{F}’}$, 
   \item  $\mathcal{K}_{\mathcal{F}’} := i_{\mathcal{F}’} (\overleftrightarrow{\mathcal{K}}_\mathcal{F}) $ is hyperbolic for $\mathcal{F}’$.
   
   \end{enumerate}

   \end{theoremcontinu}
   
  \noindent Let us also recall the definition of a stable manifold in this context.  For every $k \in \mathcal{K}$ and $\eta>0$, we define the stable manifold and local stable manifold of $k$ by:
    \begin{equation} 
W^s(k ; \mathcal{F}):= \{ k’ \in \mathcal{M}  : d(\mathcal{F}^n(k),\mathcal{F}^n(k’)) \underset{n\to+\infty}{\longrightarrow} 0 \}  \, .    \nonumber 
   \end{equation} 
     \begin{equation} 
W^s_\eta(k ; \mathcal{F}):= \{ k’ \in \mathcal{M}  : \eta>  d(\mathcal{F}^n(k),\mathcal{F}^n(k’)) \underset{n\to+\infty}{\longrightarrow} 0 \}  \, .    \nonumber 
   \end{equation} 
   \noindent
   For $\underline{k} \in  \overleftrightarrow{\mathcal{K}}_\mathcal{F} $  and $\eta>0$, the unstable manifold and local unstable manifold of $\underline{k}$ are:
   $$W^u(\underline{k}, \mathcal{F} ) = \{ k_0’ \in  \mathcal{M}   :  \exists (k’_i)_{i<0} \text{ s.t. }  \mathcal{F}(k’_{i-1} )= k’_i  \text{ and }      d( k_n ,  k’_n ) \underset{n\to - \infty}{\longrightarrow} 0\} \, .  $$ 
   $$W^u_\eta(\underline{k}, \mathcal{F} ) = \{ k_0’ \in  \mathcal{M}   :  \exists (k’_i)_{i<0} \text{ s.t. }  \mathcal{F}(k’_{i-1} )= k’_i  \text{ and }   \eta>   d( k_n ,  k’_n ) \underset{n\to - \infty}{\longrightarrow} 0\} \, .  $$
   
  \noindent    These sets are properly embedded $C^r$-manifolds. For simplicity, we do not denote $\mathcal{F}$ in the article when it is obvious: $W^s(k)$ for example. 

   
    \begin{theoremcontinub} \label{reff}
  Let $r \ge 1$ and  let $\mathcal{M}$ be a manifold. Suppose that $(\mathcal{F}_p)_p$ is a $C^r$-family of local diffeomorphisms $\mathcal{F}_p$ of $\mathcal{M}$  leaving  invariant the continuation of a  compact hyperbolic  set $\mathcal{K}_{p}$. Then  there  exists $\eta > 0$ s.t. 
 the families $(W_\eta^s( k_p ; \mathcal{F}_p))_{ p \in \mathcal{P} } $ and $(W_\eta^u(\underline{k}_p ; \mathcal{F}_p))_{ p \in \mathcal{P} }$  
  of $C^r$-submanifolds are of class $C^r$ and depend $C^0$  on respectively $k_0 \in \mathcal{K}_{0}$ and $\underline{k}_0 \in \overleftrightarrow{\mathcal{K}}_{\mathcal{F}_0}$. 
     \end{theoremcontinub}

\begin{Remark}  \label{epicerie}
 An immediate adaptation of the proof of Theorem C.5 in Appendix C of \cite{bbb137} actually shows that the families $(W_\eta^s( k_p ; \mathcal{F}_p))_{ p \in \mathcal{P} } $ and $(W_\eta^u(\underline{k}_p ; \mathcal{F}_p)_{ p \in \mathcal{P} }$ depend H\"older for the $C^{r-1}$-topology on respectively $k_0 \in \mathcal{K}_{0}$ and $\underline{k}_0 \in \overleftrightarrow{\mathcal{K}}_{\mathcal{F}_0}$.
\end{Remark}

\noindent We will need the following parametric inclination lemma:

    \begin{Lemma} \label{refff}
  Let $r \ge 1$ and $U \Subset \mathbb{R}^m$. Suppose that $(\mathcal{F}_p)_p$ is a $C^r$-family of local diffeomorphisms $\mathcal{F}_p$ of $U$  leaving  invariant a compact hyperbolic  set $\mathcal{K}_{p}$. Let $\underline{k} =(k^i_0)_i \in \overleftrightarrow{\mathcal{K}}_{\mathcal{F}_0}$ and $(\Gamma_p)_p$ be a $C^{r}$-family of manifolds of the same dimension as $W^s_\eta(\underline{k}_p ; \mathcal{F}_p)$. Suppose that  $\Gamma_p$ does not intersect the stable set of  $\mathcal{K}_p$ and    
  $(\Gamma_p)_p$ intersects transversally $(W^u_\eta(\underline{k}_p ; \mathcal{F}_p))_p$ at a $C^r$-family of points $(z_p)_p$. Then for any  $\epsilon>0$ and  $n$ large there is a submanifold $\Gamma^n_p$  $C^r$-close to $W^s_\eta(k^{-n}_p ; \mathcal{F}_p)$, whose image by $\mathcal{F}^n_p$ is  in a  $\epsilon$-neighborhood of $z_p$ in $\Gamma_p$ and s.t. $(\Gamma^n_p)_p$ is $C^{r}$-close to $(W^s_\eta(k_p^{-n}  ; \mathcal{F}_p))_p$.
     \end{Lemma}
   
   \begin{proof}
   The proof is similar  to the one of the parametric inclination Lemma C.6  of \cite{bbb137} but for inverse iterations this time: one extends $\mathcal{F}_p$ on a neighborhood $U’$ of $U$ it in such a way that   $\Gamma_p$ is included in the stable manifold of some saddle point. Then we apply Theorem F. 
   \end{proof}

\noindent \emph{Sébastien Biebler} \newline
\noindent  sebastien.biebler@imj-prg.fr \newline 
\noindent Sorbonne Université, CNRS \newline 
\noindent IMJ-PRG, 4 Place Jussieu, 75005 Paris

\end{document}